\documentclass[a4paper, 12pt]{amsart}
\usepackage{amsmath,amsthm,amssymb}
\usepackage[T1]{fontenc}
\usepackage{url}
\usepackage{amsmath,amsthm,amssymb}
\usepackage{verbatim}
\usepackage{mathrsfs}
\usepackage{graphics}
\usepackage{graphicx}
\usepackage{subfigure}
\usepackage{hyperref}
\usepackage{mathtools}
\usepackage{color}
\usepackage{tikz}
\usepackage{tikz-cd}
\allowdisplaybreaks[1]

\newtheorem{theorem}{Theorem}[section]
\newtheorem{proposition}[theorem]{Proposition}
\newtheorem{lemma}[theorem]{Lemma}
\newtheorem{corollary}[theorem]{Corollary}

\theoremstyle{remark}
\newtheorem{remark}[theorem]{Remark}
\newtheorem{definition}{Definition}

\numberwithin{equation}{section}

\newcommand{\R}{{\mathbb{R}}}

\newcommand{\vep}{\varepsilon}

\newcommand{\C}{{\mathbb{C}}}
\newcommand{\Z}{{\mathbb{Z}}}
\newcommand{\N}{{\mathbb{N}}}

\newcommand{\var}{\operatorname{Var}}

\newcommand{\bv}{\operatorname{BV}}

\newcommand{\SL}{\operatorname{SL}(\psi_\R)}

\newcommand{\Arg}{\operatorname{Arg}}

\newcommand{\norm}[1]{\left\|#1\right\|} 
\newcommand{\M}{\mathcal{M}}

\newcommand{\pa}{\operatorname{P_a}}

\newcommand{\pag}{\operatorname{P_aG}}

\newcommand{\cpang}{C^{n+\pag}}

\newcommand{\Int}{\operatorname{Int}}

\makeatletter
\newcommand{\customlabel}[2]{\protected@write\@auxout{}{\string\newlabel{#1}{{#2}{\thepage}{#2}{#1}{}}}\hypertarget{#1}{#2}}
\makeatother

\begin{document}

\title[Solving cohomological equation - part II. global obstructions]
{Solving the cohomological equation for locally hamiltonian flows, part II - global obstructions}
\author[K.\ Fr\k{a}czek]{Krzysztof Fr\k{a}czek}
\address{Faculty of Mathematics and Computer Science, Nicolaus
Copernicus University, ul. Chopina 12/18, 87-100 Toru\'n, Poland}
\email{fraczek@mat.umk.pl}

\author[M. Kim]{Minsung Kim}


\address{Department of Mathematics,  Kungliga Tekniska H\"ogskolan, Lindstedtsv\"agen 25, SE-100 44 Stockholm, Sweden}
\email{minsung@kth.se}

\address{}
\email{}
\date{\today}

\thanks{}
\maketitle

\begin{abstract}
Continuing the research initiated in  \cite{Fr-Ki2}, we provide a solution to the problem of existence and regularity of solutions to the cohomological equation $X u=f$ for locally Hamiltonian flows (determined by a vector field $X$) on a compact surface $M$ of genus $g\geq 1$, when the flow is restricted to its minimal component.
We go beyond the case studied so far by Forni in \cite{Fo1,Fo3}, when the flow is minimal over the entire surface and the function $f$ satisfies some Sobolev regularity conditions. We deal with the flow restricted to any of its minimal components and any smooth function $f$ whenever the restricted flow  satisfies the Full Filtration Diophantine Condition (FFDC),  {a full measure condition}.

The main goal of this article is to quantify the optimal regularity of solutions. For this purpose, we construct a family of invariant distributions $\mathfrak{F}_{\bar t}$, $\bar{t}\in\mathscr{TF}^*$, which play a role analogous to Forni's invariant distributions constructed in \cite{Fo1,Fo3} by using the language of translation surfaces.
Unlike the locally defined distributions $\mathfrak{d}^k_{\sigma,j}$,  $(\sigma,k,j)\in\mathscr{TD}$, and $\mathfrak{C}^k_{\sigma,l}$, $(\sigma,k,l)\in\mathscr{TC}$, introduced in \cite{Fr-Ki2}, the distributions $\mathfrak{F}_{\bar t}$ are global in nature, as emphasized in the title of this article.
All three families of distributions are used to determine the optimal regularity  of solutions to the cohomological equation, see Theorem~\ref{thm:main1}~and~\ref{thm:main2}.
As a by-product, we also {obtain} an intrinsically interesting spectral result (Theorem~\ref{thm:spect}) for the Kontsevich-Zorich cocycle acting on functional spaces that arise naturally at the transition to a first-return map.
%
%
\end{abstract}


\section{Introduction}
Let $M$ be a smooth compact connected orientable surface of genus $g\geq 1$.
Our primary focus is on smooth flows $\psi_\R = (\psi_t)_{t\in\R}$ on $M$  preserving a smooth area form $\omega$, {i.e.,} such that for any (orientable)
local coordinates $(x,y)$, we have $\omega=V(x,y)dx\wedge dy$ with $V$ positive and smooth. We denote by $X:M\to TM$ the associated vector field.
Then for (orientable) local coordinates $(x,y)$ such that $\omega=V(x,y)dx\wedge dy$,  the flow $\psi_\R$ is (locally) a solution to the Hamiltonian equation
\[
\frac{dx}{dt} = \frac{\frac{\partial H}{\partial y}(x,y)}{V(x,y)},\quad
\frac{dy}{dt} = -\frac{\frac{\partial H}{\partial x}(x,y)}{V(x,y)}
\]
for a smooth real-valued locally defined function $H$. The flows $\psi_\R$ are usually called \emph{locally Hamiltonian flows} or \emph{multivalued Hamiltonian flows}.
For general introduction to locally Hamiltonian flows on surfaces, we refer readers to  \cite{Fr-Ul2, Fr-Ki, Rav}.

The main goal of this article is to fully understand the problem of {the} existence of a solution $u:M\to\R$ and its regularity for the cohomological equation $Xu=f$, if  $f:M\to\R$ is an arbitrary smooth observable
(recall that $Xu(x)=\frac{d}{dt}u(\psi_tx)|_{t=0}$).
{
Cohomological equations {constitute} 
an important area of study in dynamical systems since they are related to smooth conjugacy problems via Kolmogorov-Arnold-Moser techniques. In the context of flows on higher genus surfaces, it is worth {mentioning} the pioneering work of Marmi-Moussa-Yoccoz \cite{Ma-Mo-Yo2} on the conjugacy problem for small perturbations of translation flows.}

We always assume that all fixed points of $\psi_\R$ are isolated.  In that case, the set of fixed points $\mathrm{Fix}(\psi_\R)$ is finite. As $\psi_\R$ is area-preserving, every fixed point is either a center or a saddle.
In what follows, we deal only with \emph{perfect} (\emph{harmonic}) saddles defined as follows:
a fixed point  $\sigma\in \mathrm{Fix}(\psi_\R)$ is a perfect saddle of multiplicity $m_\sigma\geq 2$ if there exists a chart $(x,y)$ in a neighborhood $U_\sigma$ of $\sigma$ such that
$\omega=V(x,y)dx\wedge dy$ and $H(x,y)=\Im (x+iy)^{m_\sigma}$. We call $(x,y)$ \emph{a singular chart}.
We denote by $\mathrm{Sd}(\psi_\R)$ the set of perfect saddles of  $\psi_\R$.

We call a \emph{saddle connection} an orbit of $\psi_\R$ running from a saddle to a saddle. A \emph{saddle loop} is a saddle connection
joining the same saddle.
We deal only with flows such that all their saddle connections are loops.
If every fixed point is isolated, then $M$ splits into a finite number of components ($\psi_\R$-invariant surfaces with boundary) so that every component is either a \emph{minimal component} (every orbit, except fixed points and saddle loops, is dense in the component) or is a \emph{periodic component} (filled by periodic orbits, fixed points and saddle loops). {The existence of such a decomposition (with transitive components) was actually proved by Maier in \cite{Ma:tra} (see also \cite[Theorem~3.1.7]{NZ:flo} and \cite[Theorem~3]{Co-Gu-Li} for a more modern presentation). Minimality follows from standard arguments in \cite{Keane}.}

{
Recall that a linear bounded functional $f\mapsto \mathfrak{D}(f)$ is an \emph{invariant distribution} if $\mathfrak{D}(Xu)=0$ for any $u\in C^{\infty}(M)$.}
The problem of existence and regularity of solutions for the cohomological equation $Xu=f$ was essentially solved in two seminal articles \cite{Fo1,Fo3} by Forni, under the assumption that the flow $\psi_\R$ is minimal over the whole surface $M$ (has no saddle connection) and the function $f$ belongs to a certain weighted Sobolev space $H_W^s(M)$, $s\geq 1$.
Let us mention that being an element of {the} weighted Sobolev space enforces significant constraints on the behavior of the function $f$ around saddles, even for smooth functions, as described in \cite{Fo3}.
In \cite{Fo1,Fo3}, for a.e.\ locally Hamiltonian flows, Forni proved the existence of fundamental invariant distributions on $H_W^s(M)$, which are responsible for the regularity of the solution $u$ to $Xu=f$, for $f\in H_W^s(M)$.
If all Forni's {invariant} distributions evaluated at $f\in H_W^s(M)$ are zero, then  a solution $u\in H_\omega^{s'}(M)$ exists for some $0<s'<s$.

The problem of solving cohomological equations for other classes of smooth dynamical systems of  parabolic nature, as well as the regularity of solutions using invariant distributions, has also been studied in  \cite{Av-Fa-Ko,Av-Ko,Fl-Fo03,Fl-Fo07,Fo4,Gu-Li,Kat,LMS,Ta,Wa,Zu}.

\subsection{Invariant distributions and the main results when saddle loops exist}

The main goal of this article is to go beyond the case of a minimal flow on the whole surface $M$,
{as well as  beyond the case where the function $f$ belongs to weighted Sobolev spaces.}
We deal with locally Hamiltonian flows restricted to any minimal component $M'\subset M$, where $f:M\to\R$ is any smooth function.
The main novelty of the proposed approach is its applicability to studying the regularity of the solution $u$ when the flow has saddle loops, a setting that has not been systematically investigated before.
The study of locally Hamiltonian flows in such a context gives rise to new invariant distributions, which, unlike Forni's {invariant} distributions, are local in nature. Two families of such local functionals $ \mathfrak{C}^k_{\sigma,l}$ and $ \mathfrak{d}^k_{\sigma,j}$ were introduced by the authors in \cite{Fr-Ki2}.
As shown in \cite{Fr-Ki2}, both families play an essential role in understanding the regularity of {solutions} to {the} cohomological equation if $f$ is any smooth function.

Throughout the article, we use the notations $x\vee y=\max\{x,y\}$ and $x\wedge y=\min\{x,y\}$ for any pair of real numbers $x,y$.
We denote by $\mathscr{TD}$ the set of triples $(\sigma,k,j)\in (\mathrm{Sd}(\psi_\R)\cap M')\times \Z_{\geq 0}\times\Z_{\geq 0}$ such that
 $0\leq j\leq k\wedge(m_\sigma-2)$ and $j\neq k-(m_\sigma-1)\operatorname{mod} m_\sigma$.
{Following \cite{Fr-Ki2} and using singular local coordinates around $\sigma$,}
for every $(\sigma,k,j)\in \mathscr{TD}$, we define the functional $\mathfrak{d}^k_{\sigma,j}:C^{k}(M)\to\C$ as follows:
\begin{equation}\label{def:gothd}
\mathfrak{d}^k_{\sigma,j}(f)=\sum_{0\leq n\leq \frac{k-j}{m_\sigma}}\frac{\binom{k}{j+nm_\sigma}\binom{\frac{(m_\sigma-1)-j}{m_\sigma}-1}{n}}{\binom{\frac{(k-j)-(m_\sigma-1)}{m_\sigma}}{n}}
\frac{\partial^k(f\cdot V)}{\partial z^{j+nm_\sigma} \partial \overline{z}^{k-j-nm_\sigma}}(0,0).
\end{equation}
The real number
$\widehat{\mathfrak{o}}(\mathfrak{d}^k_{\sigma,j})=\widehat{\mathfrak{o}}(\sigma,k)=k-(m_\sigma-2)$
we call the \emph{hat-order} of  $\mathfrak{d}^k_{\sigma,j}$.

For any $\sigma\in\mathrm{Sd}(\psi_\R)\cap M'$, its  neighbourhood $U_\sigma$ splits into $2m_\sigma$ angular sectors bounded by separatrices.  In singular coordinates $z=(x,y)$ they are of the form
\[U_{\sigma,l}:=\{z\in U_\sigma: \Arg z\in(\tfrac{\pi l}{m_\sigma},\tfrac{\pi (l+1)}{m_\sigma})\}\text{ for }0\leq l<2m_\sigma.\]
We denote by $\mathscr{TC}$ the set of triples $(\sigma,k,l)\in (\mathrm{Sd}(\psi_\R)\cap M')\times \Z_{\geq 0}\times\Z_{\geq 0}$ such that
  $0\leq l< 2m_\sigma$ and $U_{\sigma,l}\subset M'$. {Following \cite{Fr-Ki2}, and again using singular local coordinates around $\sigma$,} for every $(\sigma,k,l)\in \mathscr{TC}$, we define
the functional  $ \mathfrak{C}^k_{\sigma,l}:C^k(M)\to\C$  as follows:
\[\mathfrak{C}^k_{\sigma,l}(f):=\sum_{\substack{0\leq i\leq k\\i\neq m_\sigma-1\operatorname{mod} m_\sigma\\i\neq k-(m_\sigma-1)\operatorname{mod} m_\sigma}}
\theta_\sigma^{l(2i-k)}\binom{k}{i}\mathfrak{B}(\tfrac{(m_\sigma-1)-i}{m_\sigma},\tfrac{(m_\sigma-1)-k+i}{m_\sigma})\frac{\partial^{k}(f\cdot V)}{\partial z^i\partial\overline{z}^{k-i}}(0,0),\]
where $\theta_\sigma$ is the principal $2m_\sigma$-th root of unity. The (beta-like) function $\mathfrak{B}(x,y)$ is defined for any pair $x,y$ of real numbers such that  $x,y\notin \Z$  as follows:
\[\mathfrak{B}(x,y)=\frac{\pi e^{i\frac{\pi}{2}(y-x)}}{2^{x+y-2}}\frac{\Gamma(x+y-1)}{\Gamma(x)\Gamma(y)},\]
where we adopt the convention $\Gamma(0)=1$ and $\Gamma(-n)=1/(-1)^n n!$. For the real number
$\mathfrak{o}(\mathfrak{C}^k_{\sigma,l})=\mathfrak{o}(\sigma,k)=\frac{k-(m_\sigma-2)}{m_\sigma}$,
we call it the \emph{order} of  $\mathfrak{C}^k_{\sigma,l}$.

{\begin{remark}
As shown in \cite{Fr-Ki2} (see the proof of Corollary~5.5 therein), the value of the functional $\mathfrak{C}^k_{\sigma,l}(f)$ can be read as the weight of a singularity of some function of one variable depending only on $f$ and the flow around $\sigma$. Hence, it does not depend on the choice of singular coordinates. Since every functional $\mathfrak{d}^k_{\sigma,j}$ is a linear combination of the functionals $\mathfrak{C}^k_{\sigma,l}$ (see \cite[Section~4.3]{Fr-Ki2}), its values also do not depend on the choice of singular coordinates.
\end{remark}}

In this paper, for a.e.\ locally Hamiltonian flow (satisfying the Full Filtration Diophantine Condition (\ref{FDC}), defined in Section~\ref{sec;FDC-RTC}), we define a third family of distributions $\mathfrak{F}_{\bar{t}}$. {These distributions have global nature} and are smooth analogues of Forni's invariant distributions introduced in \cite{Fo1,Fo3}. We should emphasize that the definition of $\mathfrak{F}_{\bar{t}}$, unlike Forni's approach, does not use translational surface tools. Such techniques cannot be used due to the existence of saddle loops. Our approach is based on a modification of  correction operators invented by Marmi-Moussa-Yoccoz in \cite{Ma-Mo-Yo} (see also \cite{Ma-Mo-Yo2} and \cite{Ma-Yo}) in its simplest version and later extended in \cite{Fr-Ul} and \cite{Fr-Ki}.

Let $g\geq 1$ be the genus of $M'$ and let $\gamma$ be the number of saddles in $M'$.
We denote by $\mathscr{TF}^*$ the set of triples of the form $(k,+,i)$, $(k,0,s)$ or $(k,-,j)$ for $k\geq 0$, $1\leq i,j\leq g$ and $1\leq s<\gamma$. Let $\mathscr{TF}$ be the subset of triples in $\mathscr{TF}^*$ obtained by removing all triples of the form $(k,-,1)$ for $k\geq 0$. We denote by $0<\lambda_g<\ldots<\lambda_2<\lambda_1$ the positive Lyapunov exponents {of the Kontsevich-Zorich cocycle} associated with a flow satisfying the \ref{FDC} (see again Section~\ref{sec;FDC-RTC}).
In Section~\ref{sec:Forgendis}, for every triple $\bar{t}\in \mathscr{TF}^*$, we define {the} corresponding functional $\mathfrak{F}_{\bar{t}}$.
For any $\bar{t}\in \mathscr{TF}^*$, the real number
\[\mathfrak{o}(\mathfrak{F}_{\bar{t}})=\mathfrak{o}(\bar{t})=\left\{
\begin{array}{cl}
k-\frac{\lambda_i}{\lambda_1}&\text{ if }\bar{t}=(k,+,i),\\
k&\text{ if }\bar{t}=(k,0,s),\\
k+\frac{\lambda_j}{\lambda_1}&\text{ if }\bar{t}=(k,-,j),
\end{array}
\right.\]
we call the \emph{order} of  $\mathfrak{F}_{\bar{t}}$.

\medskip

Let $m$ be the maximal multiplicity of saddles in $\mathrm{Sd}(\psi_\R)\cap M'$.
Following \cite{Fr-Ki2}, for every {$r\geq -\frac{m-2}{m}$}, let
\[k_r=\left\{
\begin{array}{cl}
\lceil mr+(m-1)\rceil &\text{if }m=2\text{ and } r\leq \frac{1}{2},\\
\lceil mr+(m-2)\rceil &\text{otherwise.}
\end{array}
\right.\]
Recall that
\begin{align*}\label{eq:expkr}
&\max\{k\geq 0:\exists_{\sigma\in \mathrm{Sd}(\psi_\R)\cap M'}\mathfrak{o}(\sigma,k)< r\}+1=\lceil mr+(m-2)\rceil\leq k_r,\\
&\max\{k\geq 0:\exists_{\sigma\in \mathrm{Sd}(\psi_\R)\cap M'}\widehat{\mathfrak{o}}(\sigma,k)< r\}+1=\lceil r+(m-2)\rceil\leq k_r.
\end{align*}
Then, for every flow $\psi_\R$ restricted to its minimal component $M'$ and satisfying {the} \ref{FDC}, and for every $\bar{t}\in \mathscr{TF}^*$, the corresponding functional $\mathfrak{F}_{\bar{t}}$ is defined on $C^{k_{\mathfrak{o}(\bar{t})}+1}(M)$. The following two main results show how the three families of invariant distributions influence the regularity of the solution to the cohomological equation $Xu=f$, with smooth $u$ defined on the end compactification $M'_e$ of $M'\setminus \mathrm{Sd}(\psi_\R)$ as considered in \cite{Fr-Ki2}.

\begin{theorem}\label{thm:main1}
Let $\psi_\R$ be a locally Hamiltonian flow such that its restriction to a minimal component $M'$ satisfies the \ref{FDC}. Let $r\in \R_{>0}\setminus(\{\mathfrak{o}(\sigma,k):k\geq 0,\sigma\in \mathrm{Sd}(\psi_\R)\cap M'\}
\cup\{\mathfrak{o}(\bar{t}):\bar{t}\in \mathscr{TF}\})$. Suppose that $f\in C^{k_r}(M)$ and
\begin{itemize}
\item[(i)] $\mathfrak{d}^k_{\sigma,j}(f)=0$ for all $(\sigma,k,j)\in\mathscr{TD}$ with $\widehat{\mathfrak{o}}(\mathfrak{d}^k_{\sigma,j})<r$;
\item[(ii)] $\mathfrak{C}^k_{\sigma,l}(f)=0$ for all $(\sigma,k,l)\in\mathscr{TC}$ with ${\mathfrak{o}}(\mathfrak{C}^k_{\sigma,l})<r$;
\item[(iii)] $\mathfrak{F}_{\bar{t}}(f)=0$ for all $\bar{t}\in\mathscr{TF}$ with ${\mathfrak{o}}(\mathfrak{F}_{\bar{t}})<r$.
\end{itemize}
Then there exists $u\in C^r(M'_e)$ such that $Xu=f$ on $M'_e$. Moreover, there exists $C_r>0$ such that $\|u\|_{C^r(M'_e)}\leq C_r\|f\|_{C^{k_r}(M)}$.
\end{theorem}

\begin{theorem}[optimal regularity]\label{thm:main2}
Let $\psi_\R$ be a locally Hamiltonian flow such that its restriction to a minimal component $M'$ satisfies the \ref{FDC}. Fix  $r>0$ and suppose that $f\in C^{k_r}(M)$ and there exists $u\in C^r(M'_e)$ such that $Xu=f$ on $M'_e$. Then
\begin{itemize}
\item[(i)] $\mathfrak{d}^k_{\sigma,j}(f)=0$ for all $(\sigma,k,j)\in\mathscr{TD}$ with $\widehat{\mathfrak{o}}(\mathfrak{d}^k_{\sigma,j})<r$;
\item[(ii)] $\mathfrak{C}^k_{\sigma,l}(f)=0$ for all $(\sigma,k,l)\in\mathscr{TC}$ with ${\mathfrak{o}}(\mathfrak{C}^k_{\sigma,l})<r$;
\item[(iii)] $\mathfrak{F}_{\bar{t}}(f)=0$ for all $\bar{t}\in\mathscr{TF}$ with ${\mathfrak{o}}(\mathfrak{F}_{\bar{t}})<r$.
\end{itemize}
\end{theorem}
{
The invariant distributions  $\mathfrak{d}^k_{\sigma,j}$ and $\mathfrak{C}^k_{\sigma,l}$ were already introduced in  \cite{Fr-Ki2}, where their significance for the regularity problem of solving the cohomological equation was highlighted.
More precisely, in \cite{Fr-Ki2} we proved two results (Theorems~1.2~and~1.3 therein) that are fundamental to the present work. They appear identical to the present Theorems~\ref{thm:main1}~and~\ref{thm:main2}, except that their last condition (iii)  is replaced by another condition stating that the cohomological equation $v\circ T-v=\varphi_f$ has a $C^r$-solution. Here, $T$ is an IET which is a Poincar\'e map for the flow, and $\varphi_f$ is the first-recurrence integral of the function $f$ (for a formal description of these objects, see Section~\ref{sec:coheqIET}).
Therefore, to prove Theorem~\ref{thm:main1} using \cite{Fr-Ki2}, we only need to prove that conditions (i), (ii), and (iii) together imply the existence of a $C^r$-solution to the cohomological equation $v\circ T-v=\varphi_f$. This constitutes the primary challenge of the current paper. Conversely, to prove Theorem~\ref{thm:main2} using \cite{Fr-Ki2}, it is enough to show that the existence of a $C^r$-solution to the cohomological equation for $\varphi_f$ implies the condition (iii), which is the vanishing of the {new invariant} distributions $\mathfrak{F}_{\bar{t}}$.

To define the distributions $\mathfrak{F}_{\bar{t}}$, we first define a family of distributions $\mathfrak{f}_{\bar{t}}$ acting on some spaces $C^{n+\pag}$ of  $n$-differentiable functions on exchanged intervals, so that the $n$-th derivative has singularities at the end of the intervals.
In fact, the most significant contribution of the present paper is the construction of functionals $\mathfrak{f}_{\bar{t}}$. These functionals demonstrate how to correct the function $\varphi_f$ by piecewise polynomial functions, ensuring that a sequence of renormalizations (arising from the Kontsevich-Zorich cocycle) of the (polynomially) corrected function decays exponentially at an appropriate rate. As will be shown later, this condition is needed to prove the existence of a solution $v$ of class $C^r$.
Using {other} results from \cite{Fr-Ki2}, we know that if $f$ satisfies conditions (i) and (ii), then $\varphi_f$ belongs to a space $C^{n+\pag}$.
However,  if we consider an arbitrary smooth function $f$ (i.e., without regard for conditions (i) and (ii)), we must correct the function $\varphi_f$  to remove singularities for its derivatives of order less than $n$.
To define an appropriate (singular) correction that eliminates the influence of these singularities, we must combine distributions $\mathfrak{f}_{\bar{t}}$ (cf.\ \eqref{def:corrsing})  with distributions $\mathfrak{C}^k_{\sigma,l}$ (cf.\ \eqref{def:rsk}) as weights.
Finally, $\mathfrak{F}_{\bar{t}}(f)$ is defined by applying $\mathfrak{f}_{\bar{t}}$ to the function $\varphi_f$ after applying the singular correction (see Definition~\ref{def:Ft}). Consequently, if conditions (i), (ii), and (iii) are all satisfied, the correction is trivial (due to (ii)), and $\mathfrak{f}_{\bar{t}}(\varphi_f)=0$ (due to (iii)). It follows that the sequence of renormalizations of $\varphi_f$ decays exponentially with an appropriate rate.

In summary, the  results presented in \cite{Fr-Ki2} enable us to reduce the original regularity problem of solving the cohomological equation $Xu=f$ for the flow to the regularity problem of solving the cohomological equation $v\circ T-v=\varphi_f$ for the IET. Furthermore, these results are also essential for defining the {invariant} distributions $\mathfrak{F}_{\bar{t}}$.

Theorem~\ref{thm:main1} provides a key tool for investigating the existence and regularity of solutions to the cohomology equation $Xu=f$ on the entire surface $M$, rather than only on its minimal components. As noted earlier, the surface $M$ decomposes into a finite number of minimal and periodic components. The solvability of the cohomological equation on periodic components boils down to considering additional obstructions (infinitely many) arising from periodic orbits and saddle connections (e.g., integrals of the function $f$ along these orbits). By gluing the solutions constructed on both minimal and periodic components, we obtain a framework for analyzing the existence and regularity of solutions to the cohomological equation on the whole surface.

%
}



\subsection{Cohomological equations over IETs and a spectral result}\label{sec:coheqIET}
Let us consider the restriction of a locally Hamiltonian flow $\psi_\mathbb{R}$ on $M$ to its minimal component $M'\subset M$ and let $I\subset M'$ be a  smooth transversal curve.
We always assume that each end of $I$ is the first meeting point of a separatrix (that is not a saddle connection) emanating {from} a saddle (incoming or outgoing) with the curve $I$.
By minimality, $I$ is a global transversal, and the first return map $T:I\to I$ is an interval exchange transformation (IET) in the so-called standard coordinates on $I$.
We denote by $I_\alpha$, $\alpha\in\mathcal{A}$ the intervals exchanged by $T$, and by $\tau:I\to\R_{>0}\cup\{+\infty\}$ the first return time map to the curve $I$, also called the {\emph{roof function}}.
The roof function $\tau: I \rightarrow \R_{>0}\cup\{+\infty\}$ is smooth on the interior of each exchanged interval and  has
\emph{singularities} at discontinuities of $T$.
For any continuous observable $f:M\to\C$, we deal with the corresponding map $\varphi_f:I\to\C\cup\{\infty\}$ given by
\[\varphi_f(x)=\int_0^{\tau(x)}f(\psi_t x)dt.\]
If $u$ is a solution of the cohomological equation $Xu=f$, then
\begin{equation}\label{eq:coheqT}
v(Tx)-v(x)=\varphi_f(x) \text{ on }I,
\end{equation}
where $v$ is the restriction of $u$ to the curve $I$. Therefore, the existence and regularity of the solution to the cohomological equation \eqref{eq:coheqT}
 is clearly a  necessary condition for the existence and the same regularity of the solution to $Xu=f$. As shown in \cite[Theorem~1.2]{Fr-Ki2}, this is also a sufficient condition under additional assumptions related to {the vanishing of} certain distributions  $\mathfrak{C}^k_{\sigma,l}$ and $\mathfrak{d}^k_{\sigma,j}$  on  $f$.
 Moreover, the regularity of solution $u$ depends on the regularity of the solution $v$ and the vanishing of all the mentioned distributions up to some level of their order or hat-order. For this reason, in the present paper, we primarily focus on the cohomological equation $v\circ T-v=\varphi_f$.
The regularity of  $\varphi_f$ was completely understood in \cite{Fr-Ki2}. It was shown there that $\varphi_f\in C^{n+\pag}(\sqcup_{\alpha \in \mathcal{A}} I_\alpha)$, {i.e.,}\ is piecewise $C^{n+1}$ and its $n$-th derivative has polynomial (of degree at most $0<a<1$) or logarithmic (if $a=0$) singularities at  discontinuities of the IET $T$. The degree of smoothness $n$ depends on the maximal order of vanishing for the distributions  $\mathfrak{C}^k_{\sigma,l}$, see Theorem~1.1 in \cite{Fr-Ki2}.

For any $k\in \N\cup\{\infty\}$, we denote by $\Phi^k(\sqcup_{\alpha \in \mathcal{A}} I_\alpha)$ the space of functions of the form $\varphi_f$ for $f\in C^k(M)$. The main tool used to solve the {cohomological} equation \eqref{eq:coheqT} is a spectral analysis of the functional version (on $\Phi^k(\sqcup_{\alpha \in \mathcal{A}} I_\alpha)$) of the Kontsevich-Zorich cocycle $S(j)$ (see Section~\ref{SpecialBS} for the definition). A certain type of spectral analysis (for positive Lyapunov exponents) of the cocycle $S(j)$ was already used in \cite{Fr-Ul} and \cite{Fr-Ki} to fully understand {the deviation of ergodic integrals for a.a.\ locally Hamiltonian flows on smooth observables.}
Our techniques are motivated by the correction operators invented by Marmi-Moussa-Yoccoz in \cite{Ma-Mo-Yo} (see also \cite{Ma-Mo-Yo2} and \cite{Ma-Yo}) in their simplest version (without singularities) and later extended in \cite{Fr-Ul} and \cite{Fr-Ki}.

To represent formally the main spectral result, let us consider an equivalence relation $\sim$ on the set of triples $\mathscr{TC}$, introduced in \cite{Fr-Ki2}.
Two triples $(\sigma,k,l),(\sigma,k,l')\in\mathscr{TC}$ are equivalent with respect to the equivalence relation $\sim$ if the angular sectors $U_{\sigma,l}$ and $U_{\sigma,l'}$ are connected through a chain of saddle loops emanating from the saddle $\sigma$. For every equivalence class $[(\sigma,k,l)]\in \mathscr{TC}/\sim$, let
\[\mathfrak{C}_{[(\sigma,k,l)]}(f):=\sum_{(\sigma,k,l')\sim (\sigma,k,l)}\mathfrak{C}^k_{\sigma,l}(f).\]
For any $k\geq 0$, let $\Gamma_k(\sqcup_{\alpha \in \mathcal{A}} I_\alpha)$ be the space of functions which are polynomials of degree at most $k$ on each interval $I_\alpha$, $\alpha\in \mathcal{A}$. This space plays an important role in solving cohomological equations in \cite{Fo-Ma-Ma}. We will define two families of functions $\{h_{\bar{t}}:\bar{t}\in\mathscr{TF}^*\}$ and $\{\xi_{[(\sigma,k,l)]}:[(\sigma,k,l)]\in \mathscr{TC}/\sim\}$ {such that
\begin{gather*}
 h_{\bar{t}}\in \Gamma_k(\sqcup_{\alpha \in \mathcal{A}} I_\alpha) \text{ if } \bar{t}=(k,\,\cdot\, ,\,\cdot\,);\\
\xi_{[(\sigma,k,l)]} \in C^{n+\pag}(\sqcup_{\alpha \in \mathcal{A}} I_\alpha) \text{ with }n=\lceil\mathfrak{o}(\sigma,k)\rceil,\ a=\mathfrak{o}(\sigma,k)-n,
\end{gather*}}
which are the keys to understanding the spectral properties of the Kontsevich-Zorich cocycle $S(j)$.
The main spectral result is as follows.
\begin{theorem}[spectral theorem]\label{thm:spect}
Let $\psi_\R$ be a locally Hamiltonian flow such that its restriction to a minimal component $M'$ satisfies the \ref{FDC}. For every $r>-\frac{m-2}{m}$ and $f\in C^{k_r}(M)$, we have
\[\varphi_f=\sum_{\substack{\bar{t}\in\mathscr{TF}^*\\ \mathfrak{o}(\bar{t})<r}}\mathfrak{F}_{\bar t}(f)h_{\bar{t}}+\sum_{\substack{[(\sigma,k,l)]\in\mathscr{TC}/\sim\\ \mathfrak{o}(\sigma,k)<r}}\mathfrak{C}_{[(\sigma,k,l)]}(f)\xi_{[(\sigma,k,l)]}+\mathfrak{r}_r(f)\]
so that {
\begin{gather}
 \label{eq:expt}
 \lim_{j\to\infty}\frac{1}{j}\log\|S(j)h_{\bar{t}}\|_{\sup}=\lim_{j\to\infty}\frac{1}{j}\log\|S(j)h_{\bar{t}}\|_{L^1}=-\lambda_1\mathfrak{o}(\bar{t});
\\
\label{eq:expsi1}
\lim_{j\to\infty}\frac{1}{j}\log\|S(j)\xi_{[(\sigma,k,l)]}\|_{L^1}=-\lambda_1\mathfrak{o}(\sigma,k);\\
\label{eq:expsi2}
\lim_{j\to\infty}\frac{1}{j}\log\|S(j)\xi_{[(\sigma,k,l)]}\|_{\sup}=-\lambda_1\mathfrak{o}(\sigma,k)\text{ if }\mathfrak{o}(\sigma,k)>0;\\
\label{eq:exrem}
\begin{split}
&\limsup_{j\to\infty}\frac{1}{j}\log\|S(j)\mathfrak{r}_r(f)\|_{\sup}\leq -\lambda_1 r \text{ if }r>0\text{ and }\\
&\limsup_{j\to\infty}\frac{1}{j}\log\|S(j)\mathfrak{r}_r(f)\|_{L^1}\leq -\lambda_1 r \text{ if }r\leq 0.
\end{split}
\end{gather}}
\end{theorem}
This theorem can be seen as a counterpart to spectral results from \cite{FGL} in {a} general (non-pseudo-Anosov) setting.
However, the most important advantage of Theorem~\ref{thm:spect} is that it (more precisely, its preceding version, Theorem~\ref{thm;spdecomp}) is used to solve (in Section~\ref{sec:highreg}) the regularity of solutions to the cohomological equation $v\circ T-v=\varphi_f$  (see Theorem~\ref{thm:cohsolmain1}).

{ The sequence of renormalizations  $(S(j)(\varphi_f))_{j\geq 0}$ (the Kontsevich-Zorich cocycle) arises by inducing the function (cocycle) $\varphi$ on a decreasing sequence of intervals $(I^{(j)})_{j\geq 0}$ that shrink to zero. As Zorich observed in \cite{Zo}, the sequence $(S(j)(\varphi_f))_{j\geq 0}$ is useful for controlling the growth of arbitrary Birkhoff sums of the function $\varphi_f$ by appropriately decomposing such sums into components of the form $S(j)(\varphi_f)$.  If the sequence $(S(j)(\varphi_f))_{j\geq 0}$ exhibits exponential decay, then all Birkhoff sums are uniformly bounded (for a.e.\ $T$), and classical Gottschalk-Hedlund arguments ensure a continuous solution to the equation $v\circ T-v=\varphi_f$. To improve the regularity of the solution, we modify techniques concerning orbit decompositions and space decompositions developed by Marmi-Yoccoz in \cite{Ma-Yo}, relating the regularity of the solution to the decay rate of the sequence $(S(j)(\varphi_f))_{j\geq 0}$. More precisely, if $S(j)(\varphi_f)$ decreases faster than $e^{-r\lambda_1 j}$, then the solution $v$ is of class $C^r$.

Our spectral theorem (Theorem~\ref{thm:spect}) states that if distribution vanishing conditions (ii) and (iii) are met, then $S(j)(\varphi_f)$ decreases faster than $e^{-(r-\epsilon)\lambda_1 j}$ for any small $\epsilon>0$. Therefore, $v\in C^{r-\epsilon}$ for any small $\epsilon>0$. This highlights the fundamental importance of the spectral theorem and the decay rate of the renormalization sequence.
}


\subsection{A new family of invariant distributions via extended correction operators}
In \cite{Fr-Ki2}, the authors defined two families of invariant distributions, $\mathfrak{C}^k_{\sigma,l}$ and $\mathfrak{d}^k_{\sigma,j}$, inspired by local analysis of higher order derivatives {of} $\varphi_f$ around {the} ends of intervals exchanged by $T$ {(see also recent developments concerning the deviation of ergodic integrals for locally Hamiltonian flows in the most general setting \cite[Theorem 2.1]{BFT}).
In the current paper, we introduce a new family $\mathfrak{f}_{\bar{t}}$, $\bar{t}\in\mathscr{TF}^*$, of invariant distributions over IETs and transport them to the level of the surface $M$ by composing with the operator $f\mapsto \varphi_f$.
The resulting distributions $\mathfrak{F}_{\bar{t}}$, $\bar{t}\in\mathscr{TF}^*$, generalize (emulate) the notion of Forni's invariant distributions (associated with Lyapunov exponents of {the Kontsevich}-Zorich cocycle). However, the method of construction is entirely different from the original {one}.

The invariant distributions $\mathfrak{f}_{\bar{t}}$ over IETs are defined on the space $C^{n+\pag}(\sqcup_{\alpha \in \mathcal{A}} I_\alpha)$ (if $\mathfrak{o}(\bar{t})<n-a$). In \cite{Fr-Ki}, the authors constructed invariant distributions for $n=0$, using correction through piecewise constant functions. We constructed so-called {\emph{correction operators}} $\mathfrak{h}_j$, $1\leq j\leq g$, but the construction was limited to the unstable subspace (corresponding to positive Lyapunov exponents) of the Kontsevich-Zorich cocycle. The original idea of correcting smooth functions was introduced by Marmi-Moussa-Yoccoz in \cite{Ma-Mo-Yo} and then developed in \cite{Fr-Ul} and \cite{Fr-Ki}.
}

In this paper, there are three types of new functionals that arise from other parts of the Oseledets splitting ($+/-/0$ denoting unstable/stable/central, {respectively}) associated with Lyapunov exponents (see Section \ref{sec;inv-distpag}).
Their construction is based on using new correction operators $\mathfrak{h}_{-j,i}$, $\mathfrak{h}^*_j$, $\mathfrak{h}_{0}$ and their higher-order derivatives.
The new correction operators allow us to correct  $\varphi_f$ by piecewise constant functions, not only related to unstable vectors as before, but also {to} central and stable vectors.
The construction of these three new types of correction operators is the most important technical novelty of the article, which allows {for} defining the counterparts of Forni's invariant {distributions} for flows with saddle loops. Together with the previously defined local invariant distributions $\mathfrak{C}^k_{\sigma,l}$ and $\mathfrak{d}^k_{\sigma,j}$, they provide complete and optimal knowledge of the regularity of solutions on the H\"older scale. This optimality of regularity seems to be the most important overall novelty of the article.

\subsection{Structure of the paper}
In \S~\ref{sec;IET}, we recall some basic notions related to
IETs, Rauzy-Veech induction, and accelerations of the Kontsevich-Zorich cocycle.
In \S~\ref{sec;DC}, we review {the} Oseledets filtration of the accelerated KZ-cocycles and formulate the corresponding  Full Filtration Diophantine Condition (\ref{FDC}). In the next section, we set up a new infinite series which are necessary for constructing extended correction operators.
In \S~\ref{sec;correction}, extended correction operators $\mathfrak{h}_{-j,i}$, $\mathfrak{h}^*_j$, and $\mathfrak{h}_{0}$ are constructed, and their basic properties are proved.
In \S~\ref{sec;spec}, we compute {the} Lyapunov exponents of the renormalization cocycle $S(j)$ for piecewise polynomial function $h_{i,l}$, $c_{s,l}$, and $h_{-j,l}$. These three classes of functions are then used to construct the functionals $\mathfrak{f}_{\bar t}$. The culmination of this section is the proof of the spectral result (Theorem~\ref{thm;spdecomp}), which is the main component of the proof of Theorem~\ref{thm:spect}.
Cohomological equations  for {IETs} and the  regularity of their solutions are studied in \S~\ref{sec;coh-iet}.
Finally, in \S~\ref{sec;last}, we conclude the regularity of solutions to cohomological equations for locally Hamiltonian flows. The main results are obtained from the main theorems in \cite{Fr-Ki2} {(the existence of solution under the vanishing of local invariant distributions $\mathfrak{C}^k_{\sigma,l}$ and $\mathfrak{d}^k_{\sigma,j}$) and the results of \S~\ref{sec;coh-iet} (involving the appearance of {global} invariant distributions $\mathfrak{F}_{\bar{t}}$) for the cohomological equations  for IETs.}

\section{Interval exchange transformations (IET)}\label{sec;IET}
 Let $\mathcal{A}$
be a $d$-element alphabet and let $\pi=(\pi_0,\pi_1)$ be a pair of
bijections $\pi_\vep:\mathcal{A}\to\{1,\ldots,d\}$ for $\vep=0,1$.
For every $\lambda=(\lambda_\alpha)_{\alpha\in\mathcal{A}}\in
\R_{>0}^{\mathcal{A}}$, let
$|\lambda|:=\sum_{\alpha\in\mathcal{A}}\lambda_\alpha$, $I:=\left[0,|\lambda|\right)$, and for every $\alpha\in\mathcal{A}$,
\begin{gather*}
 I_{\alpha}:=[l_\alpha,r_\alpha),\text{ where
}l_\alpha=\sum_{\pi_0(\beta)<\pi_0(\alpha)}\lambda_\beta,\;\;\;r_\alpha
=\sum_{\pi_0(\beta)\leq\pi_0(\alpha)}\lambda_\beta.
\end{gather*}
We denote by
$\mathcal{S}^0_{\mathcal{A}}$ the subset of \emph{irreducible} pairs,
{i.e.,}\ $\pi_1\circ\pi_0^{-1}\{1,\ldots,k\}\neq\{1,\ldots,k\}$ for $1\leq
k<d$.  We will always assume that $\pi \in
\mathcal{S}^0_{\mathcal{A}}$.
An \emph{interval exchange transformation} (IET) $T = T_{(\pi,\lambda)}:I\to I$ is a piecewise translation determined by the data $(\pi, \lambda )$, so that $T_{(\pi,\lambda)}$ translates the interval $I_{\alpha}$ for each $\alpha \in \mathcal{A}$ so that $T(x)=x+w_\alpha$ for $x\in I_\alpha$, where
$w=\Omega_\pi\lambda$ and $\Omega_\pi$ is  the matrix
$[\Omega_{\alpha\,\beta}]_{\alpha,\beta\in\mathcal{A}}$ given by
\[\Omega_{\alpha\,\beta}=
\left\{\begin{array}{cl} +1 & \text{ if
}\pi_1(\alpha)>\pi_1(\beta)\text{ and
}\pi_0(\alpha)<\pi_0(\beta),\\
-1 & \text{ if }\pi_1(\alpha)<\pi_1(\beta)\text{ and
}\pi_0(\alpha)>\pi_0(\beta),\\
0& \text{ in all other cases.}
\end{array}\right.\]
 An IET  $T_{(\pi,\lambda)}$
satisfies the {\em Keane condition} (see \cite{Keane}) if
$T_{(\pi,\lambda)}^m l_{\alpha}\neq l_{\beta}$ for all $m\geq 1$
and for all $\alpha,\beta\in\mathcal{A}$ with $\pi_0(\beta)\neq 1$.

\subsection{Rauzy-Veech induction}\label{sec;RVI}
Rauzy-Veech induction (see \cite{Ra}) and its accelerations are standard renormalization procedures for IETs. For general background, we refer readers to the lecture notes by Yoccoz \cite{Yo,Yoc} or Viana \cite{ViB}.

Let $T=T_{(\pi,\lambda)}$ be an interval exchange transformation satisfying Keane's condition. Let
$\widetilde I:= \big[0,\max (l_{\pi_0^{-1}(d)}, l_{\pi_1^{-1}(d)}) \big)$ and
denote by $\mathcal{R}(T) = \widetilde T : \widetilde I \to \widetilde I$ the first return map of $T$ to the interval $\widetilde I$.
Let
\begin{eqnarray*}
\epsilon =\epsilon(\pi,\lambda) = \left\{\begin{array}{cll}
0& \text{ if
}&\lambda_{\pi_0^{-1}(d)} > \lambda_{\pi_1^{-1}(d)},\\
1& \text{ if
}&\lambda_{\pi_0^{-1}(d)} < \lambda_{\pi_1^{-1}(d)} \end{array} \right.
\end{eqnarray*}
and
\begin{equation*}
A(T) = A(\pi,\lambda) = Id+E_{\pi_{\epsilon}^{-1}(d)\,\pi_{1-\epsilon}^{-1}(d)} \in SL_{\mathcal A}(\Z),
\end{equation*}
where $Id$ is the identity matrix and $(E_{ij})_{kl} = \delta_{ik}\delta_{jl}$, using the Kronecker delta notation.
Then,  by Rauzy (see \cite{Ra}), $\widetilde T$ is also an IET on $d$-intervals satisfying Keane's condition, and
$\widetilde T = T_{(\widetilde \pi,\widetilde \lambda)}$ for some $\widetilde{\pi}=(\widetilde{\pi}_0,\widetilde{\pi}_1)\in\mathcal{S}^0_{\mathcal{A}}$ and
$\tilde \lambda = A^{-1}(\pi,\lambda)\lambda$.
Moreover, the renormalized version of the matrix $\Omega_{\widetilde{\pi}}$ is of the form
\begin{equation*}
\Omega_{\widetilde{\pi}}=A^t(\pi,\lambda)\cdot\Omega_{\pi}\cdot A(\pi,\lambda).
\end{equation*}
Thus, taking $H(\pi) = \Omega_\pi(\R^{\mathcal A})$, we have $H(\tilde \pi) = A^t(\pi,\lambda) H(\pi)$.

\subsection{Kontsevich-Zorich cocycle and its accelerations}\label{sec;KZco}
Let $T=T_{(\pi,\lambda)}$ be an IET satisfying Keane's condition.
For every $n \geq 1$, we define
\[A^{(n)}(T): = A(T) \cdot A(\mathcal{R}(T)) \cdot \dotsc \cdot A(\mathcal{R}^{n-1}(T))\in  SL_{\mathcal A}(\Z).\]
This defines a multiplicative cocycle $A$ over the transformation $\mathcal{R}$ and it is called the \emph{Kontsevich-Zorich cocycle}.
Let $(n_k)_{k\geq 0}$ be an increasing sequence of integers with $n_0=0$, called an \emph{accelerating sequence}.
For every $k \geq 0$, let $T^{(k)}:= \mathcal{R}^{n_k}(T) : I^{(k)} \to I^{(k)}$.
Then
$T^{(k)}:I^{(k)}\to I^{(k)}$ is the first return map of
$T:I\to I$ to the interval $I^{(k)}\subset I$.
The sequence of IETs $(T^{(k)})_{k\geq 0}$ gives an
\emph{acceleration} of the Rauzy-Veech renormalization procedure associated with the {accelerating sequence} $(n_k)_{k\geq0}$.

Let $(\pi^{(k)},\lambda^{(k)})$ be the pair defining $T^{(k)}$ and let $I_\alpha^{(k)}$, $\alpha \in \mathcal A$, be the intervals exchanged by $T^{(k)}$. Then $ \lambda^{(k)} = (\lambda_\alpha^{(k)})_{\alpha \in \mathcal A}$, where $\lambda^{(k)}_\alpha=|I_\alpha^{(k)}|$ for $\alpha \in \mathcal A$.

For every $k\geq 0$, define $Z(k+1):=A^{(n_{k+1}-n_{k})}(\mathcal{R}^{n_k}(T))^t$. We then have
\[\lambda^{(k)} = Z(k+1)^t\lambda^{(k+1)}, \quad k \geq 0.\]
Following notations from \cite{Ma-Mo-Yo}, for each $0 \leq k < l$, let
\[
Q(k,l) = Z(l)\cdot Z(l-1) \cdot \dotsc \cdot Z(k+2) \cdot Z(k+1) = A^{(n_{l}-n_{k})}(\mathcal{R}^{n_k}(T))^t.
\]
Then, $Q(k,l) \in SL_{\mathcal A}(\Z)$, and $ \lambda^{(k)} = Q(k,l)^t\lambda^{(l)}$. We write $Q(k) = Q(0,k)$.


\subsection{Rokhlin towers related to accelerations}\label{sec;Rokhlin}
Note that $Q_{\alpha\beta}(k)$ is the time spent by
any point of $I^{(k)}_{\alpha}$ in $I_{\beta}$ until it
returns to $I^{(k)}$. Then
$Q_{\alpha}(k)=\sum_{\beta\in\mathcal{A}}Q_{\alpha\beta}(k)$
is the first return time of points from $I^{(k)}_{\alpha}$ to
$I^{(k)}$.
The IET $T:I\to I$ thus splits into a set of $d$ {\emph{Rokhlin towers}} of the form
\[\big\{ T^i (I^{(k)}_{\alpha}), \ 0\leq i < Q_{\alpha}(k)\big\},\quad \alpha\in \mathcal{A},\]
so that the $Q_{\alpha}(k)$ floors of the $\alpha$-th tower are pairwise disjoint intervals.


\section{Diophantine conditions for IETs { and locally Hamiltonian flows}}\label{sec;DC}
In this section, we introduce a new Diophantine condition for IETs, which is a full-measure condition on the set of IETs. The Diophantine condition is a modified version of the previously introduced one in \cite{Fr-Ki} (see also \cite{Fr-Ul2}), {referred to as the} Filtration Diophantine condition (FDC). {This condition} is improved by extending the Oseledets filtration to stable and central subspaces. Based on this condition, we show that certain series involving matrices of the accelerated cocycle grows in a controlled manner. 

\subsection{Oseledets filtration}\label{sec;OF}
Fix $\pi\in \mathcal{S}^0_{\mathcal{A}}$.
Suppose that there exist $\lambda_1>\ldots > \lambda_g>\lambda_{g+1}=0$ such that for a.e.\ IET $(\pi,\lambda)$ there exists a filtration of linear subspaces (Oseledets filtration)
\begin{gather}\label{eq:flagsp}
\begin{split}
\{0\}=E_{0}(\pi,\lambda)\subset E_{-1}(\pi,\lambda)\subset\ldots\subset E_{-g}(\pi,\lambda)\subset E_{cs}(\pi,\lambda)\\
=E_{g+1}(\pi,\lambda)\subset E_{g}(\pi,\lambda)\subset\ldots\subset E_{1}(\pi,\lambda)=\Gamma:=\R^{\mathcal{A}}
\end{split}
\end{gather}
such that for every  $1\leq i\leq g$, we have
\begin{align}\label{eq:Oscond}
\begin{split}
&\lim_{n\to+\infty}\frac{\log\|Q(n)h\|}{n}=\lambda_{-i}:=-\lambda_i\text{ for all } h\in E_{-i}(\pi,\lambda)\setminus E_{-i+1}(\pi,\lambda),\\
&\lim_{n\to+\infty}\frac{\log\|Q(n)h\|}{n}=0\text{ for all } h\in E_{cs}(\pi,\lambda)\setminus E_{-g}(\pi,\lambda),\\
&\lim_{n\to+\infty}\frac{\log\|Q(n)h\|}{n}=\lambda_i\text{ for all } h\in E_{i}(\pi,\lambda)\setminus E_{i+1}(\pi,\lambda),\\
&\dim E_{-i}(\pi,\lambda)-\dim E_{-i+1}(\pi,\lambda)=\dim E_{i}(\pi,\lambda)-\dim E_{i+1}(\pi,\lambda)=1.
\end{split}
\end{align}
Suppose that there exists a filtration of linear subspaces which is complementary to the Oseledets filtration \eqref{eq:flagsp}:
\begin{align}\label{eq:osel}
\begin{split}
&\{0\}=U_1\subset U_2\subset\ldots \subset U_{g}\subset U_{g+1}\subset U_{-g}\subset\ldots \subset U_{-1}\subset U_0=\Gamma\\
&\text{such that }U_{g+1}\subset H(\pi) \text{ and }E_j(\pi,\lambda)\oplus U_j = \Gamma\text{ for } -g\leq j\leq g+1.
\end{split}
\end{align}
As $E_{-g}\oplus U_{g+1}=H(\pi)$, $U_{j+1}=U_{j}\oplus (U_{j+1}\cap E_{j})$ and $\dim (U_{j+1}\cap E_{j})=1$, for every $ j\in\pm\{1,\ldots, g\}$, there exists $h_j\in U_{j+1}\cap E_{j}$ such that
\[h_j\in H(\pi),\quad U_{j+1}=U_{j}\oplus\R h_j\text{ and }\lim_{n\to+\infty}\frac{\log\|Q(n)h_j\|}{n}=\lambda_j.\]
Let $c_1,\ldots, c_{\gamma-1}$ be a basis of $U_{-g}\cap E_{g+1}$.
Then, for every $2\leq j\leq g+1$, the linear subspace $U_j\subset \Gamma$ is  generated by $h_1,\ldots, h_{j-1}$
and for every $0\leq j\leq g$, the linear subspace $U_{-j}\subset \Gamma$ is  generated by $h_1,\ldots, h_{g}$, $c_1,\ldots, c_{\gamma-1}$ and
$h_{-g},\ldots, h_{-j-1}$. Moreover,
\begin{equation}\label{neq:posexph}
\text{if $0\neq h \in U_j$ then }\lim_{n\to+\infty}\frac{\log\|Q(n)h\|}{n}\geq \lambda_{j-1},
\end{equation}
where $\lambda_{-g-1}=-\lambda_{g+1}=0$.

{
For any $k\geq 0$, let $\Gamma^{(k)} \subset L^1(I^{(k)})$ be the subspace of functions that are constant on $I_\alpha^{(k)} \subset I^{(k)}$, for $\alpha \in \mathcal A$. Then, we identify any function $\sum_{\alpha \in \mathcal A} h_\alpha \chi_{I_\alpha^{(k)}} \in \Gamma^{(k)}$ with the vector $h  = (h_\alpha)_{\alpha \in \mathcal A} \in \R^{\mathcal{A}}$. We also write $\Gamma = \Gamma^{(0)}$.

For any $k\geq 0$ and $-g\leq j\leq g+1$, let $E^{(k)}_j:=Q(k)E_j$ and $U^{(k)}_j:=Q(k)U_j$. Then,
$
E^{(k)}_j\oplus U^{(k)}_j=\Gamma^{(k)}.
$
For  any choice of a complementary filtration \eqref{eq:osel}, all $0\leq k\leq l$, and every $-g\leq j\leq g+1$, we consider
the restrictions of the operator $Q(k,l):\Gamma^{(k)}\to \Gamma^{(l)}$ given by
\begin{gather*}
Q|_{E_j}(k,l):E_{j}^{(k)}\to E_{j}^{(l)},\quad
Q|_{U_j}(k,l):U_{j}^{(k)}\to U_{j}^{(l)}.
\end{gather*}
}

\subsection{Rokhlin Tower Condition and Filtration Diophantine Condition}\label{sec;FDC-RTC}
The following Rokhlin Towers Condition (RTC) was introduced in \cite{Fr-Ul2}.

\begin{definition}[RTC]\label{def:RTC}
An IET $T_{(\pi,\lambda)}$ together with  an acceleration  satisfies  RTC if there exists a constant $0<\delta<1$ such that
\begin{align}
 \tag{RT} \label{def:FDC-g}
\begin{split}
& \text{for any $k\geq 1,$ there exists a number $0<p_{k}\leq \min_{\alpha\in\mathcal{A}}Q_\alpha(k)$} \text{ such that}   \\
&  \{T^iI^{(k)}:0\leq i<p_k\} \text{ is a Rokhlin tower of intervals with measure $\geq \delta|I|$.}
\end{split}
\end{align}
\end{definition}
For any  sequence $(r_n)_{n\geq 0}$ of real numbers, and for all $0\leq k\leq l$, we will use the notation $r(k,l):=\sum_{k\leq j<l}r_j$.

\begin{definition}[\customlabel{FDC}{FFDC}]\label{def;FDC}
An IET $T : I \to I$ satisfying Keane's condition and being Oseledets-generic ({i.e.,}\ there is a filtration of linear subspaces \eqref{eq:flagsp} satisfying \eqref{eq:Oscond}), satisfies
the \emph{Full Filtration Diophantine Condition (FFDC)} if for every $\tau > 0$, there exist constants  $C,\kappa\geq 1$, an accelerating sequence $(n_k)_{k\geq0}$, a  sequence of natural numbers $(r_n)_{n\geq 0}$ with $r_0 = 0$, and
a complementary filtration $(U_j)_{{-g} \leq j \leq g+1}$ (satisfying \eqref{eq:osel}) such that \eqref{def:FDC-g} holds and:
\begin{align}
\lim_{n\to+\infty}\frac{r(0,n)}{n}&\in(1,1+\tau)\label{def;sdc0},\\
\norm{Q|_{E^{(k)}_j}(k,l)} \leq Ce^{(\lambda_j+\tau)r(k,l)} &\text{ for all }0\leq k<l\text{ and }1\leq j\leq g+1 \label{def;sdc1},\\
\norm{Q|_{E^{(k)}_{-j}}(k,l)} \leq Ce^{(-\lambda_j+\tau)r(k,l)} &\text{ for all }0\leq k<l\text{ and }1\leq j\leq g \label{def;sdc11},\\
\norm{Q|_{U^{(k)}_j}(k,l)^{-1}} \leq Ce^{(-\lambda_{j-1}+\tau)r(k,l)} &\text{ for all }0\leq k<l\text{ and }2\leq j\leq g+1 \label{def;sdc12},\\
\norm{Q|_{U^{(k)}_{-j}}(k,l)^{-1}} \leq Ce^{(\lambda_{j+1}+\tau)r(k,l)} &\text{ for all }0\leq k<l\text{ and }0\leq j\leq g \label{def;sdc13},\\
\norm{Z(k+1)} \leq Ce^{\tau k} &\text{ for all } k\geq 0\label{def;sdc2},\\
C^{-1}e^{\lambda_1k}\leq \norm{Q(k)} \leq Ce^{\lambda_1(1+\tau)k} &\text{ for all }k\geq 0 \label{def;sdc3},\\
\max_{\alpha\in\mathcal{A}}\frac{|I^{(k)}|}{|I^{(k)}_\alpha|}\leq \kappa &\text{ for all }k\geq 0 \label{def;sdc4},\\
\big|\sin \angle \big(E_j^{(k)}, U_j^{(k)}\big) \big| \geq c  \norm{Q(k)}^{-\tau} &\text{ for all }k\geq 0 \text{ and }-g\leq j\leq g+1. \label{def;sdc5}
\end{align}
\end{definition}

{
\begin{remark}
Note that {the Diophantine condition (FFDC) we have defined here} clearly implies the Filtration Diophantine Condition (FDC) introduced in \cite{Fr-Ki}, which in turn implies the Uniform Diophantine Condition (UDC) introduced in \cite{Fr-Ul2}. As shown in \cite[Remark 3.7]{Fr-Ul2}, {the UDC implies two Diophantine conditions considered classical: the Roth-type condition defined in \cite{Ma-Mo-Yo} and the restricted Roth-type condition considered in \cite{Ma-Mo-Yo2}. For more detailed discussions, we refer readers to the recent survey {by} Ulcigrai about Diophantine conditions \cite{Ul:ICM}.}
\end{remark}
}

\begin{definition}
A locally Hamiltonian flow $\psi_\R$ on $M$, with isolated fixed points and {restricted to  a  minimal component} $M'\subset M$, satisfies
the \emph{Full Filtration Diophantine Condition (FFDC)} if there exists a transversal $I\subset M'$ such that the corresponding IET $T:I\to I$ satisfies the \ref{FDC}.
\end{definition}

\begin{theorem}\label{thm;FDCRTC}
 Almost every IET satisfies the \ref{FDC}.
\end{theorem}

\begin{proof}
Most of the proof of Theorem follows similarly from the proof of Theorem 3.2 in \cite{Fr-Ki}. In addition to the proof of the FDC in {\cite[Appendix A]{Fr-Ki}}, it suffices to slightly modify the construction of the full measure set $\Xi$ {defined in \cite[\S A.3]{Fr-Ki}} to show that every $(\pi,\lambda)\in \Xi$ satisfies \eqref{def;sdc11}, \eqref{def;sdc13}, and \eqref{def;sdc5} not only on the non-negative part of the filtration (as shown in \cite{Fr-Ki}) but also on its negative part, i.e., on   $E^{(k)}_{-j}$ and $U^{(k)}_{-j}$ for $1\leq j\leq g$. Since this modification is straightforward, we omit the details.
\end{proof}

\begin{remark}
{
We denote by $\mathcal{F}$ the set of smooth locally Hamiltonian flows on $M$ with isolated fixed points, which are  centers or perfect saddles.
For every $\psi_\R\in\mathcal{F}$ preserving an area form $\omega$, let $X:M\to TM$ be the corresponding vector field $X$. Consider the $1$-form $\imath_X\omega=\omega(X, \,\cdot \,)$. Since $\omega$ is $X$-invariant, $\imath_X\omega$ is a smooth real closed $1$-form.

For any vector $\overline{m}=(m_1,m_2,\ldots,m_s)$ of natural numbers $\geq 2$ and any $0\leq c\leq\sum_{i=1}^s(m_i-1)$, we denote by  $\mathcal{F}_{\overline{m},c}$
the set of smooth locally Hamiltonian flows with $c$ centers and $s$  {perfect} saddles of multiplicity $m_1,m_2,\ldots,m_s$.
A measure-theoretical notion of typicality on $\mathcal{F}$ (on each  $\mathcal{F}_{\overline{m},c}$ separately) is defined by {the
 so-called}  \emph{Katok fundamental class} (introduced  in \cite{Ka0}).
Let $\gamma_1, \dots, \gamma_n$ be a base of $H_1(M, \mathrm{Fix}(\psi_\mathbb{R}), \mathbb{R})$, where $n=2g+s+c-1$.
Define the period map:
\[\Theta(\psi_\mathbb{R})=\Big(\int_{\gamma_1}\imath_X\omega,\ldots,\int_{\gamma_n}\imath_X\omega\Big)\in\mathbb{R}^n,\]
which is well-defined in a neighbourhood of $\psi_\mathbb{R} \in \mathcal{F}_{\overline{m},c}$.
The $\Theta$-pullback of the Lebesgue measure class ({i.e.,}\ the class of sets with zero measure) gives the desired measure class on $\mathcal{F}_{\overline{m},c}$.

When we use the expression ``a.e.\ locally Hamiltonian flow'',  we mean full measure in each $\mathcal{F}_{\overline{m},c}$  with respect to the corresponding measure class.
Note that a.e.\ flow $\psi_\mathbb{R}\in\mathcal{F}_{\overline{m},0}$ is minimal.
If  $c\geq 1$, then every $\psi_\mathbb{R}\in\mathcal{F}_{\overline{m},c}$  has a non-trivial splitting into minimal and periodic components, and we consider only the restriction of $\psi_\mathbb{R}$ to any of its minimal components $M'\subset M$.}
\medskip

In view of Theorem~\ref{thm;FDCRTC}, almost every (with respect to the Katok fundamental class) locally Hamiltonian flow $\psi_\R\in \mathcal{F}$ restricted to its minimal component $M'\subset M$ satisfies the \ref{FDC}.
\end{remark}

\begin{remark}
As $1=|I|\leq|I^{(n)}|\|Q(n)\|\leq |I|/\kappa=\kappa^{-1}$, by \eqref{def;sdc3}, we have
\begin{equation}\label{eq:invIk}
|I^{(n)}|^{-1}\leq \|Q(n)\|\leq C e^{(\lambda_1+\tau)n}\text{ and  }
|I^{(n)}|\leq {\kappa^{-1}} \|Q(n)\|^{-1}\leq \kappa^{-1}C e^{-\lambda_1n}.
\end{equation}
As $\lim_{n\to+\infty}n/r(0,n)>1/(1+\tau)>1-\tau$, there exists $c>0$ such that
\begin{equation}\label{eq:nrn}
(1-\tau)r(0,n)-c\leq n\leq r(0,n) \text{ for all }n\geq 0.
\end{equation}
\end{remark}

\begin{remark}\label{rmk:h-1}
Let us consider the map $\bar{\xi}:I\to\R$ given by $\bar{\xi}(x)=x$ and the corresponding coboundary $\bar{\xi}\circ T-\bar{\xi}$. Then {$\bar{\xi}\circ T-\bar{\xi}$ is a piecewise constant map which can be treated as a vector in $\Gamma$. Moreover,} for every $k\geq 0$, we have
\[
\big(Q(k)(\bar{\xi}\circ T-\bar{\xi})\big)_\alpha=\bar{\xi}(T^{Q_\alpha(k)}x)-\bar{\xi}(x)=T^{Q_\alpha(k)}x-x
\text{ for any }x\in I^{(k)}_\alpha.\]
Therefore $\|Q(k)(\bar{\xi}\circ T-\bar{\xi})\|\leq |I^{(k)}|\leq \kappa^{-1}C e^{-\lambda_1k}$. By \eqref{eq:Oscond}, $\bar{\xi}\circ T-\bar{\xi}\in E_{-1}(\pi,\lambda)$. Since the space $E_{-1}(\pi,\lambda)$ is one-dimensional, we have $h_{-1}=c(\bar{\xi}\circ T-\bar{\xi})$ for some $c\neq 0$.
\end{remark}

{
Recall that for any  $k\geq 0$ and $-g\leq j\leq g+1$, we have $E^{(k)}_j\oplus U^{(k)}_j=\Gamma^{(k)}$. We denote by $P_{E^{(k)}_j}:\Gamma^{(k)}\to E^{(k)}_j$ and $P_{U^{(k)}_j}:\Gamma^{(k)}\to U^{(k)}_j$
the corresponding projections, {i.e.,}\  $P_{E^{(k)}_j}+P_{U^{(k)}_j}=Id_{\Gamma^{(k)}}$.}
In view of \eqref{def;sdc5}, using the arguments of the proof of Lemma~3.5 in \cite{Fr-Ki} {(based on the fact that the  norm of projections depends on the angle between the subspaces),}
for any $\tau>0$, there exists $C>0$ such that for all $k\geq 0$ and $-g\leq j\leq g+1$,
\begin{align}\label{eqn;projbound}
\begin{split}
\|P_{E^{(k)}_j}\|\leq C \norm{Q(k)}^\tau\quad\text{and}\quad &\|P_{U^{(k)}_j}\|\leq C \norm{Q(k)}^\tau.
\end{split}
\end{align}
Moreover, by definition, for any pair $0\leq k<l$ and any $-g\leq j\leq g+1$, we have
\[Q(k,l)\circ P_{E^{(k)}_j}= P_{E^{(l)}_j}\circ Q(k,l)\quad\text{and}\quad Q(k,l)\circ P_{U^{(k)}_j}= P_{U^{(l)}_j}\circ Q(k,l).\]

\subsection{Diophantine series}\label{sec;DCS}
For every $a\geq 0$ and $s\geq 1$, let $\langle s\rangle^a= s^a$ if $a>0$, and $\langle s\rangle^a= 1+\log s$ if $a=0$.

\begin{definition}\label{def;dio-series}
For every IET $T : I \rightarrow I$ satisfying Keane's condition, any $0\leq a<1$, any $2\leq i\leq g+1$, any $\tau>0$, and any accelerating sequence, we define sequences
$(K^{a,i,\tau}_{k}(T))_{k\geq 0}, (C^{a,i,\tau}_{k}(T))_{k\geq 0}$ so that
\begin{align*}
&K^{a,i,\tau}_{k}(T):=
\sum_{l\geq k}\|Q|_{U^{(k)}_i}(k,l+1)^{-1}\|\|Z(l+1)\|\langle\|Q(l)\|\rangle^{a}\|Q(l+1)\|^\tau,
\\
&C^{a,i,\tau}_{k}(T):= \sum_{0 \leq l < k}\|Q|_{E^{(l+1)}_i}(l+1,k)\|\|Z(l+1)\|\langle \|Q(l)\|\rangle^{a} \|Q(l+1)\|^\tau.
\end{align*}
\end{definition}

\begin{proposition}\cite[Proposition 3.6]{Fr-Ki}\label{prop;FDCbound}
Let $T : I \rightarrow I$ be an IET satisfying the \ref{FDC} and let $0\leq a<1$. Suppose that $2\leq i\leq g+1$ is chosen such that $a\lambda_1<\lambda_{i-1}$.
Then, for every $0 < \tau < \frac{\lambda_{i-1}-\lambda_{1}a}{3(1+\lambda_1)}$, the sequences $(K^{a,i,\tau}_k)_{k\geq 0}, (C^{a,i,\tau}_k)_{k\geq 0}$ are well defined and
\begin{align}\label{prop;sdc-1}
\begin{split}
K^{a,i,\tau}_{k}(T) & \leq C_\tau e^{(\lambda_{1}a+5\tau(1+\lambda_1))r(0,k)},\\
C^{a,i,\tau}_{k}(T) & \leq C_\tau e^{(\max\{\lambda_{i},\lambda_1 a\}+3\tau(1+\lambda_1))r(0,k)}.
\end{split}
\end{align}
\end{proposition}

\begin{definition}\label{def;dio-series2}
For every IET $T : I \rightarrow I$ satisfying Keane's condition, any $0\leq j\leq g+1$, any non-negative sequence $\bar{s}=(s_k)_{k\geq 0}$,  any $\tau>0$, and any accelerating sequence, we define sequences
$(V^{j,\tau}_{k}(T,\bar{s}))_{k\geq 0}, (W^{j,\tau}_{k}(T,\bar{s}))_{k\geq 0}$ so that
\begin{align*}
&V^{j,\tau}_{k}(T,\bar{s}):=\sum_{l \geq k} \|Q|_{U^{(k)}_{-j}}(k,l+1)^{-1}\|\|Q(l+1)\|^\tau\|Z(l+1)\|s_{l},
\\
&W^{j,\tau}_{k}(T,\bar{s}):=\sum_{0\leq l<k} \|Q|_{E^{(l+1)}_{-j}}(l+1,k)\|\|Q(l+1)\|^\tau\|Z(l+1)\|s_l.
\end{align*}
\end{definition}

\begin{proposition}\label{prop;FDCbound4}
Let $T : I \rightarrow I$ be an IET satisfying the \ref{FDC}.  Fix $0\leq j\leq g$, $\lambda_{j+1}<\rho$, and
$0 < \tau < \frac{\rho-\lambda_{j+1}}{\lambda_1+3}$. Then there exists $C_{\tau}>0$ such that
for any non-negative sequence $\bar{s}=(s_k)_{k\geq 0}$ with $s_k\leq D e^{-\rho r(0,k+1)}$ and for all $k\geq 0$, we have
\begin{align}\label{prop;sdc-5}
V^{j,\tau}_{k}(T,\bar{s}) &\leq C_\tau D e^{(-\rho + (\lambda_1+2)\tau)r(0,k)},\\
\label{prop;sdc-6}
W^{j,\tau}_{k}(T,\bar{s}) &\leq C_\tau D e^{(\max\{-\rho,-\lambda_{j} \} + (\lambda_1+3)\tau)r(0,k)}.
\end{align}
\end{proposition}

\begin{proof} By Definition~\ref{def;FDC},
\begin{align*}
V^{j,\tau}_{k}&\leq \sum_{l\geq k}C^3De^{(\lambda_{j+1}+\tau)r(k,l+1)}e^{(\lambda_1+\tau)\tau (l+1)}e^{\tau (l+1)}e^{-\rho r(0,l+1)}\\
&\leq \sum_{l\geq k}C^3De^{(\lambda_{j+1}+\tau)r(k,l+1)}e^{(-\rho+\tau(\lambda_1+2)) r(0,l+1)}\\
&= De^{(-\rho+\tau(\lambda_1+2)) r(0,k)}\sum_{l\geq k}C^3e^{(-\rho+\lambda_{j+1}+\tau(\lambda_1+3))r(k,l+1)}\\
&\leq De^{(-\rho+\tau(\lambda_1+2)) r(0,k)}\sum_{l\geq k}C^3e^{(-\rho+\lambda_{j+1}+\tau(\lambda_1+3))(l+1-k)}\\
&= De^{(-\rho+\tau(\lambda_1+2)) r(0,k)}\sum_{l\geq 1}C^3e^{(-\rho+\lambda_{j+1}+\tau(\lambda_1+3))l}.
\end{align*}
As $-\rho+\lambda_{j+1}+\tau(\lambda_1+3)<0$, the above series is convergent, so we get \eqref{prop;sdc-5}. Moreover, again by Definition~\ref{def;FDC},
\begin{align*}
W^{j,\tau}_{k}&\leq \sum_{0\leq l< k}C^3De^{(-\lambda_{j}+\tau)r(l+1,k)}e^{(\lambda_1+\tau)\tau (l+1)}e^{\tau (l+1)}e^{-\rho r(0,l+1)}\\
&\leq \sum_{1\leq l\leq k}C^3De^{(-\lambda_{j}+\tau)r(l,k)}e^{(-\rho+\tau(\lambda_1+2)) r(0,l)}\\
& \leq C^3Dke^{(\max\{-\lambda_{j},-\rho\}+\tau(\lambda_1+2)) r(0,k)}\\
&\leq C^3D C'e^{(\max\{-\lambda_{j},-\rho\}+\tau(\lambda_1+3)) r(0,k)},
\end{align*}
which gives \eqref{prop;sdc-6}.
\end{proof}

\section{Extended correction operators}\label{sec;correction}
In this section, we define  three types of new correction operators: $\mathfrak{h}^*_{j},\mathfrak{h}_{-j,i}$, and $\mathfrak{h}_{0}$, for $2 \leq i \leq g+1$ and $ 0 \leq j \leq g$.
These operators are motivated by the correction operator $\mathfrak{h}_{i}$ previously defined on $C^{0+\pa}$, the space of functions with polynomial singularities. In \cite[\S 6]{Fr-Ki}, the maps from $C^{0+\pa}$ were corrected by piecewise constant functions coming from the unstable subspace. Our new operators are constructed to correct piecewise smooth functions whose higher order derivatives have polynomial singularities (i.e., elements of $C^{n+\pag}$), using piecewise polynomial functions. 

\subsection{$C^{n+\pag}$ space}\label{sec;cnpag}
Fix $0\leq a<1$ and an integer $n \geq 0$. Following  \cite[\S 2]{Fr-Ki2},  $C^{n+\pa}(\sqcup_{\alpha \in \mathcal{A}}I_\alpha)$ is the
space of $C^{n+1}$-functions on $\bigcup_{\alpha\in\mathcal{A}}\Int I_\alpha$ such that
\begin{align*}
p_{a}(D^{n}\varphi):
= \max_{\alpha\in\mathcal{A}}\sup_{x\in(l_\alpha,r_\alpha)}\max\{|D^{n+1}\varphi(x)(x-l_\alpha)^{1+a}|,|D^{n+1}\varphi(x)(r_\alpha-x)^{1+a}|\}
\end{align*}
is finite, and
\begin{align*}
C_{\alpha,n}^{a,+}(\varphi) &=(-1)^{n}C_{\alpha}^+(D^n\varphi):=(-1)^{n+1}\lim_{x\searrow l_\alpha}D^{n+1}\varphi(x)(x-l_\alpha)^{1+a},\\
C_{\alpha,n}^{a,-}(\varphi) &=C_{\alpha}^-(D^n\varphi):=\lim_{x\nearrow r_\alpha}D^{n+1} \varphi(x)(r_\alpha-x)^{1+a}
\end{align*}
exist. The space $C^{n+\pa}(\sqcup_{\alpha \in \mathcal{A}}I_\alpha)$
 is a Banach space equipped with the norm
 \[\|\varphi\|_{C^{n+\pa}}:= \sum_{k=0}^n\|D^{k}\varphi\|_{L^1(I)}+p_a(D^{n}\varphi).
\]
{We denote by $C^{n+\pag}(\sqcup_{\alpha \in \mathcal{A}}I_\alpha) \subset C^{n+\pa}(\sqcup_{\alpha \in \mathcal{A}}I_\alpha)$ the subset (the union of four closed subspaces of co-dimension 2) of functions $\varphi \in C^{n+\pa}(\sqcup_{\alpha \in \mathcal{A}}I_\alpha)$ of \emph{geometric type}, i.e., such that}
\[C_{\pi_0^{-1}(d),n}^{a,-} (\varphi) \cdot C_{\pi_1^{-1}(d),n}^{a,-}(\varphi) =0\quad\text{and}\quad C_{\pi_0^{-1}(1),n}^{a,+}(\varphi) \cdot C_{\pi_1^{-1}(1),n}^{a,+}(\varphi)=0.\]


\subsection{Special Birkhoff sums}
\label{SpecialBS}
Assume that an IET $T : I \rightarrow I$ satisfies Keane's condition.
 For any $0\leq k<l$ and any measurable map $\varphi:I^{(k)}\to\R$ over the IET
$T^{(k)}:I^{(k)}\to I^{(k)}$, we denote by
$S(k,l)\varphi:I^{(l)}\to\R$ the renormalized map over
$T^{(l)}$ given by
\[S(k,l)\varphi(x)=\sum_{0\leq i<Q_{\beta}(k,l)}\varphi((T^{(k)})^ix),\text{ for }x\in I^{(l)}_\beta.\]
Sums of this form are called \emph{special Birkhoff sums}. {By} convention, we write $S(k)\varphi$ for $ S(0,k)\varphi$, and $S(k,k)\varphi=\varphi$.
 If $\varphi$ is integrable, then
\begin{equation}\label{nase}
\| S(k,l)\varphi\|_{L^1(I^{(l)})}\leq
\|\varphi\|_{L^1(I^{(k)})}\quad\text{and}\quad
\int_{I^{(l)}}S(k,l)\varphi(x)\,dx=
\int_{I^{(k)}}\varphi(x)\,dx.
\end{equation}
If additionally $\varphi\in \bv(\sqcup_{\alpha \in \mathcal{A}}I_\alpha^{(k)} )$ (is of bounded variation), then
\begin{equation}\label{neq:Skvar}
\var{S(k,l)\varphi}\leq\var{\varphi}\ \text{ and }\ \|S(k,l)\varphi\|_{\sup}\leq\|Q(k,l)\|\|\varphi\|_{\sup},
\end{equation}
where $\var{\varphi}$ is the sum of variations of $\varphi$ restricted to $\Int I_\alpha$, for $\alpha\in\mathcal{A}$.

We denote by $\Gamma^{(k)}$ the set of functions on $I^{(k)}$ which are constant on all $I^{(k)}_\alpha$, $\alpha\in\mathcal{A}$.
Clearly, $S(k,l)\Gamma^{(k)} = \Gamma^{(l)}$, and $S(k,l)$ is the linear automorphism of $\R^{\mathcal A}$ whose matrix in the canonical basis is $Q(k,l)$.

\begin{remark}
In view of \S 5  in \cite{Fr-Ki},  $S(k,l):C^{n+\pag}(\sqcup_{\alpha \in \mathcal{A}}I^{(k)}_\alpha )\to C^{n+\pag}(\sqcup_{\alpha \in \mathcal{A}}I^{(l)}_\alpha )$. Moreover, for every IET $T$ satisfying the \ref{FDC}, there exists $C\geq 1$ such that
for all $0 \leq k \leq l$ and for every function $\varphi \in C^{0+\pag}(\sqcup_{\alpha \in \mathcal{A}}I^{(k)}_\alpha )$, 
\begin{align}\label{eqn;renormpaos2}
\begin{split}
&p_a(S(k,l)\varphi) \leq Cp_a(\varphi)\text{ if }0<a<1,\\
&p_a(S(k,l)\varphi) \leq C(1+\log\|Q(k,l)\|)p_a(\varphi)\text{ if }a=0.
\end{split}
\end{align}
\end{remark}

\subsection{Correction operator on $C^{0+\pag}$}
For any integrable map $f:I\to\R$ and any subinterval $J\subset I$, let $m(f,J)$ stand for the mean value of $f$ on $J$,
that is,
\begin{equation*}\label{def;m(f,J)}
m(f,J)=\frac{1}{|J|}\int_Jf(x)\,dx.
\end{equation*}
For the IET $T^{(k)}$ let  $\mathcal{M}^{(k)}:L^1(I^{(k)})\to \Gamma^{(k)}$ be the corresponding mean-value \emph{projection operator}, given by
\begin{equation*}\label{eqn;projec}
\mathcal{M}^{(k)}(f)=\sum_{\alpha\in\mathcal{A}}m(f,I^{(k)}_\alpha)\chi_{I^{(k)}_\alpha}.
\end{equation*}
This operator projects any map onto a piecewise constant function whose values are equal to the mean value of $f$ on the exchanged intervals $I^{(k)}_\alpha$, $\alpha\in\mathcal{A}$. 

\begin{theorem}[Theorem 6.1 in \cite{Fr-Ki}]\label{thm;correction}
Assume that $T$ satisfies the \ref{FDC}. For any $0\leq a<1$, take $2\leq j\leq g+1$ so that $ \lambda_1 a<\lambda_{j-1}$. There exists a bounded linear operator
$\mathfrak{h}_j: C^{0+\pag}(\sqcup_{\alpha \in \mathcal{A}} I_\alpha) \rightarrow U_j $, {which satisfies the following:
for any $\tau>0$, there exists $C=C_\tau\geq 1$ such that for every $\varphi \in C^{0+\pag}(\sqcup_{\alpha \in \mathcal{A}} I_\alpha) $ with $\mathfrak{h}_j(\varphi) = 0$,}
\begin{equation}\label{eqn;maincorrection}
\|\mathcal{M}^{(k)}(S(k)\varphi)\| \leq C\left(\big(K^{a,j,\tau}_k+C^{a,j,\tau}_k\big) p_a(\varphi)+ \|{Q|_{E_j}}(k)\| \frac{\norm{\varphi}_{L^1(I^{(0)})}}{|I^{(0)}|}\right).
\end{equation}
\end{theorem}
The operator $\mathfrak{h}_j: C^{0+\pag}(\sqcup_{\alpha \in \mathcal{A}} I_\alpha) \rightarrow U_j \subset H(\pi)$ called the \emph{correction operator}, is given by
\begin{align}
\label{def:hi}
\mathfrak{h}_j(\varphi) &=  \lim_{k\to\infty} Q(0,k)^{-1}\circ P_{U^{(k)}_j}\circ \mathcal{M}^{(k)}\circ S(k)(\varphi)\\
&=  \sum_{l \geq 0} Q(0,l)^{-1}\circ P_{U^{(l)}_j}\circ\big(\mathcal{M}^{(l)}\circ S(l)-Z(l)\circ\mathcal{M}^{(l-1)}\circ S(l-1)\big)(\varphi),\nonumber
\end{align}
where $\mathcal{M}^{(-1)}=0$.
\begin{remark}\label{rem:hjhj'}
Note that for all $2\leq j'\leq j\leq g+1$, we have $P_{U^{(0)}_{j'}}\circ P_{U^{(0)}_j}=P_{U^{(0)}_{j'}}$. It follows that $P_{U^{(0)}_{j'}}\circ \mathfrak{h}_j=\mathfrak{h}_{j'}$, hence $\mathfrak{h}_j(\varphi)=0$ implies $\mathfrak{h}_{j'}(\varphi)=0$. Moreover, by definition,
$\mathfrak{h}_j(h)=0$ for every $h\in E_j$, and $\mathfrak{h}_j(h)=h$ for every $h\in U_j$, in particular $\mathfrak{h}_j\circ \mathfrak{h}_j=\mathfrak{h}_j$.
\end{remark}

%

\subsection{First step: correction operator $\mathfrak{h}^*_j$ on BV}\label{sec;first-cor}
As {the} first step, we construct an initial extended correction operator $\mathfrak{h}^*_j$ on the space of bounded variation functions taking value in the space  $U_{-j}$ from the complementary filtration.  \medskip

By definition, for every $\varphi \in \bv(\sqcup_{\alpha \in \mathcal{A}}I_\alpha^{(k)} )$, we have
\begin{equation}\label{eqn:normbv}
\big\|\M^{(k)}(\varphi)\big\| \leq  \|\varphi\|_{\sup}\text{ and }\big\|\varphi-\M^{(k)}(\varphi)\big\|_{\sup} \leq  \var{\varphi}.
\end{equation}
Let
$P^{(k)}_0: L^1(\sqcup_{\alpha \in \mathcal{A}} I_\alpha^{(k)}) \rightarrow L^1(\sqcup_{\alpha \in \mathcal{A}} I_\alpha^{(k)})$ be a linear operator given by
\[P^{(k)}_0(\varphi) = \varphi - \M^{(k)}(\varphi).\]
If $\varphi \in BV(\sqcup_{\alpha \in \mathcal{A}} I_\alpha)$, then
\begin{equation}\label{eqn;Pkv}
\|P^{(k)}_0(S(k)\varphi)\|_{\sup} \leq \var(S(k)\varphi).
\end{equation}
By \S 6.1 in \cite{Fr-Ki}, for every $0\leq a<1$ and $\varphi \in C^{0+\pa}(\sqcup_{\alpha \in \mathcal{A}}I_\alpha^{(k)} )$,
\begin{gather}\label{eqn;normMk}
\big\|\M^{(k)}(\varphi)\big\|_{L^1(I^{(k)})} \leq 2 \norm{\varphi}_{L^1(I^{(k)})}\\
\label{eqn;Mkaverage}
\big\|\varphi - \M^{(k)}(\varphi)\big\|_{L^1(I^{(k)})} \leq  \frac{2^{2+a}d}{1-a}p_a(\varphi)|I^{(k)}|^{1-a}.
\end{gather}
Therefore, for $\varphi \in C^{0+\pa}(\sqcup_{\alpha \in \mathcal{A}}I_\alpha )$, we obtain
\begin{equation}\label{eqn;upperboundvarphi}
\frac{\|S(k)\varphi\|_{L^1(I^{(k)})}}{|I^{(k)}|} \leq \norm{\mathcal{M}^{(k)}(S(k)\varphi)} +p_a(S(k)\varphi)
\frac{2^{2+a}}{(1-a)|I^{(k)}|^{a}}.
\end{equation}
As
\begin{equation}\label{equiv-norm}
\frac{|I^{(k)}| \norm{h}}{\kappa} \leq \min_{\beta \in \mathcal{A}} |I_\beta^{(k)}|\norm{h} \leq \norm{h}_{L^1(I^{(k)})} \leq |I^{(k)}|\norm{h}
\text{ for every }h \in \Gamma^{(k)},
\end{equation}
by \eqref{eqn;normMk},
for every $\varphi \in C^{0+\pa}(\sqcup_{\alpha \in \mathcal{A}}I_\alpha )$,
\begin{equation}\label{eqn;Mknorm}
\|\M^{(k)}(\varphi)\| \leq \frac{2\kappa }{|I^{(k)}|}\norm{\varphi}_{L^1(I^{(k)})}.
\end{equation}

\begin{lemma}\label{lem:defh-j}
Let $0\leq j\leq g$ and $\varphi \in BV(\sqcup_{\alpha \in \mathcal{A}} I_\alpha)$ be such that
\begin{equation}\label{eq:ser}
\sum_{l \geq 1} \|Q|_{U^{(0)}_{-j}}(l)^{-1}\|\|Q(l)\|^\tau\|Z(l)\|\var(S(l-1)\varphi)<+\infty.
\end{equation}
Then the limit
\begin{equation}\label{def:h*j}
\mathfrak{h}^*_j(\varphi)=\lim_{l\to\infty}Q(0,l)^{-1}\circ P_{U^{(l)}_{-j}}\circ\mathcal{M}^{(l)}\circ S(l)(\varphi)\in U_{-j}
\end{equation}
exists, and there exists a universal constant $C>0$ such that
\begin{align}\label{eq:normh}
\|\mathfrak{h}^*_j(\varphi)\|\leq  C\Big( \norm{\varphi}_{\sup}
+\sum_{l \geq 1} \|Q|_{U^{(0)}_{-j}}(l)^{-1}\|\|Q(l)\|^\tau\|Z(l)\|\var(S(l-1)\varphi)\Big).
\end{align}
Moreover, for every $k\geq 1$, we have
\begin{gather}\label{eq:szam}
\begin{split}
&\big\|\mathcal{M}^{(k)}(S(k)(\varphi-\mathfrak{h}^*_j(\varphi)))\big\|\\
&\leq C\Big(\sum_{l> k} \|Q|_{U^{(k)}_{-j}}(k,l)^{-1}\|\|Q(l)\|^\tau\|Z(l)\|\var(S(l-1)\varphi) \\
&+\sum_{1\leq l\leq k}\|Q|_{E^{(l)}_{-j}}(l,k)\|\|Q(l)\|^\tau\|Z(l)\|\var(S(l-1)\varphi)+
\|Q|_{E^{(0)}_{-j}}(k)\| \norm{\varphi}_{\sup}\Big).
\end{split}
\end{gather}
\end{lemma}
\begin{proof}
Let  $v_k := \M^{(k)} \circ S(k)(\varphi)$.
Direct calculation shows that
\begin{align}\label{eq:Stodelta}
\begin{split}
\big(S(k,k+1)&\circ P^{(k)}_0 \circ S(k)(\varphi) -P_0^{(k+1)}\circ
S(k,k+1)\circ S(k)(\varphi)\big) \\
& =  -S(k,k+1)\circ \M^{(k)} \circ S(k)(\varphi) + \M^{(k+1)}\circ S(k+1)(\varphi)\\
& = -Z(k+1)v_k + v_{k+1}.
\end{split}
\end{align}
Then, by \eqref{eqn;Pkv} and \eqref{neq:Skvar},
\begin{align*}
\|S(k,k+1)\circ P^{(k)}_0 \circ S(k)(\varphi)\|_{\sup} &\leq \|Z(k+1)\|\|P^{(k)}_0 \circ S(k)(\varphi)\|_{\sup}\\
& \leq \|Z(k+1)\| \var(S(k)\varphi),
\end{align*}
and
\begin{align*}
\|P^{(k+1)}_0 \circ S(k+1)(\varphi)\|_{\sup} \leq \var(S(k+1)\varphi)\leq \var(S(k)\varphi).
\end{align*}
This gives
\begin{equation}\label{eqn;anvv0}
\|Z(k+1)v_k - v_{k+1}\| \leq 2\|Z(k+1)\| \var(S(k)\varphi).
\end{equation}

For any sequence $(x_k)_{k\geq 0}$ in $\R^{\mathcal{A}}$, let $\Delta x_{k+1} = x_{k+1}-Z(k+1)x_k$ for $k\geq 0$, and $\Delta x_0=x_0$.
Then, by telescoping,
\begin{align}\label{def;ehk0}
x_{k} = \sum_{j=0}^{k} Q(j,k) \Delta x_j.
\end{align}
By \eqref{eqn:normbv} and \eqref{eqn;anvv0},
\begin{equation}\label{eqn:delta00}
\|\Delta v_0\|  \leq  \|\varphi\|_{\sup}\text{ and }\norm{\Delta v_{k+1}}\leq 2\|Z(k+1)\| \var(S(k)\varphi).
\end{equation}
For every $k\geq 0$, let $e_k=P_{E^{(k)}_{-j}}v_k\in E_{-j}^{(k)}$ and $u_k=P_{U^{(k)}_{-j}}v_k\in U_{-j}^{(k)}$. Then $v_ k = u_k + e_k$. Since $Z(k+1)(E_{-j}^{(k)})=E_{-j}^{(k+1)}$ and $Z(k+1)(U_{-j}^{(k)})=U_{-j}^{(k+1)}$, we have
\begin{gather}
\label{eq:deltau}
\Delta u_{k+1} = u_{k+1}-Z(k+1)u_k = P_{U^{(k+1)}_{-j}}\Delta v_{k+1}, \\
 \Delta e_{k+1} = e_{k+1}-Z(k+1)e_k=P_{E^{(k+1)}_{-j}}\Delta v_{k+1},\nonumber\\
\Delta u_0=u_0=P_{U^{(0)}_{-j}}\Delta v_{0}, \ \Delta e_0 = e_0=P_{E^{(0)}_{-j}}\Delta v_{0}.\nonumber
\end{gather}
In view of \eqref{eqn;projbound} and \eqref{eqn:delta00}, we have
\begin{align}\label{u0e00}
\norm{\Delta u_0}\leq C\norm{\Delta v_0}\leq  C\norm{\varphi}_{\sup}, \quad \norm{\Delta e_0}\leq C\norm{\Delta v_0}\leq  C\norm{\varphi}_{\sup},
\end{align}
and for every $k\geq 1$, we have
\begin{align}\label{ukek0}
\begin{split}
\norm{\Delta u_{k}} \leq 2C\|Q(k)\|^{\tau}\|Z(k)\| \var(S(k-1)\varphi), \\
\norm{\Delta e_{k}} \leq 2C\|Q(k)\|^{\tau}\|Z(k)\| \var(S(k-1)\varphi).
\end{split}
\end{align}

Let us consider the infinite series $v :=  \sum_{l \geq 0} Q(l)^{-1}\Delta u_l$. Since
\begin{align}\label{eq:suml}
\begin{split}
\sum_{l \geq 0}& \|Q|_{U^{(0)}_{-j}}(l)^{-1}\|\|\Delta u_l\| \\
&\leq C\big(\norm{\varphi}_{\sup}+2\sum_{l \geq 1}\|Q|_{U^{(0)}_{-j}}(l)^{-1}\|\|Q(l)\|^{\tau}\|Z(l)\| \var(S(l-1)\varphi)\big)
\end{split}
\end{align}
is finite,
$v\in U_{-j}$ is well defined. In view of \eqref{eq:Stodelta} and \eqref{eq:deltau}, we have
\begin{align*}
 Q&(l)^{-1}\Delta u_l=Q(l)^{-1}\circ P_{U^{(l)}_{-j}}(\M^{(l)}\circ S(l)(\varphi)-S(l-1,l)\circ \M^{(l-1)} \circ S(l-1)(\varphi))\\
 &=Q(l)^{-1}\circ P_{U^{(l)}_{-j}}\circ\M^{(l)}\circ S(l)(\varphi)-Q(l-1)^{-1}\circ P_{U^{(l-1)}_{-j}}\circ\M^{(l-1)}\circ S(l-1)(\varphi).
\end{align*}
It follows that $\mathfrak{h}^*_j(\varphi)$ is well defined and $\mathfrak{h}^*_j(\varphi)=v$, so by \eqref{eq:suml}, we obtain \eqref{eq:normh}.

By the definition of $v$, \eqref{def;ehk0} and \eqref{ukek0}, for every $k\geq 0$, we have
\begin{align}\label{eqn;qedifference0}
\begin{split}
 \|Q(k)v - u_{k}\|& = \Big\|\sum_{l> k} Q|_{U^{(k)}_{-j}}(k,l)^{-1}\Delta u_{l}\Big\|\leq \sum_{l> k} \|Q|_{U^{(k)}_{-j}}(k,l)^{-1}\|\|\Delta u_{l}\| \\
  &\leq 2C\sum_{l> k} \|Q|_{U^{(k)}_{-j}}(k,l)^{-1}\|\|Q(l)\|^{\tau}\|Z(l)\| \var(S(l-1)\varphi).
\end{split}
 \end{align}
To obtain the bound of norm of $e_k \in E_{-j}^{(k)}$, we apply \eqref{def;ehk0}, \eqref{ukek0} and \eqref{u0e00},
\begin{align*}\label{eqn;hk}
\begin{split}
&\|e_k\| \leq \sum_{0\leq l\leq k}\|Q(l,k)\Delta e_{l}\|\leq\sum_{0\leq l\leq k}\|Q|_{E^{(l)}_{-j}}(l,k)\|
\|\Delta e_{l}\|\\
&\leq C\big(\|Q|_{E^{(0)}_{-j}}(k)\|\|\varphi\|_{\sup}+2\sum_{1\leq l\leq k}\|Q|_{E^{(l)}_{-j}}(l,k)\|
\|Q(l)\|^{\tau}\|Z(l)\| \var(S(l-1)\varphi)\big).
\end{split}
\end{align*}
Combining with \eqref{eqn;qedifference0}, we conclude
\begin{align*}
\|&Q(k)v - v_{k}\|\leq 2C\Big(\sum_{l> k} \|Q|_{U^{(k)}_{-j}}(k,l)^{-1}\|\|Q(l)\|^\tau\|Z(l)\|\var(S(l-1)\varphi) \\
&+\sum_{1\leq l\leq k}\|Q|_{E^{(l)}_{-j}}(l,k)\|\|Q(l)\|^\tau\|Z(l)\|\var(S(l-1)\varphi)+
\|Q|_{E^{(0)}_{-j}}(k)\| \norm{\varphi}_{\sup}\Big).
\end{align*}
Since $\mathcal{M}^{(k)}(S(k)(\mathfrak{h}^*_{j}(\varphi)))=Q(k)v$, this gives \eqref{eq:szam}.
\end{proof}
\begin{remark}
\label{rem:hjhj'-}
Suppose that $0\leq j\leq j'\leq g$. Then, the operator $\mathfrak{h}^*_{j'}$ is well defined, and $P_{U^{(0)}_{-j'}}\circ \mathfrak{h}^*_j=\mathfrak{h}^*_{j'}$.  Hence $\mathfrak{h}^*_j(\varphi)=0$ implies $\mathfrak{h}^*_{j'}(\varphi)=0$. In view of \eqref{def:hi} and \eqref{def:h*j}, the same arguments show that for every $2\leq l\leq g+1$, we have $P_{U^{(0)}_{l}}\circ \mathfrak{h}^*_j=\mathfrak{h}_{l}$.  Hence, $\mathfrak{h}^*_j(\varphi)=0$ implies $\mathfrak{h}_{l}(\varphi)=0$. Moreover, by definition,
$\mathfrak{h}^*_{j}(h)=0$ for every $h\in E_{-j}$, and $\mathfrak{h}^*_{j}(h)=h$ for every $h\in U_{-j}$.
\end{remark}

\subsection{Second step: correction operator $\mathfrak{h}_{-j,i}$}\label{sec;extcp}
Now we introduce a second type correction operators $\mathfrak{h}_{-j,i}:C^{1+\pag}(\sqcup_{\alpha \in \mathcal{A}} I_\alpha)\to U_{-j}$ for $2\leq i\leq g+1$ and $1\leq j\leq g$. These operators extend previous (standard) correction operators $\mathfrak{h}_{i}$ to the complement of the stable part of the Oseledets filtration. For this purpose, we {pre-compose  the operator $\mathfrak{h}^*_{j}$ with an operator $K_i$ such that $\mathfrak{h}_{i}\circ D\circ K_i=0$.}

For all $0\leq a<1$ and $2\leq i\leq g+1$, let
\[C_i^{1+\pag}(\sqcup_{\alpha \in \mathcal{A}} I_\alpha)=\{\varphi\in C^{1+\pag}(\sqcup_{\alpha \in \mathcal{A}} I_\alpha):\mathfrak{h}_i(D\varphi) = 0\}.\]
Let us consider the sequence $\bar{s}=(s_k)_{k\geq 0}$ given by $s_k:=|I^{(k)}|(K^{a,i,\tau}_{k}+C^{a,i,\tau}_{k})$. 

{Throughout the paper, we use standard big $O$ notation. We  emphasize that the constant  implicit in this notation does not depend on the function $\varphi$, but does depend on the parameter $\tau>0$.}

\begin{theorem}\label{thm;correction2}
Assume that $T$ satisfies the \ref{FDC}. Let $0\leq a<1$,  $2\leq i\leq g+1$ and $1\leq j\leq g$ so that
$a\lambda_1  < \lambda_{i-1}$
and $\max\{a\lambda_1,\lambda_i\}  < \lambda_1-\lambda_{j+1}$.
Then the linear operator
$\mathfrak{h}_j^*: C_i^{1+\pag}(\sqcup_{\alpha \in \mathcal{A}} I_\alpha) \rightarrow U_{-j} $ is well defined and bounded. Moreover, for any $0<\tau<\frac{\max\{a\lambda_1,\lambda_i\} - \lambda_1+\lambda_{j+1}}{9(1+\lambda_1)}$, there exists a constant $C=C_\tau\geq 1$ such that for any $\varphi \in C_i^{1+\pag}(\sqcup_{\alpha \in \mathcal{A}} I_\alpha) $ with $\mathfrak{h}^*_j(\varphi) = 0$, we have
\begin{align}
\label{eqn;varsk}
\var(S(k)\varphi) &\leq C s_k\norm{D\varphi}_{C^{0+\pa}},\\
\label{eqn;maincorrection3}
\|\mathcal{M}^{(k)}(S(k)\varphi)\| &\leq C\big(\big(W_k^{j,\tau}(\bar{s})+V_k^{j,\tau}(\bar{s})\big) \norm{D\varphi}_{C^{0+\pa}}+ \|Q|_{E^{(0)}_{-j}}(k)\| \norm{\varphi}_{\sup}\big)
\end{align}
with
\begin{align*}
s_k&=O(e^{(\max\{\lambda_{i},\lambda_1 a\}-\lambda_1+6\tau(1+\lambda_1))r(0,k+1)}),\\
V_{k}^{j,\tau}(T,\bar{s})&=O(e^{(\max\{\lambda_{i},\lambda_1 a\}-\lambda_1+8\tau(1+\lambda_1))r(0,k)}),\\
W_{k}^{j,\tau}(T,\bar{s})&=O(e^{(\max\{\lambda_{i}-\lambda_1,\lambda_1 a-\lambda_1,-\lambda_j\}+9\tau(1+\lambda_1))r(0,k)}).
\end{align*}
\end{theorem}

\begin{proof}
As  $\varphi \in C_i^{1+\pag}(\sqcup_{\alpha \in \mathcal{A}} I_\alpha)$,  we have $D\varphi \in C^{0+\pag}(\sqcup_{\alpha \in \mathcal{A}} I_\alpha)$ and $\mathfrak{h}_{i}(D\varphi) = 0$.
By \eqref{eqn;upperboundvarphi}, \eqref{eqn;renormpaos2},  \eqref{eq:invIk}, and Theorem \ref{thm;correction}, we have
\begin{align*}
\begin{split}
\var(&S(k)\varphi) = \norm{S(k)(D\varphi)}_{L^1(I^{(k)})} \\
& \leq |I^{(k)}|\Big(\norm{\mathcal{M}^{(k)}(S(k)D\varphi)} +p_a(S(k)D\varphi)\frac{2^{a+2}}{(1-a)||I^{(k)}|^{a}}\Big)\\
& \leq C_a|I^{(k)}|\left(\big(K^{a,i,\tau}_k+C^{a,i,\tau}_k\big) p_a(D\varphi)+ \|Q|_{E_{i}}(k)\| \frac{\norm{D\varphi}_{L^1(I^{(0)})}}{|I^{(0)}|}\right)\\
& \leq C_a|I^{(k)}|\big(K^{a,i,\tau}_k+C^{a,i,\tau}_k\big)\norm{D\varphi}_{C^{0+\pa}}\leq C_\tau s_k\norm{D\varphi}_{C^{0+\pa}},
\end{split}
\end{align*}
which gives \eqref{eqn;varsk}. It follows that
\begin{align*}
\|Q&|_{U^{(k)}_{-j}}(k,l)^{-1}\|\|Q(l)\|^\tau\|Z(l)\|\var(S(l-1)\varphi) \\
&=O\big(\|Q|_{U^{(k)}_{-j}}(k,l)^{-1}\|\|Q(l)\|^\tau \|Z(l)\| s_{l-1}\norm{D\varphi}_{C^{0+\pa}}\big),\\
\|Q&|_{E^{(l)}_{-j}}(l,k)\|\|Q(l)\|^\tau\|Z(l)\|\var(S(l-1)\varphi)\\
&=O\big(\|Q|_{E^{(l)}_{-j}}(l,k)\|\|Q(l)\|^\tau \|Z(l)\| s_{l-1}\norm{D\varphi}_{C^{0+\pa}}\big).
\end{align*}
In view of \eqref{prop;sdc-1}, \eqref{eq:invIk}, and \eqref{eq:nrn}, we have
\begin{align*}
s_k&= O(e^{-\lambda_1k}e^{(\max\{\lambda_{i},\lambda_1 a\}+5\tau(1+\lambda_1))r(0,k)})\\
&= O(e^{-\lambda_1(1-\tau)r(0,k+1)}e^{(\max\{\lambda_{i},\lambda_1 a\}+5\tau(1+\lambda_1))r(0,k)})\\
&= O(e^{(\max\{\lambda_{i},\lambda_1 a\}-\lambda_1+6\tau(1+\lambda_1))r(0,k+1)}).
\end{align*}
As $\max\{\lambda_{i},\lambda_1 a\}+\lambda_{j+1}-\lambda_1+6\tau(1+\lambda_1)+\tau(3+\lambda_1)<0$, by Proposition~\ref{prop;FDCbound4},
\begin{align*}
V_k^{j,\tau}(T,\bar{s})&=O(e^{(\max\{\lambda_{i},\lambda_1 a\}-\lambda_1+8\tau(1+\lambda_1))r(0,k)}),\\
W_k^{j,\tau}(T,\bar{s})&=O(e^{(\max\{\lambda_{i}-\lambda_1,\lambda_1 a-\lambda_1,-\lambda_j\}+9\tau(1+\lambda_1))r(0,k)}).
\end{align*}
As $V_0^{j,\tau}(T,\bar{s})$ is finite, the series \eqref{eq:ser} is convergent. By Lemma~\ref{lem:defh-j} (see \eqref{eq:normh}), the operator
$\mathfrak{h}_j^*: C_i^{1+\pag}(\sqcup_{\alpha \in \mathcal{A}} I_\alpha) \rightarrow U_{-j} $ is well defined and bounded.
Moreover, in view of \eqref{eq:szam}, this also gives \eqref{eqn;maincorrection3}.
\end{proof}

For every $\varphi\in L^1(I)$, we denote by $\widetilde{\varphi}\in AC(I)$ its primitive integral $\widetilde{\varphi}(x)=\int_0^x\varphi(y)dy$.

\begin{corollary}\label{cor:h-j}
Assume that $T$ satisfies the \ref{FDC}. Let $0\leq a<1$, $2\leq i\leq g+1$, and $1\leq j\leq g$ so that $a\lambda_1<\lambda_{i-1}$ and $\max\{a\lambda_1,\lambda_i\}<\lambda_1-\lambda_{j+1}$.
There exists a bounded operator $\mathfrak{h}_{-j,i}:C^{1+\pag}(\sqcup_{\alpha \in \mathcal{A}} I_\alpha)\to U_{-j}$
such that for every $\varphi \in C^{1+\pag}(\sqcup_{\alpha \in \mathcal{A}} I_\alpha)$ with $\mathfrak{h}_{-j,i}(\varphi)=0$ and $\mathfrak{h}_{i}(D\varphi)=0$, we have
\begin{equation}\label{eqn;maincorrection2}
\|S(k)\varphi\|_{\sup} = O(e^{(\max\{\lambda_{i}-\lambda_1,\lambda_1 a -\lambda_1, -\lambda_{j}\} +\tau)r(0,k)})\|\varphi\|_{C^{1+\pa}}
\quad\text{for any}\quad\tau>0.
\end{equation}
\end{corollary}

\begin{proof}
Let $K_{i}:C^{1+\pag}(\sqcup_{\alpha \in \mathcal{A}} I_\alpha)\to C_i^{1+\pag}(\sqcup_{\alpha \in \mathcal{A}} I_\alpha)$ be the bounded operator defined by $K_{i}(\varphi)=\varphi-\widetilde{\mathfrak{h}_{i}(D\varphi)}$. Since
$\mathfrak{h}_{i}(DK_{i}(\varphi))=\mathfrak{h}_{i}(D\varphi)-\mathfrak{h}_{i}(\mathfrak{h}_{i}(D\varphi))=0$,
we really have $K_{i}(\varphi)\in C_i^{1+\pag}(\sqcup_{\alpha \in \mathcal{A}} I_\alpha)$.
We can use Theorem~\ref{thm;correction2} to define $\mathfrak{h}_{-j,i}:C^{1+\pag}(\sqcup_{\alpha \in \mathcal{A}} I_\alpha)\to U_{-j}$ as $\mathfrak{h}_{-j,i}:=\mathfrak{h}^*_{j}\circ K_{i}$.

Suppose that $\varphi \in C^{1+\pag}(\sqcup_{\alpha \in \mathcal{A}} I_\alpha)$ is such that
$\mathfrak{h}_{-j,i}(\varphi)=0$ and $\mathfrak{h}_{i}(D\varphi)=0$. Then
{$\mathfrak{h}^*_{j}(K_{i}(\varphi))=\mathfrak{h}_{-j,i}(\varphi)=0 \text{ and } \varphi=K_{i}(\varphi)+\widetilde{\mathfrak{h}_{i}(D\varphi)}=K_{i}(\varphi),$}
 so $\mathfrak{h}^*_{j}(\varphi)=0$. In view of \eqref{eqn;Pkv} and Theorem~\ref{thm;correction2}, 
\begin{align*}
\|S(k)\varphi\|_{\sup}&\leq \|\mathcal{M}^{(k)}(S(k)\varphi)\|+\var(S(k)\varphi)=O(e^{(\max\{\lambda_{i}-\lambda_1,\lambda_1 a -\lambda_1, -\lambda_{j}\} +\tau)r(0,k)}).
\end{align*}
\end{proof}

\begin{remark}
\label{rem:hji}
Suppose that $1\leq j\leq j'\leq g$ and $2\leq i'\leq i\leq g+1$. Then the operator $\mathfrak{h}_{-j',i'}$ is well defined. By Remarks~\ref{rem:hjhj'} and \ref{rem:hjhj'-}, $\mathfrak{h}_{-j,i}(\varphi)=0$ and $\mathfrak{h}_{i}(D\varphi)=0$ imply
\[\mathfrak{h}_{-j',i'}(\varphi)=0, \ \mathfrak{h}_{i'}(D\varphi)=0 ,\quad\text{and}\quad \mathfrak{h}_{l}(\varphi)=0\quad \text{for every}\quad 2\leq l\leq g+1.\]
Moreover, by the same remarks, we also have
$\mathfrak{h}_{-j,i}(h)=0$ for every $h\in E_{-j}$, and $\mathfrak{h}_{-j,i}(h)=h$ for every $h\in U_{-j}$, in particular $\mathfrak{h}_{-j,i}\circ \mathfrak{h}_{-j,i}=\mathfrak{h}_{-j,i}$.
\end{remark}

\subsection{Third step: correction operator $\mathfrak{h}_{0}$}\label{sec;higher-cor}
The last correction operator  $\mathfrak{h}_{0}:C^{2+\pag}(\sqcup_{\alpha \in \mathcal{A}} I_\alpha)\to U_{0}=\Gamma$ plays the same roles as  $\mathfrak{h}_{-j,i}$ but {only} for the parameter $j=0$.
As in the construction of $\mathfrak{h}_{-j,i}$, we also use the  operator $\mathfrak{h}^*_{j}$ (for $j=0$), but we need to link it with the derivative of $\mathfrak{h}_{-g,2}$, and the second derivative of $\mathfrak{h}_{2}$.

\begin{theorem}\label{thm:defh0}
Assume that $T$ satisfies the \ref{FDC}. Let $0\leq a<1$.
There exists a bounded operator $\mathfrak{h}_{0}:C^{2+\pag}(\sqcup_{\alpha \in \mathcal{A}} I_\alpha)\to U_{0}$
such that if $\varphi \in C^{2+\pag}(\sqcup_{\alpha \in \mathcal{A}} I_\alpha)$ satisfies
\begin{gather*}
\mathfrak{h}_{0}(\varphi)=0,\ \mathfrak{h}_{-g,2}(D\varphi)=0,\ \mathfrak{h}_{2}(D^2\varphi)=0,\text{  and }
\|S(k)D\varphi\|_{\sup} = O(e^{-\rho r(0,k)})c(D\varphi),
\end{gather*}
{for some  constants $\rho>0$ and $c(D\varphi)\geq 0$,}
then for every
\[0<\tau<\min\{\lambda_{1}-\lambda_2,\lambda_1(1-a), \lambda_{g},\rho\}/3(1+\max\{\lambda_1,\rho\}),\]
we have
\begin{equation}\label{eqn;maincorrection4}
\|S(k)\varphi\|_{\sup} = O(e^{(-\rho-\lambda_1+2\tau(\lambda_1+\rho+1))r(0,k)})c(D\varphi).
\end{equation}
\end{theorem}

\begin{proof}
Let us consider
\[
C_{-g,2}^{2+\pag}(\sqcup_{\alpha \in \mathcal{A}} I_\alpha)=\{\varphi\in C^{2+\pag}(\sqcup_{\alpha \in \mathcal{A}} I_\alpha):\mathfrak{h}_{-g,2}(D\varphi) = 0, \mathfrak{h}_2(D^2\varphi) = 0\}.\]
By Corollary~\ref{cor:h-j}, for every $\varphi\in C_{-g,2}^{2+\pag}(\sqcup_{\alpha \in \mathcal{A}} I_\alpha)$, we have
\begin{align*}
\|S(k)D\varphi\|_{\sup}&=O(e^{(\max\{\lambda_{2}-\lambda_1,\lambda_1 a -\lambda_1, -\lambda_{g}\} +\tau)r(0,k)})\|D\varphi\|_{C^{1+\pa}}\\
&=O(e^{-\rho_0r(0,k)})\|D\varphi\|_{C^{1+\pa}},
\end{align*}
with $\rho_0:=\min\{\lambda_{1}-\lambda_2,\lambda_1(1-a), \lambda_{g}\} -\tau>0$. As
\[\var(S(k)\varphi)=\|S(k)D\varphi\|_{L^1(I^{(k)})}\leq |I^{(k)}|\|S(k)D\varphi\|_{\sup},\]
it follows that, for $l>k$, we have
\begin{align*}
\|Q&|_{U^{(k)}_{0}}(k,l)^{-1}\|\|Q(l)\|^\tau\|Z(l)\|\var(S(l-1)\varphi) \\
&\leq \|Q(k,l)^{-1}\|\|Q(l)\|^\tau\|Z(l)\||I^{(l-1)}|\|S(l-1)D\varphi\|_{\sup}\\
&= O(e^{(\lambda_1+\tau)r(k,l)}e^{\tau(\lambda_1+\tau)l}e^{\tau l}e^{-\lambda_1(l-1)}e^{-\rho_0r(0,l-1)})\|D\varphi\|_{{C^{1+\pa}}}.
\end{align*}
By \eqref{eq:nrn}, it follows that
\begin{align}\label{eq:U01}
\begin{split}
\|&Q|_{U^{(k)}_{0}}(k,l)^{-1}\|\|Q(l)\|^\tau\|Z(l)\|\var(S(l-1)\varphi) \\
&= O(e^{(\lambda_1+\tau)r(k,l)}e^{\tau(\lambda_1+2)r(0,l)}e^{-(1-\tau)\lambda_1r(0,l)}e^{-(1-\tau)\rho_0r(0,l)})\|D\varphi\|_{{C^{1+\pa}}}\\
&= O(e^{(-\lambda_1-\rho_0+\tau(3\lambda_1+2))r(0,k)}e^{(-\rho_0+\tau(3\lambda_1+3))r(k,l)})\|D\varphi\|_{{C^{1+\pa}}}.
\end{split}
\end{align}
The same arguments show that if additionally $\|S(k)D\varphi\|_{\sup} = O(e^{-\rho r(0,k)})c(D\varphi)$, then
\begin{align}\label{eq:U02}
\begin{split}
\|Q&|_{U^{(k)}_{0}}(k,l)^{-1}\|\|Q(l)\|^\tau\|Z(l)\|\var(S(l-1)\varphi) \\
&= O(e^{(-\lambda_1-\rho+\tau(2\lambda_1+\rho+2))r(0,k)}e^{(-\rho+\tau(2\lambda_1+\rho+3))r(k,l)})c(D\varphi).
\end{split}
\end{align}
As $-\rho_0+3\tau(\lambda_1+1)<0$, by \eqref{eq:U01},  the series \eqref{eq:ser} is convergent for $j=0$. By Lemma~\ref{lem:defh-j},
the operator $\mathfrak{h}_0^*:C_{-g,2}^{2+\pag}(\sqcup_{\alpha \in \mathcal{A}} I_\alpha)\to U_0=\Gamma$ is well defined, and
if $\mathfrak{h}_0^*(\varphi)=0$, then
\[
\big\|\mathcal{M}^{(k)}(S(k)\varphi)\big\|\leq C\sum_{l> k} \|Q|_{U^{(k)}_{0}}(k,l)^{-1}\|\|Q(l)\|^\tau\|Z(l)\|\var(S(l-1)\varphi).
\]
Therefore,
\begin{gather}\label{eq:szam1}
\begin{split}
\|S(k)\varphi\|_{\sup}&\leq
\big\|\mathcal{M}^{(k)}(S(k)\varphi)\big\|+\var(S(k)\varphi)\\
&\leq 2C\sum_{l> k} \|Q|_{U^{(k)}_{0}}(k,l)^{-1}\|\|Q(l)\|^\tau\|Z(l)\|\var(S(l-1)\varphi).
\end{split}
\end{gather}
Let $K:C^{2+\pag}(\sqcup_{\alpha \in \mathcal{A}} I_\alpha)\to C_{-g,2}^{2+\pag}(\sqcup_{\alpha \in \mathcal{A}} I_\alpha)$ be the bounded operator defined by \[K(\varphi):=\varphi-\widetilde{\widetilde{\mathfrak{h}_{2}(D^2\varphi)}}-\widetilde{\mathfrak{h}_{-g,2}(D\varphi)}+\widetilde{\mathfrak{h}_{-g,2}(\widetilde{\mathfrak{h}_{2}(D^2\varphi)})}.\]
Then
\[DK(\varphi)=D\varphi-\widetilde{\mathfrak{h}_{2}(D^2\varphi)}-\mathfrak{h}_{-g,2}(D\varphi-\widetilde{\mathfrak{h}_{2}(D^2\varphi)})\text{ and }
D^2K(\varphi)=D^2\varphi-{\mathfrak{h}_{2}(D^2\varphi)}.\]
Since
$\mathfrak{h}_{2}(D^2K(\varphi))=\mathfrak{h}_{2}(D^2\varphi)-\mathfrak{h}_{2}(\mathfrak{h}_{2}(D^2\varphi))=0$
and \[\mathfrak{h}_{-g,2}(DK(\varphi))=\mathfrak{h}_{-g,2}(D\varphi-\widetilde{\mathfrak{h}_{2}(D^2\varphi)})-
\mathfrak{h}_{-g,2}(\mathfrak{h}_{-g,2}(D\varphi-\widetilde{\mathfrak{h}_{2}(D^2\varphi)}))=0,\]
we really have $K(\varphi)\in C_{-g,2}^{2+\pag}(\sqcup_{\alpha \in \mathcal{A}} I_\alpha)$.
Finally, we define $\mathfrak{h}_{0}:C^{2+\pag}(\sqcup_{\alpha \in \mathcal{A}} I_\alpha)\to U_{0}$ as $\mathfrak{h}_{0}=\mathfrak{h}^*_{0}\circ K$.

Suppose that  $\varphi \in C^{2+\pag}(\sqcup_{\alpha \in \mathcal{A}} I_\alpha)$ is such that
\[
\mathfrak{h}_{0}(\varphi)=0,\ \mathfrak{h}_{-g,2}(D\varphi)=0,\ \mathfrak{h}_{2}(D^2\varphi)=0,  \text{ and } \|S(k)D\varphi\|_{\sup} = O(e^{-\rho r(0,k)})c(D\varphi).
\]
Then $\mathfrak{h}^*_{0}(K(\varphi))=\mathfrak{h}_{0}(\varphi)=0$ with $K(\varphi)=\varphi$, so $\mathfrak{h}^*_{0}(\varphi)=0$.
In view of \eqref{eq:szam1} and \eqref{eq:U02}, this gives
\begin{align*}
\|S(k)\varphi\|_{\sup}&= e^{(-\lambda_1-\rho+\tau(2\lambda_1+\rho+2))r(0,k)}O\Big(\sum_{l>k}e^{(-\rho+\tau(2\lambda_1+\rho+3))r(k,l)}\Big)
c(D\varphi).
\end{align*}
As
\[\sum_{l>k}e^{(-\rho+\tau(2\lambda_1+\rho+3))r(k,l)}\leq\sum_{l\geq 1}e^{(-\rho+\tau(2\lambda_1+\rho+3))l}<+\infty,\]
this gives \eqref{eqn;maincorrection4}.
\end{proof}

\begin{remark}\label{rem:h0}
Using Remark~\ref{rem:hjhj'-}~and~\ref{rem:hji}, we obtain that $\mathfrak{h}_{0}(\varphi)=0$, $\mathfrak{h}_{-g,2}(D\varphi)=0$, and $\mathfrak{h}_{2}(D^2\varphi)=0$  imply $\mathfrak{h}_{-j,i}(\varphi)=0$, for any pair $i,j$, and $\mathfrak{h}_{l}(\varphi)=0$ for any $2\leq l\leq g+1$. By Remark~\ref{rem:hjhj'-}, we also have $\mathfrak{h}_{0}(h)=h$ for every $h\in \Gamma$, in particular $\mathfrak{h}_{0}\circ \mathfrak{h}_{0}=\mathfrak{h}_{0}$.
\end{remark}

Finally, we prove a fast exponential decay of special Birkhoff sums for  $\varphi \in C^{n+\pag}(\sqcup_{\alpha \in \mathcal{A}} I_\alpha)$ under some vanishing conditions for derivatives of previously defined  correction operators.
This result is a key step in the construction of invariant distributions $\mathfrak{f}_{\bar t}$, and the proof of the spectral theorem.
\begin{theorem}\label{thm:h-j}
Assume that $T$ satisfies the \ref{FDC}. Let $0\leq a<1$, $2\leq i\leq g+1$, and $1\leq j\leq g$ with $a\lambda_1<\lambda_{i-1}$ and $\max\{a\lambda_1,\lambda_i\}<\lambda_1-\lambda_{j+1}$.
Let $n\geq 1$.
Suppose that $\varphi \in \cpang(\sqcup_{\alpha \in \mathcal{A}} I_\alpha)$ is such that $\mathfrak{h}_{-j,i}(D^{n-1}\varphi)=0$, $\mathfrak{h}_{i}(D^{n}\varphi)=0$  and $\mathfrak{h}_{0}(D^{l}\varphi)=0$ for all $0\leq l<n-1$. Then,
 for every small enough $\tau>0$, we have
\begin{equation}\label{eqn;maincorrection5}
\|S(k)\varphi\|_{\sup} = O(e^{(-n\lambda_1+\max\{\lambda_{i},\lambda_1 a,\lambda_1 -\lambda_{j}\} +\tau)r(0,k)})\|\varphi\|_{{C^{n+\pa}}}.
\end{equation}
\end{theorem}

\begin{proof}
First we show that $\mathfrak{h}_{-g,2}(D^{l+1}\varphi)=0$ and $\mathfrak{h}_{2}(D^{l+2}\varphi)=0$ for all $0\leq l<n-1$. As $\mathfrak{h}_{-j,i}(D^{n-1}\varphi)=0$ and $\mathfrak{h}_{i}(D^{n}\varphi)=0$, by Remark~\ref{rem:hji} applied to $D^{n-1}\varphi$, we have $$\mathfrak{h}_{-g,2}(D^{n-1}\varphi)=0, \ \mathfrak{h}_{2}(D^{n}\varphi)=0 \text{ and } \mathfrak{h}_{2}(D^{n-1}\varphi)=0.$$
This gives our claim for $l=n-2$. As $\mathfrak{h}_{0}(D^{n-2}\varphi)=0$, by Remark~\ref{rem:h0} applied to $D^{n-2}\varphi$, we obtain $\mathfrak{h}_{-g,2}(D^{n-2}\varphi)=0$. Together with $\mathfrak{h}_{2}(D^{n-1}\varphi)=0$ this gives our claim  for $l=n-3$. Repeating the same arguments for lower-order derivatives and using induction, we get our claim for every $0\leq l<n-1$.

The proof of \eqref{eqn;maincorrection5} is also done by induction on $n$. The base case $n=1$ follows directly from Corollary~\ref{cor:h-j}.
Assume that the induction hypothesis \eqref{eqn;maincorrection5} holds for a particular $n\geq 1$. Suppose that $\varphi \in C^{n+1+\pag}(\sqcup_{\alpha \in \mathcal{A}} I_\alpha)$ is such that $\mathfrak{h}_{-j,i}(D^{n}\varphi)=0$, $\mathfrak{h}_{i}(D^{n+1}\varphi)=0$,  and $\mathfrak{h}_{0}(D^{l}\varphi)=0$ for all $0\leq l<n$. By the induction hypothesis applied to $D\varphi$,  for every small enough $\tau>0$, we have
\[\|S(k)D\varphi\|_{\sup} = O(e^{(-n\lambda_1+\max\{\lambda_{i},\lambda_1 a,\lambda_1 -\lambda_{j}\} +\tau)r(0,k)})\|D\varphi\|_{{C^{n+\pa}}}.\]
By assumption and the first part of the proof,
$$\mathfrak{h}_{0}(\varphi)=0, \  \mathfrak{h}_{-g,2}(D\varphi)=0, \ \mathfrak{h}_{2}(D^2\varphi)=0.$$
In view of Theorem~\ref{thm:defh0} applied to $\rho=n\lambda_1-\max\{\lambda_{i},\lambda_1 a,\lambda_1 -\lambda_{j}\} -\tau$, we get
\[\|S(k)\varphi\|_{\sup} = O(e^{(-(n+1)\lambda_1+\max\{\lambda_{i},\lambda_1 a,\lambda_1 -\lambda_{j}\} +2(n+1)(\lambda_1+1)\tau)r(0,k)})\|\varphi\|_{{C^{n+1+\pa}}}.
\]
\end{proof}

\section{Spectrum of the functional KZ-cocycles}\label{sec;spec}
{The} special Birkhoff sums cocycle $S(k)$ is an infinite dimensional extension of the KZ-cocycle. In this section, we compute {the} Lyapunov exponents of the cocycle $S(k)$ on $C^{n+\pa}$.
We construct a finite set of piecewise polynomial functions that form a basis for the spectral {theorem (Theorem~\ref{thm;spdecomp}).
These piecewise polynomials are obtained by applying {the} correction operators constructed in the previous section. In turn, their Lyapunov exponents for the cocycle $S(k)$ are related to the Lyapunov exponents of {the} standard KZ-cocycle.}


\subsection{Lyapunov exponents for piecewise polynomials}
For every $l\geq 0$, we denote by $\R_l[x]$ the linear space of polynomials of degree not greater than $l$. Since every linear operator
defined on a finite dimensional linear space is bounded, for every $l\geq 0$, there exists a constant $c_l>0$ such that for every $f\in \R_l[x]$,
we have $c_l\|D^lf\|_{C^0([0,1])}\leq\|f\|_{L^1([0,1])}$. Therefore, for every interval $I=[a,b]\subset \R$, we obtain
\[\frac{\norm{f}_{L^1(I)}}{|I|}=\norm{f(a+|I|(\,\cdot\,))}_{L^1([0,1])}\geq c_l \|\frac{d^l}{dx^l}f(a+|I|x)\|_{C^0([0,1])}=c_l|I|^l\|D^lf\|_{C^0(I)}. \]
For every  $l\geq 0$, we denote by $\Gamma_l(\sqcup_{\alpha \in \mathcal{A}} I_\alpha)$ the space of maps $f:I\to\R$ such that for every $\alpha\in\mathcal{A}$, the restriction of $f$ to $I_\alpha$ belongs to $\R_l[x]$. { Note that each element of the space $\Gamma_0(\sqcup_{\alpha \in \mathcal{A}} I_\alpha)$ is a piecewise constant function, so it can be treated as an element of the space $\Gamma$. Therefore, we will identify $\Gamma_0(\sqcup_{\alpha \in \mathcal{A}} I_\alpha)$ with $\Gamma$. Then, for every $f\in \Gamma_l(\sqcup_{\alpha \in \mathcal{A}} I_\alpha)$, we have $D^lf\in\Gamma_0(\sqcup_{\alpha \in \mathcal{A}} I_\alpha)=\Gamma$}, and
\begin{equation}\label{eq:lowest}
\frac{1}{|I|}\norm{f}_{L^1(I)}\geq c_l\Big(\min_{\alpha\in\mathcal{A}}|I_\alpha|\Big)^l \|D^lf\|.
\end{equation}

Let  $h_1,\ldots, h_{g}, c_1,\ldots, c_{\gamma-1},h_{-g},\ldots, h_{-1}$ be a basis of $\Gamma$ described in
Section~\ref{sec;OF}. Then
\begin{equation}\label{eq:grhi0}
\lim_{k \rightarrow \infty} \frac{\log\norm{Q(k)h_i}}{k} = \lambda_i \text{ for } 1\leq |i|\leq g,
\ \lim_{k \rightarrow \infty} \frac{\log\norm{Q(k)c_s}}{k}=0\text{ for }1\leq s<\gamma.
\end{equation}
For every $2\leq i\leq g+1$, choose $1\leq j_i\leq g$ such that $\lambda_1-\lambda_{j_i}\leq \lambda_i<\lambda_1-\lambda_{j_i+1}$, and
for every $1\leq j\leq g$, choose $2\leq i_j\leq g+1$ such that $\lambda_{i_j}\leq \lambda_1-\lambda_{j}< \lambda_{i_j-1}$.

\begin{definition}\label{def;hnl}
For every $l\geq 0$, let $h_{i,l}$ for $1\leq i\leq g$, $c_{s,l}$ for $1\leq s<\gamma$, and $h_{-j,l}$, for $1\leq j\leq g$, be elements of  $\Gamma_l(\sqcup_{\alpha \in \mathcal{A}} I_\alpha)$ defined inductively as follows:
\begin{gather*}
h_{i,0}=h_{i},\ h_{i,1}=\widetilde{h_{i}}-\mathfrak{h}_{-j_i,i}(\widetilde{h_{i}}),\ h_{i,l+1}=\widetilde{h_{i,l}}-\mathfrak{h}_0(\widetilde{h_{i,l}}),\text{ for }l\geq 1\text{ if }2\leq i\leq g,\\
h_{1,0}=h_{1},\ h_{1,1}=\widetilde{h_{1}}-\mathfrak{h}_{g+1}(\widetilde{h_{1}}),\ h_{1,2}=\widetilde{h_{1,1}}-\mathfrak{h}_{-1,g+1}(\widetilde{h_{1,1}}),\\
h_{1,l+1}=\widetilde{h_{1,l}}-\mathfrak{h}_0(\widetilde{h_{1,l}}),\text{ for }l\geq 2,\\
c_{s,0}=c_{s},\ c_{s,1}=\widetilde{c_{s}}-\mathfrak{h}_{-1,g+1}(\widetilde{c_{s}}),\ c_{s,l+1}=\widetilde{c_{s,l}}-\mathfrak{h}_0(\widetilde{c_{s,1}}),\text{ for }l\geq 1,\\
h_{-j,0}=h_{-j},\ h_{-j,l+1}=\widetilde{h_{-j,l}}-\mathfrak{h}_0(\widetilde{h_{-j,l}}),\text{ for }l\geq 0.
\end{gather*}
\end{definition}
Since $\mathfrak{h}_0\circ \mathfrak{h}_0=\mathfrak{h}_0$, $\mathfrak{h}_{g+1}\circ \mathfrak{h}_{g+1}=\mathfrak{h}_{g+1}$, and $\mathfrak{h}_{-1,g+1}\circ \mathfrak{h}_{-1,g+1}=\mathfrak{h}_{-1,g+1}$, we obtain
\begin{gather}
\label{eq:prophch1}
D^nh_{i,l}=h_{i,l-n},\ D^nc_{s,l}=c_{s,l-n},\ D^nh_{-j,l}=h_{-j,l-n}, \text{ if }0\leq n\leq l,\\
\label{eq:prophch2}
 \mathfrak{h}_{-j_i,i}(h_{i,1})=0,\ \mathfrak{h}_0(h_{i,l})=0, \text{ for }l\geq 2\text{ if }2\leq i\leq g, \\
\label{eq:prophch3}
 \mathfrak{h}_{g+1}(h_{1,1})=0,\  \mathfrak{h}_{-1,g+1}(h_{1,2})=0,\ \mathfrak{h}_0(h_{1,l})=0, \text{ for }l\geq 3, \\
\label{eq:prophch4}
 \mathfrak{h}_{-1,g+1}(c_{s,1})=0,\ \mathfrak{h}_0(c_{s,l})=0,\text{ for }l\geq 2,\\
\label{eq:prophch5}
 \mathfrak{h}_0(h_{-j,l})=0\text{ for }l\geq 1.
\end{gather}

In view of \eqref{eq:prophch1}, $h_{i,l}$, for $1\leq |i|\leq g$ and $0\leq l\leq n$, together with $c_{s,l}$, for $1\leq s<\gamma$ and $0\leq l\leq n$, form a basis of the space $\Gamma_n(\sqcup_{\alpha \in \mathcal{A}} I_\alpha)$.
Hence every $h\in \Gamma_n(\sqcup_{\alpha \in \mathcal{A}} I_\alpha)$ has a unique decomposition
\[h=\sum_{0\leq l\leq n}\Big(\sum_{1\leq |i|\leq g}d(h,h_{i,l})h_{i,l}+\sum_{1\leq s<\gamma}d(h,c_{s,l})c_{s,l}\Big).\]

Lyapunov exponents of {the cocycle} $S(k)$ for $h_{i,l}$, $c_{s,l}$ are computed by adapting inductive definitions and using Theorem~\ref{thm:h-j}. Their lower bounds are obtained by applying the \ref{FDC} properties of $T$.
\begin{proposition}\label{lem;Qkh}
Assume that $T$ satisfies the \ref{FDC}. Then, for every $l\geq 0$,
\begin{align}\label{eq:grhi}
\begin{split}
&\lim_{k \rightarrow \infty} \frac{\log \norm{S(k)h_{i,l}}_{\sup}}{k} =
\lim_{k \rightarrow \infty} \frac{\log (\norm{S(k)h_{i,l}}_{L^1(I^{(k)})}/|I^{(k)}|)}{k} =\lambda_i - l\lambda_1 \\
&\lim_{k \rightarrow \infty} \frac{\log \norm{S(k)c_{s,l}}_{\sup}}{k} = \lim_{k \rightarrow \infty} \frac{\log (\norm{S(k)c_{s,l}}_{L^1(I^{(k)})}/|I^{(k)}|)}{k}=- l\lambda_1,
\end{split}
\end{align}
 for  $i\in\pm\{1,\ldots,g\}$ and for  $1\leq s<\gamma$.
Moreover, for every $h\in \Gamma_n(\sqcup_{\alpha \in \mathcal{A}} I_\alpha)$,
\begin{align}\label{eq:skcomb}
\begin{split}
\lim_{k \rightarrow \infty} \frac{\log \norm{S(k)h}_{\sup}}{k} =\max\big(&\{\lambda_i - l\lambda_1:0\leq l\leq n,1\leq |i|\leq g,d(h,h_{i,l})\neq 0\}\\
&\cup\{- l\lambda_1:0\leq l\leq n,1\leq s<\gamma,d(h,c_{s,l})\neq 0\}\big).
\end{split}
\end{align}
\end{proposition}
\begin{proof}
If $l=0$, then \eqref{eq:grhi} follows directly from \eqref{eq:grhi0}.
\smallskip

{\textbf{Case 1.}} Suppose that $\varphi=h_{-j,l}$ for some $l\geq 1$. Then $\varphi\in C^{l+1+\pag}(\sqcup_{\alpha \in \mathcal{A}} I_\alpha)$ with $a=0$,  $\mathfrak{h}_{-j,i_j}(D^{l}\varphi)=0$, $\mathfrak{h}_{i_j}(D^{l+1}\varphi)=0$,  and $\mathfrak{h}_{0}(D^{p}\varphi)=0$ for all $0\leq p<l$. Indeed, as $D^{l}\varphi=h_{-j}\in E_{-j}$, by Remark~\ref{rem:hji}, we have $\mathfrak{h}_{-j}(D^{l}\varphi)=\mathfrak{h}_{-j}(h_{-j})=0$ and $D^{l+1}\varphi=Dh_{-j}=0$. In view of Theorem~\ref{thm:h-j}, this gives
\[\limsup_{k \rightarrow \infty} \frac{\log \norm{S(k)\varphi}_{\sup}}{k}\leq -(l+1)\lambda_1+\max\{\lambda_{i_j},\lambda_1 a,\lambda_1 -\lambda_{j}\}=-l\lambda_1-\lambda_j.\]

\textbf{Case 2.} Suppose that $\varphi=h_{i,l}$ for some $2\leq i\leq g$ and $l\geq 1$. Then $\varphi\in C^{l+\pag}(\sqcup_{\alpha \in \mathcal{A}} I_\alpha)$ with $a=0$,  $\mathfrak{h}_{-j_i,i}(D^{l-1}\varphi)=0$, $\mathfrak{h}_{i}(D^{l}\varphi)=0$,  and $\mathfrak{h}_{0}(D^{p}\varphi)=0$ for all $0\leq p<l-1$. Indeed,
as $D^{l}\varphi=h_{i}\in E_{i}$, by Remark~\ref{rem:hjhj'}, we have $\mathfrak{h}_{i}(D^{l}\varphi)=0$. Moreover, by definition,  $\mathfrak{h}_{-j_i,i}(D^{l-1}h_{i,l})=0$. In view of Theorem~\ref{thm:h-j}, this gives
\[\limsup_{k \rightarrow \infty} \frac{\log \norm{S(k)\varphi}_{\sup}}{k}\leq -l\lambda_1+\max\{\lambda_{i},\lambda_1 a,\lambda_1 -\lambda_{j_i}\}=-l\lambda_1+\lambda_i.\]

{\textbf{Case 3.}} Suppose that $\varphi=h_{1,1}$. Then $\varphi\in C^{0+\pag}(\sqcup_{\alpha \in \mathcal{A}} I_\alpha)$ with $a=0$ and   $\mathfrak{h}_{g+1}(\varphi)=\mathfrak{h}_{g+1}(h_{1,1})=0$. In view of Theorem~\ref{thm;correction} and Proposition~\ref{prop;FDCbound}, for every $\tau>0$ small enough,
$\|\mathcal{M}^{(k)}(S(k)\varphi)\|=O(e^{\tau k})$.
As $\varphi=h_{1,1}$ is of bounded variation, we also have $\var(S(k)\varphi)\leq \var(\varphi)$. Since $\|S(k)\varphi\|_{\sup}\leq \|\mathcal{M}^{(k)}(S(k)\varphi)\|+\var(S(k)\varphi)$, this gives
\[\limsup_{k \rightarrow \infty} \frac{\log \norm{S(k)\varphi}_{\sup}}{k}\leq 0=-\lambda_1+\lambda_1.\]

{\textbf{Case 4.}} Suppose that $\varphi=h_{1,l}$ for some  $l\geq 2$.
Then, $\varphi\in C^{l-1+\pag}(\sqcup_{\alpha \in \mathcal{A}} I_\alpha)$ with $a=0$,  $\mathfrak{h}_{-1,g+1}(D^{l-2}\varphi)=0$, $\mathfrak{h}_{g+1}(D^{l-1}\varphi)=0$,  and $\mathfrak{h}_{0}(D^{p}\varphi)=0$ for all $0\leq p<l-2$.
 In view of Theorem~\ref{thm:h-j}, this gives
\[\limsup_{k \rightarrow \infty} \frac{\log \norm{S(k)\varphi}_{\sup}}{k}\leq -(l-1)\lambda_1+\max\{\lambda_{g+1},\lambda_1 a,\lambda_1 -\lambda_{1}\}=-l\lambda_1+\lambda_1.\]

{\textbf{Case 5.}} Suppose that $\varphi=c_{s,l}$ for some $l\geq 1$. Then, $\varphi\in C^{l+\pag}(\sqcup_{\alpha \in \mathcal{A}} I_\alpha)$ with $a=0$,  $\mathfrak{h}_{-1,g+1}(D^{l-1}\varphi)=0$, $\mathfrak{h}_{g+1}(D^{l}\varphi)=0$,  and $\mathfrak{h}_{0}(D^{p}\varphi)=0$ for all $0\leq p<l-1$. Indeed,
as $D^{l}\varphi=c_{s}\in E_{g+1}$, by Remark~\ref{rem:hjhj'}, we have $\mathfrak{h}_{g+1}(D^{l}\varphi)=0$.  In view of Theorem~\ref{thm:h-j}, this gives
\[\limsup_{k \rightarrow \infty} \frac{\log \norm{S(k)\varphi}_{\sup}}{k}\leq -l\lambda_1+\max\{\lambda_{g+1},\lambda_1 a,\lambda_1 -\lambda_{1}\}=-l\lambda_1.\]

{\textbf{Final step.}}
In summary, for every $\varphi\in \Gamma_l(\sqcup_{\alpha \in \mathcal{A}} I_\alpha)$ of the form $h_{i,l}$, $c_{s,l}$ or $h_{-j,l}$, we have $D^l\varphi\in\Gamma$, and
\begin{equation}\label{neq:+lal}
\limsup_{k \rightarrow \infty} \frac{\log \!\norm{S(k)\varphi}_{\sup}}{k}\!\leq\! -l\lambda_1+\lambda(D^l\varphi)\ \text{ with }\ \lambda(D^l\varphi)\!=\!\lim_{k \rightarrow \infty} \!\frac{\log \!\norm{Q(k)D^l\varphi}}{k}.
\end{equation}
It follows that \eqref{neq:+lal} holds also for any $\varphi\in \Gamma_l(\sqcup_{\alpha \in \mathcal{A}} I_\alpha)$.
On the other hand, if additionally $D^l\varphi\neq 0$, then, by \eqref{eq:lowest}, \eqref{def;sdc4}, and \eqref{eq:invIk}, {for any $\tau > 0$ there exist constants  $C,\kappa>0$ such that}
\[\frac{1}{|I^{(k)}|}\norm{S(k)\varphi}_{L^1(I^{(k)})}\geq c_l\kappa^l|I^{(k)}|^l\norm{S(k)D^l\varphi}\geq c_l\kappa^lC^{-l}e^{-(\lambda_1+\tau)lk}\norm{Q(k)D^l\varphi}.\]
It follows that
\[\liminf_{k \rightarrow \infty} \frac{\log (\norm{S(k)\varphi}_{L^1(I^{(k)})}/|I^{(k)}|)}{k}\geq -l\lambda_1+\lambda(D^l\varphi),\]
so
\[\lim_{k \rightarrow \infty} \frac{\log \norm{S(k)\varphi}_{\sup}}{k}=\lim_{k \rightarrow \infty} \frac{\log (\norm{S(k)\varphi}_{L^1(I^{(k)})}/|I^{(k)}|)}{k}= -l\lambda_1+\lambda(D^l\varphi).\]
This completes the proof.
\end{proof}

\subsection{New functionals arising from correcting operators}\label{sec;reduction}
In this section, we develop the idea of constructing invariant distributions by decomposing correction operators with respect to base elements,
{a concept} introduced in \cite{Fr-Ul2} and \cite[\S 9.1]{Fr-Ki}.
The original idea is to decompose the operator $\mathfrak{h}_i:C^{0+\pag}(\sqcup_{\alpha \in \mathcal{A}} I_\alpha)\to U_i$ relative to the base elements $h_1,\ldots,h_{i-1}$ of $U_{i}$.
We extend this idea by taking a decomposition of the correction operators $\mathfrak{h}_{-j,i}$ and $\mathfrak{h}_{0}$.
Using an inductive procedure, we get a new family of functionals defined on $C^{n+\pa}$, which  are adjusted to define invariant distributions $\mathfrak{f}_{\bar t}$ {in  Section~\ref{sec;inv-distpag}}.
\medskip

For every  $0\leq a<1$, let $2\leq i_a\leq g+1$ and $1\leq j_a\leq g$ such that
\[
\lambda_{i_a}\leq \lambda_1 a<\lambda_{i_a-1} \text{ and } \lambda_1-\lambda_{j_a}\leq \lambda_1 a<\lambda_1-\lambda_{j_a+1}.
\]
Let us consider the bounded operators $d^+_{i,0}:C^{0+\pag}(\sqcup_{\alpha \in \mathcal{A}} I_\alpha)\to\R$, for $1\leq i<i_a$, such that
for every $\varphi\in C^{0+\pag}(\sqcup_{\alpha \in \mathcal{A}} I_\alpha)$,
\begin{equation}\label{eqn;corhj1}
\mathfrak{h}_{i_a}(\varphi) = \sum_{1\leq i<i_a}d^+_{i,0}(\varphi)h_i.
\end{equation}
Since $\mathfrak{h}_{i_a}:C^{0+\pag}(\sqcup_{\alpha \in \mathcal{A}} I_\alpha)\to U_{i_a}$ is bounded and $h_1,\ldots, h_{i_a-1}$ is a basis of  $U_{i_a}$, the functionals are well defined and bounded.

Next, let us consider the bounded operators:
\begin{align*}
d^+_{i,1}&:C^{1+\pag}(\sqcup_{\alpha \in \mathcal{A}} I_\alpha)\to\R, \text{ for } 1\leq i\leq g,\\
d^0_{s,1}&:C^{1+\pag}(\sqcup_{\alpha \in \mathcal{A}} I_\alpha)\to\R, \text{ for } 1\leq s<\gamma,\\
d^-_{-j,1}&:C^{1+\pag}(\sqcup_{\alpha \in \mathcal{A}} I_\alpha)\to\R, \text{ for } j_a< j\leq g,
\end{align*}
such that
for every $\varphi \in C^{1+\pag}(\sqcup_{\alpha \in \mathcal{A}} I_\alpha)$,
\begin{gather}\label{eqn;corhj2}
\begin{split}
\mathfrak{h}_{-j_a,i_a}&\Big(\varphi-\sum_{1\leq i<i_a}d^+_{i,0}(D\varphi)h_{i,1}\Big) \\
&= \sum_{1\leq i\leq g}d^+_{i,1}(\varphi)h_i+\sum_{1\leq s<\gamma}d^0_{s,1}(\varphi)c_s+\sum_{j_a<j\leq g}d^-_{-j,1}(\varphi)h_{-j}.
\end{split}
\end{gather}
Since  $\mathfrak{h}_{-j_a,i_a}:C^{1+\pag}(\sqcup_{\alpha \in \mathcal{A}} I_\alpha)\to U_{-j_a}$ is bounded and  $h_1,\ldots, h_{g},c_1,\ldots c_s,h_{-g},\ldots,$ $h_{-j_a+1}$ is a basis of  $U_{-j_a}$, the
functionals are well defined and bounded.

Next, let us consider the bounded operators:
\begin{align*}
d^+_{i,2}&:C^{2+\pag}(\sqcup_{\alpha \in \mathcal{A}} I_\alpha)\to\R, \text{ for }1\leq i\leq g,\\
d^0_{s,2}&:C^{2+\pag}(\sqcup_{\alpha \in \mathcal{A}} I_\alpha)\to\R, \text{ for } 1\leq s<\gamma,\\
d^-_{-j,2}&:C^{2+\pag}(\sqcup_{\alpha \in \mathcal{A}} I_\alpha)\to\R, \text{ for } 1\leq j\leq g,
\end{align*}
such that
for every $\varphi \in C^{2+\pag}(\sqcup_{\alpha \in \mathcal{A}} I_\alpha)$,
\begin{align}\label{eqn;corhj3}
\begin{split}
\mathfrak{h}_{0}&\Big(\varphi-\sum_{1\leq i<i_a}d^+_{i,0}(D^2\varphi)h_{i,2}-\sum_{1\leq i\leq g}d^+_{i,1}(D\varphi)h_{i,1}
-\sum_{1\leq s<\gamma}d^0_{s,1}(D\varphi)c_{s,1}\Big)\\
& = \sum_{1\leq i\leq g}d^+_{i,2}(\varphi)h_i+\sum_{1\leq s<\gamma}d^0_{s,2}(\varphi)c_s+\sum_{1\leq j\leq g}d^-_{-j,2}(\varphi)h_{-j}.
\end{split}
\end{align}

For any $l\geq 3$, let us consider the bounded operators:
\begin{align*}
d^+_{i,l}&: C^{l+\pag}(\sqcup_{\alpha \in \mathcal{A}} I_\alpha)\to\R,  \text{ for } 1\leq i\leq g,\\
d^0_{s,l}&:C^{l+\pag}(\sqcup_{\alpha \in \mathcal{A}} I_\alpha)\to\R, \text{ for }1\leq s<\gamma,\\
d^-_{-j,l}&:C^{l+\pag}(\sqcup_{\alpha \in \mathcal{A}} I_\alpha)\to\R, \text{ for } 1\leq j\leq g,
\end{align*}
such that
for every $\varphi \in C^{l+\pag}(\sqcup_{\alpha \in \mathcal{A}} I_\alpha)$,
\begin{align}\label{eqn;corhj4}
\begin{split}
\mathfrak{h}_{0}&\Big(\varphi-\sum_{1\leq i\leq g}d^+_{i,l-2}(D^2\varphi)h_{i,2}-\sum_{1\leq i\leq g}d^+_{i,l-1}(D\varphi)h_{i,1}-\sum_{1\leq s<\gamma}d^0_{s,l-1}(D\varphi)c_{s,1}\Big)\\
& = \sum_{1\leq i\leq g}d^+_{i,l}(\varphi)h_i+\sum_{1\leq s<\gamma}d^0_{s,l}(\varphi)c_s+\sum_{1\leq j\leq g}d^-_{-j,l}(\varphi)h_{-j}.
\end{split}
\end{align}

The following lemma is {necessary} for proving lower bounds for the growth of the cocycle $S(k)$ in the sense of $L^1$-norm.
\begin{lemma}\label{lem:estlowa}
Assume that $T$ satisfies the \ref{FDC}. Let  $0\leq a<1$ and $n\geq 0$. Then,  for every $\varphi\in \cpang(\sqcup_{\alpha \in \mathcal{A}} I_\alpha)$ with $\sum_{\alpha\in\mathcal{A}}(|C_{\alpha,n}^{a,+}(\varphi)|+|C_{\alpha,n}^{a,-}(\varphi)|)>0$, we have
\begin{equation}\label{eq:estlowa}
\liminf_{k \rightarrow \infty} \frac{\log (\norm{S(k)(\varphi)}_{L^1(I^{(k)})}/|I^{(k)}|)}{k}\geq (a-n)\lambda_1.
\end{equation}
\end{lemma}
\begin{proof}
By the proof of Theorem~1.1 (see Part V) in \cite{Fr-Ki}, if $C_\alpha^{\pm}(D^n\varphi)\neq 0$, then there exist $\vep>0$ and
a sequence of intervals $\widehat{J}^{(k)}\subset I^{(k)}_\alpha$, $k\geq 1$, such that
\begin{equation}\label{eq:IJ}
|\widehat{J}^{(k)}|\geq \frac{\vep|I^{(k)}_\alpha|}{4}\text{ and }\ |(S(k)D^n\varphi)(x)| \geq  \frac{|C^{\pm}_\alpha|}{|I^{(k)}_\alpha|^{a}}
\text{ for all }x\in \widehat{J}^{(k)} \text{ and }k\geq 1.
\end{equation}
An elementary argument shows that if $f:I\to\R$ is a $C^1$-function such that $|Df(x)|\geq a>0$ for all $x\in I$, then there exists
a subinterval $J\subset I$ such that $|J|\geq |I|/4$ and $|f(x)|\geq a|I|/4$ {for any $x\in J$} (see \cite[Lemma 4.7]{Fr-Ki}). It follows that, for every $n\geq 1$, if $f:I\to\R$ is a $C^n$-function such that $|D^nf(x)|\geq a>0$ for all $x\in I$, then there exists
a subinterval $J\subset I$ such that $|J|\geq |I|/4^n$ and $|f(x)|\geq a|I|^n/4^{n(n+1)/2}$.

In view of \eqref{eq:IJ}, it follows that there exists
a sequence of intervals ${J}^{(k)}\subset\widehat{J}^{(k)}\subset I^{(k)}_\alpha$, $k\geq 1$, such that
\[
|{J}^{(k)}|\geq \frac{\vep|I^{(k)}_\alpha|}{4^{n+1}}\text{ and }\ |(S(k)\varphi)(x)| \geq  \vep^n\frac{|C^{\pm}_\alpha|}{|I^{(k)}_\alpha|^{a}}\frac{|I^{(k)}_\alpha|^{n}}{4^{n(n+3)/2}}
\text{ for all }x\in {J}^{(k)} \text{ and }k\geq 1.
\]
Therefore,
\[\frac{1}{|I^{(k)}|}\norm{S(k)(\varphi)}_{L^1(I^{(k)})}\geq \vep^{n+1}\frac{1}{|I^{(k)}|}\frac{|C^{\pm}_\alpha|}{|I^{(k)}_\alpha|^{a}}\frac{|I^{(k)}_\alpha|^{n+1}}{4^{(n+1)^2}}.\]
By \eqref{def;sdc4} and \eqref{eq:invIk}, {for any $\tau > 0$ there exist constants  $C,\kappa>0$ such that}
\[\frac{|I^{(k)}_\alpha|^{n+1-a}}{|I^{(k)}|}\geq \kappa^{n+1-a}|I^{(k)}|^{n-a}\geq \kappa^{n+1-a}C^{a-n}e^{-(\lambda_1+\tau)(n-a)k}.\]
This gives \eqref{eq:estlowa}.
\end{proof}

In the following theorem, we prove the first version of the spectral result for the cocycle $S(k)$ on $\cpang$.
Each map $\varphi\in \cpang(\sqcup_{\alpha \in \mathcal{A}} I_\alpha)$ is decomposed with respect to the base elements $h_{i,l}$, $c_{s,l}$, $h_{-j,l}$ {in $\Gamma_l(\sqcup_{\alpha \in \mathcal{A}} I_\alpha)$} with weights determined by derivatives of the functionals {$d^+_{i,l}, d^0_{s,l}, d^-_{-j,l}$} defined at the beginning of {this} subsection.  Theorem~\ref{thm:h-j} is again  the main tool of the proof.

\begin{theorem}\label{thm:redPa}
Assume that $T$ satisfies the \ref{FDC}.
For any $0\leq a<1$ and $n\geq 1$, there exists a bounded operator $\mathfrak{r}_{a,n}:\cpang(\sqcup_{\alpha \in \mathcal{A}} I_\alpha)\to \cpang(\sqcup_{\alpha \in \mathcal{A}} I_\alpha)$ such that for every $\varphi\in \cpang(\sqcup_{\alpha \in \mathcal{A}} I_\alpha)$,
\begin{align*}
\begin{split}
&\varphi=\mathfrak{r}_{a,n}(\varphi)+\sum_{1\leq i<i_a}d^+_{i,0}(D^n\varphi)h_{i,n}\\
&+\sum_{1\leq i\leq g}d^+_{i,1}(D^{n-1}\varphi)h_{i,n-1}+\sum_{1\leq s<\gamma}d^0_{s,1}(D^{n-1}\varphi)c_{s,n-1}+\sum_{j_a< j\leq g}d^-_{-j,1}(D^{n-1}\varphi)h_{-j,n-1}\\
&+\!\sum_{2\leq l\leq n}\!\Big(\!\sum_{1\leq i\leq g}d^+_{i,l}(D^{n-l}\varphi)h_{i,n-l}+\!\sum_{1\leq s<\gamma}d^0_{s,l}(D^{n-l}\varphi)c_{s,n-l}+\!\sum_{1\leq j\leq g}d^-_{-j,l}(D^{n-l}\varphi)h_{-j,n-l}\Big),
\end{split}
\end{align*}
and, for any $\tau>0$,
\begin{equation}\label{eq:ran<}
 \norm{S(k)\mathfrak{r}_{a,n}(\varphi)}_{\sup}= O(e^{\lambda_1(a-n+\tau)k})\|\mathfrak{r}_{a,n}(\varphi)\|_{C^{n+\pa}}.
\end{equation}
If additionally $\sum_{\alpha\in\mathcal{A}}(|C_{\alpha,n}^{a,+}(\varphi)|+|C_{\alpha,n}^{a,-}(\varphi)|)>0$, then
\begin{equation}\label{eq:ran=}
\lim_{k \rightarrow \infty} \frac{\log \norm{S(k)\mathfrak{r}_{a,n}(\varphi)}_{\sup}}{k}=\lim_{k \rightarrow \infty} \frac{\log \frac{\norm{S(k)\mathfrak{r}_{a,n}(\varphi)}_{L^1(I^{(k)})}}{|I^{(k)}|}}{k}= (a-n)\lambda_1.
\end{equation}
\end{theorem}

\begin{proof}
In view of \eqref{eq:prophch1}
, for every $0\leq m\leq n-1$,
\begin{align*}
&D^m\mathfrak{r}_{a,n}(\varphi)=D^m\varphi-\sum_{1\leq i<i_a}d^+_{i,0}(D^n\varphi)h_{i,n-m}\\
&-\sum_{1\leq i\leq g}d^+_{i,1}(D^{n-1}\varphi)h_{i,n-1-m}-\sum_{1\leq s<\gamma}d^0_{s,1}(D^{n-1}\varphi)c_{s,n-1-m}-\sum_{j_a< j\leq g}d^-_{-j,1}(D^{n-1}\varphi)h_{-j,n-1-m}\\
&-\!\sum_{2\leq l\leq n-m}\!\Big(\!\sum_{1\leq i\leq g}\!d^+_{i,l}(D^{n-l}\varphi)h_{i,n-l-m}\!-\!\sum_{1\leq s<\gamma}\!d^0_{s,l}(D^{n-l}\varphi)c_{s,n-l-m}\!-\!\sum_{1\leq j\leq g}\!d^-_{-j,l}(D^{n-l}\varphi)h_{-j,n-l-m}\!\Big).
\end{align*}
Suppose that $0\leq m\leq n-3$.
Since $\mathfrak{h}_{0}(h_{i,l})=0$ for $l\geq 3$  (see \eqref{eq:prophch3}), $\mathfrak{h}_{0}(c_{s,l})=0$ for $l\geq 2$ (see \eqref{eq:prophch4}), $\mathfrak{h}_{0}(h_{-j,l})=0$ for $l\geq 1$ (see \eqref{eq:prophch5}), and $\mathfrak{h}_{0}(h)=h$ for $h\in\Gamma$ (see Remark~\ref{rem:h0}), it follows that
\begin{align*}
\mathfrak{h}_{0}&(D^m\mathfrak{r}_{a,n}(\varphi))=\mathfrak{h}_{0}\Big(D^m\varphi-\sum_{1\leq i\leq g}d^+_{i,n-m-2}(D^{m+2}\varphi)h_{i,2}\\
&-\sum_{1\leq i\leq g}d^+_{i,n-m-1}(D^{m+1}\varphi)h_{i,1}-\sum_{1\leq s<\gamma}\!d^0_{s,n-m-1}(D^{m+1}\varphi)c_{s,1}\Big)\\
&-\sum_{1\leq i\leq g}d^+_{i,n-m}(D^{m}\varphi)h_{i}-\sum_{1\leq s<\gamma}d^0_{s,n-m}(D^{m}\varphi)c_{s}-\sum_{1\leq j\leq g}d^-_{-j,n-m}(D^{m}\varphi)h_{-j}.
\end{align*}
In view of \eqref{eqn;corhj4}, this gives $\mathfrak{h}_{0}(D^m\mathfrak{r}_{a,n}(\varphi))=0$.
The same arguments show that
\begin{align*}
\mathfrak{h}_{0}(D^{n-2}&\mathfrak{r}_{a,n}(\varphi))=\mathfrak{h}_{0}\Big(D^{n-2}\varphi-\sum_{1\leq i<i_a}d^+_{i,0}(D^n\varphi)h_{i,2}\\&-\sum_{1\leq i\leq g}d^+_{i,1}(D^{n-1}\varphi)h_{i,1}-\sum_{1\leq s<\gamma}d^0_{s,1}(D^{n-1}\varphi)c_{s,1}\Big)\\
&-\sum_{1\leq i\leq g}d^+_{i,2}(D^{n-2}\varphi)h_{i}-\sum_{1\leq s<\gamma}d^0_{s,2}(D^{n-2}\varphi)c_{s}-\sum_{1\leq j\leq g}d^-_{-j,2}(D^{n-2}\varphi)h_{-j}.
\end{align*}
In view of \eqref{eqn;corhj3}, this gives $\mathfrak{h}_{0}(D^{n-2}\mathfrak{r}_{a,n}(\varphi))=0$.

Next, we pass to the $(n-1)$-th derivative,
\begin{align*}
&D^{n-1}\mathfrak{r}_{a,n}(\varphi)=D^{n-1}\varphi-\sum_{1\leq i<i_a}d^+_{i,0}(D^n\varphi)h_{i,1}\\
&-\sum_{1\leq i\leq g}d^+_{i,1}(D^{n-1}\varphi)h_{i}-\sum_{1\leq s<\gamma}d^0_{s,1}(D^{n-1}\varphi)c_{s}-\sum_{j_a< j\leq g}d^-_{-j,1}(D^{n-1}\varphi)h_{-j}.
\end{align*}
In view of \eqref{eqn;corhj2}, this gives $\mathfrak{h}_{-j_a,i_a}(D^{n-1}\mathfrak{r}_{a,n}(\varphi))=0$.

Finally, we pass to the $n$-th derivative,
\begin{align*}
D^{n}\mathfrak{r}_{a,n}(\varphi)=D^{n}\varphi-\sum_{1\leq i<i_a}d^+_{i,0}(D^n\varphi)h_{i}.
\end{align*}
In view of \eqref{eqn;corhj1}, this gives $\mathfrak{h}_{i_a}(D^{n}\mathfrak{r}_{a,n}(\varphi))=0$.

Since $\max\{\lambda_{i_a},a\lambda_1,\lambda_1-\lambda_{j_a}\}=a\lambda_1$, by Theorem~\ref{thm:h-j}, for any $\tau>0$, we have
\[ \norm{S(k)\mathfrak{r}_{a,n}(\varphi)}_{\sup}=O(e^{\lambda_1(a-n+\tau)k})\|\mathfrak{r}_{a,n}(\varphi)\|_{C^{n+\pa}}.
\]
The final lower bound in \eqref{eq:ran=} follows directly from Lemma~\ref{lem:estlowa}.
\end{proof}

\begin{remark}\label{rmk:n=0}
Theorem~\ref{thm:redPa} remains true also in the case when $n=0$, except that in formulas \eqref{eq:ran<} and \eqref {eq:ran=} we must replace the sup norm by the $L^1$-norm. Here, $\mathfrak{r}_{a,0}:C^{0+\pag}(\sqcup_{\alpha \in \mathcal{A}} I_\alpha)\to C^{0+\pag}(\sqcup_{\alpha \in \mathcal{A}} I_\alpha)$ is given by
\[\mathfrak{r}_{a,0}(\varphi)=\varphi-\sum_{1\leq i<i_a}d^+_{i,0}(\varphi)h_{i},\]
so $\mathfrak{h}_{i_a}(\mathfrak{r}_{a,0}(\varphi))=0$ for every $\varphi\in C^{0+\pag}(\sqcup_{\alpha \in \mathcal{A}} I_\alpha)$.
By Theorem~\ref{thm;correction} and Proposition~\ref{prop;FDCbound}, for any $\tau>0$,
\[\|\mathcal{M}^{(k)}(S(k)(\mathfrak{r}_{a,0}(\varphi)))\|=O(e^{(a\lambda_1+\tau)k})\|\mathfrak{r}_{a,0}(\varphi)\|_{C^{0+\pa}}.\]
In view of \eqref{eqn;upperboundvarphi} and \eqref{eqn;renormpaos2}, it follows that
\[\|S(k)(\mathfrak{r}_{a,0}(\varphi))\|_{L^1(I^{(k)})}/|I^{(k)}|=O(e^{(a\lambda_1+\tau)k})\|\mathfrak{r}_{a,0}(\varphi)\|_{C^{0+\pa}}.\]
The lower bound follows again directly from Lemma~\ref{lem:estlowa}.
\end{remark}

\subsection{Invariant distributions on $C^{n+\pag}(\sqcup_{\alpha \in \mathcal{A}} I_\alpha)$}\label{sec;inv-distpag}
For every $0\leq a<1$ and $n\geq 0$, we denote by $\mathscr{T}^*_{a,n}$ (or, resp.\ $\mathscr{T}_{a,n}$)  the subset of triples $\bar{t}\in\mathscr{TF}^*$ (or, resp.\ $\mathscr{TF}$) of the form $(l,+,i)$, $(l,0,s)$ or $(l,-,j)$, for $0\leq l\leq n$,
with the additional {restrictions:}
\begin{itemize}
\item if $l=n$, then we deal only with $(n,+,i)$ for $1\leq i<i_a$;
\item if $l=n-1$, then we deal only with $(n-1,+,i)$ for all $1\leq i\leq g$, $(n-1,0,s)$ for all $1\leq s<\gamma$ and $(n-1,-,j)$ for $j_a< j\leq g$.
\end{itemize}
Recall that $\mathscr{TF}$ is the subset of triples in $\mathscr{TF}^*$ obtained after removing all triples of the form $(l,-,1)$.

\begin{remark}
By definition,
\[\bar{t}\in \mathscr{T}^*_{a,n}\Longleftrightarrow \mathfrak{o}(\bar{t})\leq (n-\tfrac{\lambda_{i_a-1}}{\lambda_1})\vee (n-1+\tfrac{\lambda_{j_a+1}}{\lambda_1}).\]
As $\lambda_{i_a}\leq\lambda_1 a<\lambda_{i_a-1}$ and $\lambda_{j_a+1}<\lambda_1(1-a)\leq \lambda_{j_a}$, it follows that
\begin{equation}\label{eq:maxex}
\bar{t}\in \mathscr{T}^*_{a,n}\Longleftrightarrow \mathfrak{o}(\bar{t})<n-a.
\end{equation}
\end{remark}

\begin{definition}\label{def;Dan}
For every $\bar{t}\in\mathscr{T}^*_{a,n}$, let $\mathfrak{f}_{\bar{t}}:C^{n+\pag}(\sqcup_{\alpha \in \mathcal{A}} I_\alpha)\to\C$  and  $h_{\bar{t}}\in \Gamma_n(\sqcup_{\alpha \in \mathcal{A}} I_\alpha)$ be defined as follows:
\begin{itemize}
\item  $\mathfrak f_{\bar{t}}=d^{+}_{i,n-l}\circ D^{l}$  and  $h_{\bar{t}}:=h_{i,l}$ if $\bar{t}=(l,+,i)$;
\item  $\mathfrak f_{\bar{t}}=d^{0}_{s,n-l}\circ D^{l}$  and  $h_{\bar{t}}:=c_{s,l}$ if $\bar{t}=(l,0,s)$;
\item  $\mathfrak f_{\bar{t}}=d^{-}_{-j,n-l}\circ D^{l}$  and  $h_{\bar{t}}:=h_{-j,l}$ if $\bar{t}=(l,-,j)$.
\end{itemize}
\end{definition}

\begin{theorem}\label{thm;spdecomp}
Assume that $T$ satisfies the \ref{FDC}. Let $0\leq a<1$ and $n\geq 0$. Then,  every $\varphi\in \cpang(\sqcup_{\alpha \in \mathcal{A}} I_\alpha)$ is decomposed as follows:
\begin{equation}\label{eq:decom}
\varphi=\sum_{\bar{t}\in \mathscr{T}^*_{a,n}}\mathfrak{f}_{\bar{t}}(\varphi)h_{\bar t}+\mathfrak{r}_{a,n}(\varphi),
\end{equation}
so that for any $\tau>0$ and for all $0\leq l<n$,
\begin{gather}
\label{eq:skdl}
\|S(k)(D^l\mathfrak{r}_{a,n}(\varphi))\|_{\sup}=O(e^{(-\lambda_1(n-l-a)+\tau)k})\|D^l\mathfrak{r}_{a,n}(\varphi)\|_{C^{n-l+\pa}},\\
\label{eq:skdn}
\|S(k)(D^n\mathfrak{r}_{a,n}(\varphi))\|_{L^1(I^{(k)})}/|I^{(k)}|=O(e^{(\lambda_1a+\tau)k})\|D^n\mathfrak{r}_{a,n}(\varphi)\|_{C^{0+\pa}},
\text{ and }\\
\label{eq:skhd}
\lim_{k\to\infty}\frac{1}{k}\log\Big\|S(k)\sum_{\bar{t}\in \mathscr{T}^*_{a,n}}a_{\bar t}h_{\bar t}\Big\|_{\sup}=-\lambda_1\min\{\mathfrak{o}(\bar{t}):\bar t\in \mathscr{T}^*_{a,n}, a_{\bar t}\neq 0\}.
\end{gather}
If additionally $\sum_{\alpha\in\mathcal{A}}(|C_{\alpha,n}^{a,+}(\varphi)|+|C_{\alpha,n}^{a,-}(\varphi)|)>0$, then
\begin{gather}
\label{eq:skdl1}
\lim_{k\to\infty}\frac{1}{k}\log\|S(k)(D^l\mathfrak{r}_{a,n}(\varphi))\|_{\sup}=-\lambda_1(n-l-a),\text{ for  }0\leq l<n,\text{ and }\\
\lim_{k\to\infty}\frac{1}{k}\log\big(\|S(k)(D^l\mathfrak{r}_{a,n}(\varphi))\|_{L^1(I^{(k)})}/|I^{(k)}|\big)=-\lambda_1(n-l-a),\text{ for }0\leq l\leq n.
\end{gather}
Moreover, for each $\bar t\in \mathscr{T}_{a,n}$ the functional $\mathfrak{f}_{\bar t}: \cpang(\sqcup_{\alpha \in \mathcal{A}} I_\alpha)\to \C$  is invariant, {i.e.,} for every $\varphi\in\cpang(\sqcup_{\alpha \in \mathcal{A}} I_\alpha)$ such that  $\varphi=v\circ T-v$ for some $v\in C^{r}(I)$ with $\mathfrak{o}(\bar t)<r\leq n-a$, we have $\mathfrak{f}_{\bar t}(\varphi)=0$.
Also, the functionals $C^{a,\pm}_{\alpha,n}: \cpang(\sqcup_{\alpha \in \mathcal{A}} I_\alpha)\to \C$ are invariant, {i.e.,} if
 $\varphi=v\circ T-v$ for some $v\in C^{r}(I)$ with $n-a<r$, then $C^{a,\pm}_{\alpha,n}(\varphi)=0$ for every $\alpha\in \mathcal{A}$.
\end{theorem}

\begin{proof}
All claims of the theorem, aside from invariance, are derived directly from Proposition~\ref{lem;Qkh}, Theorem~\ref{thm:redPa}, and Remark~\ref{rmk:n=0}, so we focus on invariance only.
\medskip

Suppose that $\varphi=v\circ T-v$ for some $v\in C^{r}(I)$ with $r\leq n-a$. Let $r=m+b$ with an integer $0\leq m<n$ and $0<b\leq 1$. By \eqref{eq:decom}, for every $0\leq j\leq n$,
\begin{equation}\label{eq:Djdiff}
D^j(\varphi-\mathfrak{r}_{a,n}(\varphi))=\sum_{\bar t\in \mathscr{T}^*_{a,n}}\mathfrak{f}_{\bar t}(\varphi)D^jh_{\bar t}.
\end{equation}
Then, for every $0\leq j\leq m$, we have $D^jv\in C^{m-j+b}(I)$, and for every $x\in I^{(k)}_\alpha$,
\begin{align*}
|S(k)D^j\varphi(x)|&=|D^jv(T^{Q_\alpha(k)}(x))-D^jv(x)|
\leq
\left\{\begin{array}{ll}
\|D^{j}v\|_{C^1}|I^{(k)}|& \text{if }0\leq j<m,\\
\|D^{m}v\|_{C^b}|I^{(k)}|^b& \text{if } j=m.
\end{array}\right.
\end{align*}
This also gives
\[\Big|\int_{I^{(k)}_\alpha}S(k)D^{m+1}\varphi(x)\,dx\Big|=|S(k)D^{m}\varphi(r^{(k)}_\alpha)-S(k)D^{m}\varphi(l^{(k)}_\alpha)|\leq 2\|v\|_{C^{m+b}}|I^{(k)}|^b.\]
As $|I^{(k)}|=O(e^{-\lambda_1 k})$, $|I^{(k)}_\alpha|^{-1}=O(|I^{(k)}|^{-1})$, and $|I^{(k)}|^{-1}=O(e^{(\lambda_1+\tau) k})$ for every $\tau>0$, we obtain
\begin{gather*}
\limsup_{k\to\infty}\frac{1}{k}\log\|S(k)D^j\varphi\|_{\sup}\leq -\lambda_1\text{ if }j<m;\\
\limsup_{k\to\infty}\frac{1}{k}\log\|S(k)D^m\varphi\|_{\sup}\leq -b\lambda_1;\\
\limsup_{k\to\infty}\frac{1}{k}\log\|\mathcal{M}^{(k)}(S(k)D^{m+1}\varphi)\|\leq (1-b)\lambda_1.
\end{gather*}
As $m<n$, in view of \eqref{eq:skdl} and \eqref{eq:skdn}, it follows that
\begin{gather}
\label{eq:Djdiff1}
\limsup_{k\to\infty}\frac{1}{k}\log\|S(k)(D^j(\varphi-\mathfrak{r}_{a,n}(\varphi)))\|_{\sup}\leq -\lambda_1\text{ if }0\leq j<m;\\
\label{eq:Djdiff2}
\limsup_{k\to\infty}\frac{1}{k}\log\|S(k)(D^{m}(\varphi-\mathfrak{r}_{a,n}(\varphi)))\|_{\sup}\leq -b\lambda_1;\\
\label{eq:Djdiff3}
\limsup_{k\to\infty}\frac{1}{k}\log\|\mathcal{M}^{(k)}(S(k)(D^{m+1}(\varphi-\mathfrak{r}_{a,n}(\varphi))))\|
\leq (1-b)\lambda_1.
\end{gather}
In view of \eqref{eq:Djdiff}, $\widetilde{\varphi}=D^{m+1}(\varphi-\mathfrak{r}_{a,n}(\varphi))\in \Gamma_{n-m-1}(\sqcup_{\alpha \in \mathcal{A}} I_\alpha)$. Therefore, by \eqref{eqn:normbv} and \eqref{neq:Skvar},
\[\|S(k)\widetilde{\varphi}\|_{\sup}\leq \|\mathcal{M}^{(k)}(S(k)\widetilde{\varphi})\|+\var(S(k)\widetilde{\varphi})\leq
\|\mathcal{M}^{(k)}(S(k)\widetilde{\varphi})\|+\var{\widetilde{\varphi}}.\]
In view of \eqref{eq:Djdiff3}, this gives
\begin{equation}\label{eq:Djdiff4}
\limsup_{k\to\infty}\frac{1}{k}\log\|S(k)(D^{m+1}(\varphi-\mathfrak{r}_{a,n}(\varphi)))\|_{\sup}
\leq (1-b)\lambda_1.
\end{equation}
On the other hand, by \eqref{eq:Djdiff} and \eqref{eq:skcomb},
\begin{gather*}
\lim_{k\to\infty}\frac{1}{k}\log\|S(k)(D^j(\varphi-\mathfrak{r}_{a,n}(\varphi)))\|_{\sup}=
\lim_{k\to\infty}\frac{1}{k}\log\|S(k)\!\sum_{\bar t\in \mathscr{T}^*_{a,n}}\mathfrak{f}_{\bar t}(\varphi)D^jh_{\bar t}\|_{\sup}\\
=
\lambda_1\max\left\{-\mathfrak{o}(\bar t)+j: \bar t\in \mathscr{T}^*_{a,n}, \mathfrak{f}_{\bar t}(\varphi)\neq 0, D^jh_{\bar t}\neq 0\right\}.
\end{gather*}
In view of  \eqref{eq:Djdiff1}, \eqref{eq:Djdiff2}, and \eqref{eq:Djdiff4}, this yields
\begin{align}
\label{eq:minl}
&\min\left\{\mathfrak{o}(\bar t): \bar t\in \mathscr{T}^*_{a,n}, \mathfrak{f}_{\bar t}(\varphi)\neq 0, D^lh_{\bar t}\neq 0\right\}\geq l+1\text{ if }0\leq l<m,\\
\label{eq:minm}
&\min\left\{\mathfrak{o}(\bar t): \bar t\in \mathscr{T}^*_{a,n}, \mathfrak{f}_{\bar t}(\varphi)\neq 0, D^{m}h_{\bar t}\neq 0\right\}\geq m+b,\\
\label{eq:minm+1}
&\min\left\{\mathfrak{o}(\bar t): \bar t\in \mathscr{T}^*_{a,n}, \mathfrak{f}_{\bar t}(\varphi)\neq 0, D^{m+1}h_{\bar t}\neq 0\right\}\geq m+b.
\end{align}
Let $\bar t\in \mathscr{T}_{a,n}$ be any triple such that $\mathfrak{o}(\bar t)<r=m+b$. By definition, $\mathfrak{f}_{\bar t}$, $h_{\bar t}$, and $\mathfrak{o}(\bar t)$, are of the form:
\begin{align*}
& \mathfrak{f}_{\bar t}=d^{+}_{i,n-l}\circ D^{l},\quad h_{\bar t}=h_{i,l},\quad \text{and}\quad\mathfrak{o}(\bar t)=l-\frac{\lambda_i}{\lambda_1},\quad\text{or}\\
& \mathfrak{f}_{\bar t}=d^{0}_{s,n-l}\circ D^{l},\quad h_{\bar t}=c_{s,l},\quad\text{and}\quad\mathfrak{o}(\bar t)=l, \quad\text{or}\\
& \mathfrak{f}_{\bar t}=d^{-}_{-j,n-l}\circ D^{l}, \quad h_{\bar t}=h_{-j,l},\quad \text{and}\quad\mathfrak{o}(\bar t)=l+\frac{\lambda_j}{\lambda_1}\text{ with }j\neq 1,
\end{align*}
for $0\leq l\leq m+1$. If $0\leq l<m$, then $D^lh_{\bar t}\neq 0$ and $\mathfrak{o}(\bar t)\leq l+\lambda_{2}/\lambda_1<l+1$. Then, by \eqref{eq:minl}, $\mathfrak{f}_{\bar t}(\varphi)=0$.
If $l=m$ or $m+1$, then $D^{l}h_{\bar t}\neq 0$ and $\mathfrak{o}(\bar t)<m+b$. Then, by \eqref{eq:minm} and \eqref{eq:minm+1}, $\mathfrak{f}_{\bar t}(\varphi)=0$ as well. This completes the proof of invariance for the functionals $\mathfrak{f}_{\bar t}$, $\bar t\in \mathscr{T}_{a,n}$.

\medskip

Suppose that $\varphi=v\circ T-v$ for some $v\in C^{r}(I)$ with $r> n-a$.

Assume that $0<a<1$.
Then, $D^{n-1}\varphi=D^{n-1}v\circ T-D^{n-1}v$ with $D^{n-1}v\in C^{1-a+\tau}(I)$, where $0<\tau<(r-n+a)\wedge a$. Therefore, $D^{n-1}\varphi$ is $(1-a+\tau)$-H\"older on any interval $I_{\alpha}$, $\alpha\in\mathcal{A}$. Suppose, contrary to our claim, that $C_\alpha^+(D^n\varphi)=C_{\alpha,n}^{a,+}(\varphi)\neq 0$.
Then, there exists $\vep>0$ such that
\[0<c:=|C_\alpha^+(D^n\varphi)|/2\leq |D^{n+1}\varphi(x)||x-l_\alpha|^{1+a}\text{ for }x\in (l_\alpha,l_\alpha+\vep].\]
Hence, for every $x\in (l_\alpha,l_\alpha+\vep]$,
\begin{align*}
\Big|\frac{c}{a(x-l_\alpha)^a}-\frac{c}{a\vep^a}\Big|&=\int_x^{l_\alpha+\vep}\frac{c}{(s-l_\alpha)^{1+a}}ds\leq \Big|\int_x^{l_\alpha+\vep}D^{n+1}\varphi(s)ds\Big|\\
&\leq |D^{n}\varphi(x)-D^{n}\varphi(l_\alpha+\vep)|.
\end{align*}
It follows that there exists $0<\delta<\vep$ such that
\[\frac{c}{2a(x-l_\alpha)^a}\leq |D^{n}\varphi(x)|\text{ for }x\in (l_\alpha,l_\alpha+\delta].\]
Hence, for every $x,y\in (l_\alpha,l_\alpha+\delta]$,
\begin{align*}
\frac{c}{2a(1-a)}&|(y-l_\alpha)^{1-a}-(x-l_\alpha)^{1-a}|=\int_x^{y}\frac{c}{2a(s-l_\alpha)^a}ds\leq \Big|\int_x^{y}D^{n}\varphi(s)ds\Big|\\
&\leq |D^{n-1}\varphi(x)-D^{n-1}\varphi(y)|\leq \|D^{n-1}\varphi\|_{C^{1-a+\tau}}|(y-l_\alpha)-(x-l_\alpha)|^{1-a+\tau}.
\end{align*}
It follows that $c\leq 2a(1-a)\|D^{n-1}\varphi\|_{C^{1-a+\tau}} s^\tau$  for every $s\in(0,\delta]$, contrary to $|C_\alpha^+(D^n\varphi)|=2c>0$.
This gives $C_{\alpha,n}^{a,+}(\varphi)=C_\alpha^+(D^n\varphi)=0$, and the same arguments also show that $C_{\alpha,n}^{a,-}(\varphi)=C_\alpha^-(D^n\varphi)=0$.

If $a=0$, then the proof proceeds in the same way.
In this case, $D^{n}\varphi=D^{n}v\circ T-D^{n}v$ with $D^{n}v\in C^{\tau}(I)$, where $0<\tau<(r-n)\wedge 1$. Therefore, $D^{n}\varphi$ is $\tau$-H\"older on any  $I_{\alpha}$, $\alpha\in\mathcal{A}$. Suppose that $C_{\alpha,n}^{a,+}(\varphi)\neq 0$.
As in the previous case, there exists $\vep>0$ such that
\[0<c:=|C_\alpha^+(D^n\varphi)|/2\leq |D^{n+1}\varphi(x)||x-l_\alpha|\text{ for }x\in (l_\alpha,l_\alpha+\vep].\]
Hence, for every $x,y\in (l_\alpha,l_\alpha+\vep]$,
\begin{align*}
c&|\log(y-l_\alpha)-\log(x-l_\alpha)|=\int_x^{y}\frac{c}{s-l_\alpha}ds\leq \Big|\int_x^{y}D^{n+1}\varphi(s)ds\Big|\\
&\leq |D^{n}\varphi(x)-D^{n}\varphi(y)|\leq \|D^{n}\varphi\|_{C^{\tau}}|(y-l_\alpha)-(x-l_\alpha)|^{\tau}.
\end{align*}
It follows that $c\log 2\leq \|D^{n}\varphi\|_{C^{\tau}} s^\tau$  for every $s\in(0,\vep/2]$, contrary to $|C_\alpha^+(D^n\varphi)|=2c>0$.
 This completes the proof.
\end{proof}

\begin{lemma}\label{lem:unidecom}
The decomposition \eqref{eq:decom} is unique, {i.e.,} if
\[\varphi=\sum_{\bar t\in \mathscr{T}^*_{a,n}}a_{\bar t}h_{\bar t}+\widetilde{\varphi}\quad
\text{with}
\quad\limsup_{k\to\infty}\frac{1}{k}\log\|S(k)\widetilde{\varphi}\|_{\sup}\leq -\lambda_1(n-a),\]
then $a_{\bar t}=\mathfrak{f}_{\bar t}(\varphi)$ for every $\bar t\in \mathscr{T}^*_{a,n}$. In particular,
$\mathfrak{f}_{\bar t}(\mathfrak{r}_{a,n}(\varphi))=0$ for every $\bar t\in \mathscr{T}^*_{a,n}$.
\end{lemma}
\begin{proof}
By assumption,
$\widetilde{\varphi}-\mathfrak{r}_{a,n}(\varphi)=\sum_{\bar t\in \mathscr{T}^*_{a,n}}(\mathfrak{f}_{\bar t}(\varphi)-a_{\bar t})h_{\bar t}$
and
\[\limsup_{k\to\infty}\frac{1}{k}\log\|S(k)(\widetilde{\varphi}-\mathfrak{r}_{a,n}(\varphi))\|_{\sup}\leq -\lambda_1(n-a).\]
On the other hand, by \eqref{eq:skhd},
\[
\lim_{k\to\infty}\frac{1}{k}\log\Big\|S(k)\sum_{\bar t\in \mathscr{T}^*_{a,n}}(\mathfrak{f}_{\bar t}(\varphi)-a_{\bar t})h_{\bar t}\Big\|_{\sup}=-\lambda_1\min\{\mathfrak{o}(\bar t):\bar t\in \mathscr{T}^*_{a,n}, \mathfrak{f}_{\bar t}(\varphi)\neq a_{\bar t}\}.\]
In view of \eqref{eq:maxex}, both give $a_{\bar t}=\mathfrak{f}_{\bar t}(\varphi)$ for every $\bar t\in \mathscr{T}^*_{a,n}$.
\end{proof}

\begin{remark}\label{rmk:uniquedef}
Let us consider two pairs $(n_1,a_1)$, $(n_2,a_2)$ such that $n_1-a_1<n_2-a_2$. Then $C^{n_2+\mathrm{P}_{a_2}}(\sqcup_{\alpha \in \mathcal{A}} I_\alpha)\subset C^{n_1+\mathrm{P}_{a_1}}(\sqcup_{\alpha \in \mathcal{A}} I_\alpha)$ and  $\mathscr{T}^*_{a_1,n_1}\subset \mathscr{T}^*_{a_2,n_2}$. Suppose that $\bar{t}_1\in\mathscr{T}^*_{a_1,n_1}$, $\bar{t}_2\in \mathscr{T}^*_{a_2,n_2}$ are such that $\bar{t}_1=\bar{t}_2$. By Lemma~\ref{lem:unidecom},  $\mathfrak{f}_{\bar{t}_1}:C^{n_1+\mathrm{P}_{a_1}}(\sqcup_{\alpha \in \mathcal{A}} I_\alpha)\to\C$ is an extension of $\mathfrak{f}_{\bar{t}_2}:C^{n_2+\mathrm{P}_{a_2}}(\sqcup_{\alpha \in \mathcal{A}} I_\alpha)\to\C$.
\end{remark}

\section{Solving cohomological equations on IET}\label{sec;coh-iet}
Given $\varphi\in C^{n+\pag}(\sqcup_{\alpha \in \mathcal{A}} I_\alpha)$, we provide a smooth solution $v$ (whose {$(n-1)$-st derivative} is H\"older continuous) to the cohomological equation $v\circ T-v=\varphi$ for the IET $T$, provided that the sequence $S(k)\varphi$ decays {exponentially} fast enough. By combining this with the spectral result (Theorem~\ref{thm;spdecomp}), we get {the} regularity of the solutions, {which depends} on the vanishing of the invariant distributions $\mathfrak{f}_{\bar t}$. The main regularity estimates are obtained through orbit decompositions and space decompositions invented by Marmi-Moussa-Yoccoz \cite[\S 2.2.3]{Ma-Mo-Yo} and \cite[\S 3.7-8]{Ma-Yo}.

\subsubsection*{Time decomposition }
 \begin{itemize}
\item  Let $T$ {be} an IET satisfying Keane's condition, {$x\in I$}, and $N\geq 1$. Let $y$ be the point of the orbit $(T^jx)_{0 \leq j <N}$ which is closest to $0$.
\item We split the orbit into positive/negative parts $(T^jy)_{0 \leq j <N^+}$ and $(T^jy)_{N^- \leq j <0}$, where $N = N^+ -N^-$.
\item  Let $k\geq0$ be the largest number such that at least one element of $(T^jy)_{0<j< N^+}$ belongs to $I^{(k)}$.
\item Let $y,T^{(k)}y, \ldots, {(T^{(k)})}^{q(k)}y$ be all points of $(T^jy)_{0\leq j<N^+}$ that belong to $I^{(k)}$ for some $q(k)>0$. Let $y(k):=y$.
\item We define $y(l),q(l)$ inductively backward for $0\leq l < k$. Let $y(k-1) = (T^{(k)})^{q(k)}(y)$ and let $y(l) = T^{N(l)}(y)$ be the last point of the orbit $(T^jy)_{0 \leq j <N}$ which belongs to $I^{(l+1)}$.
Let $y(l),T^{(l)}(y(l)), \ldots, {(T^{(l)})}^{q(l)}(y(l)):= y(l-1)$ be all points of $(T^jy)_{N(l)\leq j<N^+}$ that belong to $I^{(l)}$ for some $q(l)\geq 0$.
  \end{itemize}
Then,
\begin{equation}\label{def;tpd}
\sum_{0 \leq j < N^+}\varphi(T^iy) = \sum_{l=0}^k\sum_{0\leq j <q(l)}S(l)\varphi((T^{(l)})^j(y(l)))\text{ with }q(l) \leq \|Z(l+1)\|.
\end{equation}
The negative part of the orbit is divided  in a similar way.

\subsubsection*{Space decomposition}\label{sec;Spacedec}
Recall the partition into Rokhlin towers in \S~\ref{sec;Rokhlin}
\begin{equation*}
I = \bigcup_{\alpha\in\mathcal{A}} \bigcup_{i=0}^{Q_{\alpha}(k)-1} T^i(I^{(k)}_\alpha).
\end{equation*}
\begin{itemize}
\item For any pair $x_- < x_+$ of points in $I$, let $k\geq 0$ be the smallest integer such that  $(x_-, x_+)$ contains at least of one interval of the $k$-th partition.
\item Let $J^{(k)}(1), \ldots, J^{(k)}(q(k))$  be all intervals of the $k$-th partition contained in $(x_-, x_+)$. Then  $0<q(k)\leq \|Z(k)\|$.
\item For every $l \geq k$, let $x_+(l) < x_+$ be the largest end point of an interval of the $l$-th partition. Then $x_+(l) \geq x_+(l-1)$ for any $l>k$.
\item For any $l>k$, the interval $(x_+(l-1), x_+(l))$ is the union of intervals $J_+^{(l)}(1), \ldots,$ $J_+^{(l)}(q_+(l))$ of the $l$-th partition for some $0\leq q_+(l)\leq \|Z(l)\|$.
\item The point $x_-(l)$,  $0\leq q_-(l)\leq \|Z(l)\|$, and intervals $J_-^{(l)}(1), \ldots, J_-^{(l)}(q_-(l))$ of the $l$-th partition are defined in a similar way.
\end{itemize}
This yields the following decomposition of $(x_-, x_+)$:
\begin{equation}\label{def;spd}
(x_-,x_+) = \bigcup_{1 \leq q \leq q(k)} J^{(k)}(q)\cup\bigcup_{l>k}\bigcup_{\epsilon = \pm}\bigcup_{1\leq q \leq q_\epsilon(l)} J_\epsilon^{(l)}(q).
\end{equation}



\subsection{H\"older solutions}
In this section, solutions to the cohomological equation $v\circ T-v=\varphi$  are obtained by applying standard Gottschalk-Hedlund arguments for $\varphi \in C^{1+\pag}$. H\"older regularity of solutions
follows from {the} exponential decay of $S(k)\varphi$ and from certain bounds on the growth of  $S(k)D\varphi$.

\begin{lemma}
Let $0\leq a<1$ and $\varphi \in C^{1+\pag}(\sqcup_{\alpha \in \mathcal{A}} I_\alpha)$.  Suppose that  {there exists a constant $c_1(\varphi)\geq0$} such that for any $\tau>0$, we have $\|S(k)\varphi\|_{\sup}=O(e^{(-\lambda_1(1-a)+\tau)k})c_1(\varphi)$.
Then, there exists a continuous solution $v\in C^0(I)$ to the cohomological equation
$\varphi = v\circ T-v$
such that $v(0)=0$, and
\begin{equation}\label{eqn:cobosc}
\sup\{|v(x)-v(y)|:x,y\in I\}\leq 2\sum_{l=0}^\infty \norm{Z(l+1)}\norm{S(l)\varphi}_{\sup}.
\end{equation}
\end{lemma}
\begin{proof}
In view of \eqref{def;tpd}, for any $n \in \N$,
\[\norm{\varphi^{(n)}}_{\sup} \leq 2\sum_{l=0}^\infty \norm{Z(l+1)}\norm{S(l)\varphi}_{\sup}. \]
As $\norm{Z(l+1)}=O(e^{\tau l})$ and $\|S(l)\varphi\|_{\sup}=O(e^{(-\lambda_1(1-a)+\tau)l})c_1(\varphi)$, the series on the right side of the inequality converges, and
the $n$-th Birkhoff sums of $\varphi$ are uniformly bounded. By the classical Gottschalk-Hedlund-type arguments (see \cite[Theorem~3.4]{Ma-Mo-Yo2}), the cohomological equation  has a continuous solution $v$. Moreover, for any $x\in I$ and $n\geq 1$,
\[|v(T^nx)-v(x)|=|\varphi^{(n)}(x)|\leq 2\sum_{l=0}^\infty \norm{Z(l+1)}\norm{S(l)\varphi}_{\sup}.\]
As the orbit $\{T^nx\}_{n\geq 0}$ is dense and $v$ is continuous, this gives \eqref{eqn:cobosc}.
Since the function $v$ is unique up to an additive constant, it can be always be chosen so that $v(0)=0$.
\end{proof}

{In what follows, we will always deal with  solutions satisfying $v(0)=0$.}
 For any interval $J\subset I$, let $\mathrm{osc}(v,J):=\sup\{|v(x)-v(y)|:x,y\in J\}$.

\begin{corollary}\label{cor:oscIk}
Let $0\leq a<1$ and $\varphi \in C^{1+\pag}(\sqcup_{\alpha \in \mathcal{A}} I_\alpha)$. {Suppose that there exists a constant $c_1(\varphi)\geq0$} such that for any $\tau>0$, we have $\|S(k)\varphi\|_{\sup}=O(e^{(-\lambda_1(1-a)+\tau)k})c_1(\varphi)$.
Then, for every $\tau>0$,
\begin{equation}\label{eqn:cobosck}
\mathrm{osc}(v, I^{(k)})=O(e^{(-\lambda_1(1-a)+\tau)k})c_1(\varphi).
\end{equation}
\end{corollary}
\begin{proof}
As $\varphi=v\circ T-v$, for every $k\geq 0$, we have $S(k)\varphi=v\circ T^{(k)}-v$ on $I^{(k)}$. Then, by \eqref{eqn:cobosc} applied to $T^{(k)}:I^{(k)}\to I^{(k)}$, we have
\[\mathrm{osc}(v, I^{(k)})=\sup\{|v(x)-v(y)|:x,y\in I^{(k)}\}\leq 2\sum_{l\geq k}^\infty \norm{Z(l+1)}\norm{S(l)\varphi}_{\sup}.\]
As $\norm{Z(l+1)}=O(e^{\tau l})$ and $\|S(l)\varphi\|_{\sup}=O(e^{(-\lambda_1(1-a)+\tau)l})c_1(\varphi)$, this gives \eqref{eqn:cobosck}.
\end{proof}

The following elementary calculations will be used to estimate $\mathrm{osc}(v,T^i(I_\alpha^{(k)}))$ for $1\leq i < Q_\alpha(k)$ in Lemma~\ref{lem;iterated}.
\begin{lemma}
Let $\varphi \in C^{0+\pag}(\sqcup_{\alpha \in \mathcal{A}} I_\alpha)$. Then, for every $\alpha\in\mathcal{A}$ and any Borel set $J\subset I_\alpha$,
\begin{equation}\label{eq:intJ}
\int_J|\varphi(x)|dx\leq
\left\{
\begin{array}{cl}
\frac{\norm{\varphi}_{L^1(I)}|J|}{|I|}+\frac{2^{a+3}p_a(\varphi)|J|^{1-a}}{a(1-a)}&\text{if }0<a<1,\\
\frac{\norm{\varphi}_{L^1(I)}|J|}{|I|}+4p_a(\varphi)|J|(1+\log\frac{|I|}{|J|})&\text{if }a=0.
\end{array}
\right.
\end{equation}
\end{lemma}

\begin{proof}
By Remark~2.1 in \cite{Fr-Ki2}, for any $x\in \Int I_\alpha$,
\begin{align*}
|\varphi(x)|& \leq \frac{\norm{\varphi}_{L^1}}{|I|} +p_a(\varphi)\Big(\frac{1}{a\min\{x-l_\alpha, r_\alpha-x\}^{a}}+
\frac{2^{a+2}}{a(1-a)|I_\alpha|^{a}}\Big)\text{ if }0<a<1,\\
|\varphi(x)| &\leq \frac{\norm{\varphi}_{L^1}}{|I|} + p_a(\varphi)\Big(\log\frac{|I_\alpha|}{2\min\{x-l_\alpha, r_\alpha-x\}}+2 \Big)\text{ if }a=0.
\end{align*}
It follows that, if $0<a<1$, then
\[\int_J|\varphi(x)|dx\leq \frac{\norm{\varphi}_{L^1(I)}|J|}{|I|}+\frac{2^{a+2}p_a(\varphi)|J|}{a(1-a)|I|^{a}}+\frac{2p_a(\varphi)}{a}\int_{0}^{|J|}x^{-a}dx,\]
and
if $a=0$, then
\[\int_J|\varphi(x)|dx\leq \frac{\norm{\varphi}_{L^1(I)}|J|}{|I|}+2p_a(\varphi)|J|-2p_a(\varphi)\int_{0}^{|J|}\log (x/|I|)dx. \]
This gives \eqref{eq:intJ}.
\end{proof}

\begin{lemma}\label{lem;iterated}
Suppose that $\varphi \in C^{1+\pag}(\sqcup_{\alpha \in \mathcal{A}} I_\alpha)$ {and there exist $c_0(D\varphi),c_1(\varphi)\geq0$} such that  for any $\tau>0$, we have \[\|S(k)\varphi\|_{\sup}=O(e^{(-\lambda_1(1-a)+\tau)k})c_1(\varphi)\quad \text{and}\quad
\frac{\|S(k)D\varphi\|_{L^1(I^{(k)})}}{|I^{(k)}|}=O(e^{(\lambda_1a+\tau)k})c_0(D\varphi).\]
  Then, for any $k\geq0$, $\alpha\in \mathcal{A}$, and $0 \leq N < Q_\alpha(k)$,
\begin{equation}\label{eq:oscmain}
\mathrm{osc}(v,T^N(I_\alpha^{(k)}))= \mathrm{osc}(v,I_\alpha^{(k)})+ O(e^{(-\lambda_1(1-a)+\tau)k})(c_0(D\varphi)+p_a(D\varphi)).
\end{equation}
\end{lemma}
\begin{proof}
Since $\varphi = v\circ T-v$, by telescoping, for any $x_1,x_2\in I^{(k)}_\alpha$, we have
\[v(T^Nx_2)-v(T^Nx_1)-(v(x_2)-v(x_1))=\varphi^{(N)}(x_2)-\varphi^{(N)}(x_1)=\int_{x_1}^{x_2}\sum_{i=0}^{N-1}D\varphi(T^ix)\,dx.\]
Hence
\begin{equation}\label{eq:oscint}
\mathrm{osc}(v,T^N(I_\alpha^{(k)}))\leq \mathrm{osc}(v,I_\alpha^{(k)})+\int_{I_\alpha^{(k)}}\Big|\sum_{i=0}^{N-1}D\varphi(T^ix)\Big|\,dx.
\end{equation}
In view of \eqref{def;tpd}, for every $x\in I^{(k)}_\alpha$, we have
\begin{equation}\label{eqn;specialbirkhoff}
\sum_{i=0}^{N-1}D\varphi(T^ix) = \sum_{l=0}^{k}\sum_{0\leq i <q(l)}S(l)D\varphi((T^{(l)})^ix(l)),
\end{equation}
wher $0\leq q(l) \leq \norm{Z(l+1)}$, $I^{(k)}_\alpha\ni x\mapsto x(l)\in J_l\subset I^{(l)}$ is a translation, and $J_l$ is the image of $I^{(k)}_\alpha$ by this translation.
It follows that,
\begin{equation}\label{eq:sumint}
\int_{I_\alpha^{(k)}}\Big|\sum_{i=0}^{N-1}D\varphi(T^ix)\Big|\,dx
\leq \sum_{l=0}^{k}\sum_{0\leq i <q(l)}\int_{(T^{(l)})^iJ_l}|S(l)D\varphi(x)|dx.
\end{equation}

Assume that $0<a<1$. As $|(T^{(l)})^iJ_l|=|J_l|=|I^{(k)}_\alpha|$, in view of \eqref{eq:intJ},
\[\int_{(T^{(l)})^iJ_l}|S(l)D\varphi(x)|dx\leq \frac{\norm{S(l)D\varphi}_{L^1(I^{(l)})}|I^{(k)}_\alpha|}{|I^{(l)}|}+\frac{2^{a+3}p_a(S(l)D\varphi)|I^{(k)}_\alpha|^{1-a}}{a(1-a)}.\]
By \eqref{eqn;renormpaos2}, there exists $C>0$ such that
\begin{align}\label{eqn;renormpaos3}
\begin{split}
&p_a(S(l)D\varphi) \leq Cp_a(D\varphi)\text{ if }0<a<1,\\
&p_a(S(l)D\varphi) \leq C(1+\log\|Q(l)\|)p_a(D\varphi)\text{ if }a=0.
\end{split}
\end{align}
As $\frac{\|S(l)D\varphi\|_{L^1(I^{(l)})}}{|I^{(l)}|}=O(e^{(\lambda_1a+\tau)l})c_0(D\varphi)$ and $|I^{(k)}|=O(e^{-\lambda_1k})$, it follows that
\begin{equation}\label{eq:intJl}
\int_{(T^{(l)})^iJ_l}|S(l)D\varphi(x)|dx=O(e^{(-\lambda_1(1-a)+\tau)k})(c_0+p_a)(D\varphi).
\end{equation}
If $a=0$, then, by \eqref{eq:intJ},
\[\int_{(T^{(l)})^iJ_l}|S(l)D\varphi(x)|dx\leq \tfrac{\norm{S(l)D\varphi}_{L^1(I^{(l)})}|I^{(k)}_\alpha|}{|I^{(l)}|}+4p_a(S(l)D\varphi)|I^{(k)}_\alpha|(1+\log \tfrac{|I^{(l)}|}{|I^{(k)}_\alpha|}).\]
In view of \eqref{eq:invIk}, $\log |I^{(l)}|/|I^{(k)}_\alpha|\leq \log |I|/|I^{(k)}_\alpha|=\log O(e^{(\lambda_1+\tau)k})=O(e^{\tau k})$, and $\log\|Q(l)\|=O(e^{\tau k})$ for $l\leq k$. Moreover,  by \eqref{eqn;renormpaos3} we also get \eqref{eq:intJl} when $a=0$.
By \eqref{eq:sumint}, this gives
\begin{align*}
\int_{I_\alpha^{(k)}}\Big|\sum_{i=0}^{N-1}D\varphi(T^ix)\Big|\,dx&=k\|Z(k+1)\|O(e^{(-\lambda_1(1-a)+\tau)k})(c_0+p_a)(D\varphi)\\
&=O(e^{(-\lambda_1(1-a)+3\tau)k})(c_0+p_a)(D\varphi).
\end{align*}
In view of \eqref{eq:oscint}, this gives \eqref{eq:oscmain}.
\end{proof}

By combining previous lemmas, , and assuming a decay condition on  {$\|S(k)\varphi\|_{\sup}$} along with a bound on the growth of  $S(k)D\varphi$ {in the $L^1$-norm}, a H\"older continuous solution to the cohomological equation is obtained.
\begin{theorem}\label{thm;holsol}
Suppose that $\varphi \in C^{1+\pag}(\sqcup_{\alpha \in \mathcal{A}} I_\alpha)$ and {there exist  $c_0(D\varphi),c_1(\varphi)\geq0$}  such that   for any $\tau>0$, we have
\[\|S(k)\varphi\|_{\sup}=O(e^{(-\lambda_1(1-a)+\tau)k})c_1(\varphi) \quad\text{and}\quad \frac{\|S(k)D\varphi\|_{L^1(I^{(k)})}}{|I^{(k)}|}=O(e^{(\lambda_1a+\tau)k})c_0(D\varphi).\]
Then, there exists a continuous solution $v:I\to\R$ to the cohomological equation $\varphi=v\circ T-v$ such that $v(0)=0$,
and for any $0<\tau<1-a$, we have $v \in C^{(1-a)-\tau}(I)$. Moreover, there exists $C_\tau>0$ such that $\|v\|_{C^{(1-a)-\tau}}\leq C_\tau (c_1(\varphi)+c_0(D\varphi)+p_a(D\varphi))$.
\end{theorem}

\begin{proof}
For any pair $x < y$ of points in $I$, we use the space decomposition of the interval $(x,y)$ introduced in the beginning of the section.
Then,
\[
|v(y) - v(x)| \leq \sum_{q=1}^{q(k)} \mathrm{osc}(v,J^{(k)}(q)) + \sum_{l > k} \sum_{\epsilon = \pm}\sum_{q=1}^{q_\epsilon(l)}  \mathrm{osc}(v,J_\epsilon^{(l)}(q)),
\]
with $q(k)\leq \|Z(k)\|$ and $q_\pm(l)\leq \|Z(l)\|$.
As each interval $J^{(k)}(q)$ is of the form $T^nI^{(k)}_\alpha$ for some $0\leq n<Q_\alpha(k)$, and each interval $J_{\pm}^{(l)}(q)$ is of the form $T^nI^{(l)}_\alpha$ for some $0\leq n<Q_\alpha(l)$, in view of Corollary~\ref{cor:oscIk} and Lemma~\ref{lem;iterated}, for any  $\tau>0$, we have
\begin{gather*}
\mathrm{osc}(v,J^{(k)}(q))= O(e^{(-\lambda_1(1-a)+\tau)k})(c_1(\varphi)+c_0(D\varphi)+p_a(D\varphi)),\\
\mathrm{osc}(v,J_\pm^{(l)}(q))= O(e^{(-\lambda_1(1-a)+\tau)l})(c_1(\varphi)+c_0(D\varphi)+p_a(D\varphi)).
\end{gather*}
It follows that
\begin{align*}
|v(y) - v(x)| = O\Big(\sum_{l\geq k} \norm{Z(l)}e^{(-\lambda_1(1-a)+\tau)l}\Big)(c_1(\varphi)+c_0(D\varphi)+p_a(D\varphi)) .
\end{align*}
As $\|Z(l)\|=O(e^{\tau l})$, we obtain
\[|v(y) - v(x)| = O(e^{(-\lambda_1(1-a)+2\tau)k})(c_1(\varphi)+c_0(D\varphi)+p_a(D\varphi)) .\]
By  the choice of $k$,
$
|y-x| \geq \min_{\alpha \in \mathcal{A}}|I^{(k)}_\alpha| \geq c_\tau e^{-(\lambda_1+\tau)k}
$
for some $c_\tau>0$.
It follows that
\begin{align}\label{eqn;vxy}
|v(y) - v(x)| = O(1) (c_1(\varphi)+c_0(D\varphi)+p_a(D\varphi)) |y-x|^{\frac{\lambda_1(1-a)-2\tau}{\lambda_1+\tau}}.
\end{align}
{Since $v(0)=0$ and \eqref{eqn;vxy} holds for arbitrary $\tau>0$,} this completes the proof.
\end{proof}

\subsection{Higher regularity}\label{sec:highreg}
Higher regularity of solutions  is obtained by applying Theorem~\ref{thm;holsol} as the initial step of induction.
\begin{theorem}\label{thm:cohsolmain}
Let $n\geq 1$ and $0\leq a <1$. Assume that $T$ satisfies the \ref{FDC}.
Let $\varphi \in C^{n+\pag}(\sqcup_{\alpha \in \mathcal{A}} I_\alpha)$ be a map such that for any $\tau>0$, we have
\begin{equation}\label{asscoh1}
\|S(k)D^l\varphi\|_{\sup}=O(e^{(-\lambda_1(n-l-a)+\tau)k})\|D^l\varphi\|_{C^{n-l+\pa}},\text{ for }0\leq l<n,
\end{equation}
and
\begin{equation}\label{asscoh2}
\frac{1}{|I^{(k)}|}\|S(k)D^{n}\varphi\|_{L^1(I^{(k)})}=O(e^{(\lambda_1a+\tau)k})\|D^{n}\varphi\|_{C^{0+\pa}}.
\end{equation}
Then, there exists a $C^{n-1}$-solution $v:I\to\R$ to the cohomological equation $\varphi=v\circ T-v$ such that $v(0)=0$, and for any $0<\tau<1-a$,
we have $v \in C^{n-a-\tau}(I)$. Moreover,  there exists $C_{\tau,n}>0$ such that $\|v\|_{C^{n-a-\tau}}\leq C_{\tau,n}\|\varphi\|_{C^{n+\pa}}$.
\end{theorem}
\begin{proof} The proof is by induction on $n$. For $n=1$, our claim follows from Theorem~\ref{thm;holsol} applied to $c_1(\varphi)=\|\varphi\|_{C^{1+\pa}}$ and $c_0(D\varphi)=\|D\varphi\|_{C^{0+\pa}}$.

Suppose that for some $n\geq 1$, if $\varphi \in C^{n+\pag}(\sqcup_{\alpha \in \mathcal{A}} I_\alpha)$ satisfies \eqref{asscoh1} and \eqref{asscoh2}, then there exists a $C^{n-1}$-solution $v$ to the cohomological equation such that, for any $\tau>0$, we have $v \in C^{n-a-\tau}(I)$ and $\|v\|_{C^{n-a-\tau}}\leq C_{\tau,n}\|\varphi\|_{C^{n+\pa}}$.

Let  $\varphi \in C^{n+1+\pag}(\sqcup_{\alpha \in \mathcal{A}} I_\alpha)$ be such that
\begin{gather*}
\|S(k)D^l\varphi\|_{\sup}=O(e^{(-\lambda_1(n+1-l-a)+\tau)k})\|D^l\varphi\|_{C^{n+1-l+\pa}},\text{ for }0\leq l\leq n,\text{ and }\\
\frac{1}{|I^{(k)}|}\|S(k)D^{n+1}\varphi\|_{L^1(I^{(k)})}=O(e^{(\lambda_1a+\tau)k})\|D^{n+1}\varphi\|_{C^{0+\pa}}.
\end{gather*}
It follows that $D\varphi \in C^{n+\pag}(\sqcup_{\alpha \in \mathcal{A}} I_\alpha)$ satisfies \eqref{asscoh1} and \eqref{asscoh2}. By the induction hypothesis, there exists $v_0\in C^{n-1}(I)$ such that $D\varphi=v_0\circ T-v_0$, with $v_0(0)=0$, and for any $\tau>0$, we have $v_0\in C^{n-a-\tau}(I)$ with $\|v_0\|_{C^{n-a-\tau}}\leq C_{\tau,n}\|D\varphi\|_{C^{n+\pa}}$. By integrating, there exists $\chi \in \Gamma$ that satisfies
$\varphi =  \widetilde{v}_0\circ T -\widetilde{v}_0  + \chi$ (recall that $\widetilde{v}_0(x)=\int_0^xv_0(s)ds$).
Note that for any $k \geq 1$,
\[
S(k)\varphi = S(k)(\widetilde{v}_0\circ T -\widetilde{v}_0) + Q(k)\chi.
\]
By assumption,
\begin{align}\label{neq:<-la}
\begin{split}
\|S(k)\varphi\|_{\sup}&= O(e^{(-\lambda_1(n+1-a)+\tau)k})\|\varphi\|_{C^{n+1+\pa}}\\
&= O(e^{-\lambda_1k} e^{(-\lambda_1(1-a)+\tau)k})\|\varphi\|_{C^{n+1+\pa}}=O(e^{-\lambda_1k})\|\varphi\|_{C^{n+1+\pa}}.
\end{split}
\end{align}
On the other hand, for any $x\in I^{(k)}_\alpha$,
\begin{equation}\label{eq:Sku0}
|S(k)(\widetilde{v}_0\circ T-\widetilde v_0)(x)| = |\widetilde{v}_0(T^{Q_\alpha(k)} x)-\widetilde v_0(x)|\leq \|v_0\|_{\sup}|x - T^{Q_\alpha(k)} x|.
\end{equation}
It follows that
\[
\norm{S(k)(\widetilde{v}_0\circ T-\widetilde v_0)}_{\sup} \leq \|v_0\|_{\sup}|I^{(k)}| = O(e^{-\lambda_1k})\|D\varphi\|_{C^{n+\pa}}.
\]
Therefore, $\norm{Q(k)\chi} = O(e^{-\lambda_1 k})\|\varphi\|_{C^{n+1+\pa}}$. In view of \eqref{eq:Oscond}, $\chi \in E_{-1}(\pi,\lambda)$.
As $E_{-1}(\pi,\lambda)$ is one-dimensional, by Remark~\ref{rmk:h-1}, $\chi = c(\bar{\xi}-\bar{\xi}\circ T)$ for some $c=c(\varphi) \in \R$ (recall that $\bar{\xi}(x)=x$).
Note that  $|c(\varphi)|\leq \|v_0\|_{\sup}$. Indeed, by \eqref{eq:Sku0},
\[\norm{S(k)(\widetilde{v}_0\circ T-\widetilde v_0)}_{\sup} \leq \|v_0\|_{\sup}
\norm{S(k)(\bar{\xi} - \bar{\xi}\circ T)}_{\sup}.\]
As $\frac{1}{k}\log \norm{S(k)(\bar{\xi} - \bar{\xi}\circ T)}_{\sup} \to-\lambda_1$, in view of \eqref{neq:<-la}, we obtain
 $\|S(k)\varphi\|_{\sup}=o(\norm{S(k)(\bar{\xi} - \bar{\xi}\circ T)}_{\sup})$. It follows that
\begin{align*}
|c(\varphi)|\norm{S(k)(\bar{\xi} - \bar{\xi}\circ T)}_{\sup}&=\|S(k)\chi\|_{\sup}\leq \norm{S(k)\varphi}_{\sup}+
\norm{S(k)(\widetilde{v}_0\circ T-\widetilde{v}_0)}_{\sup}\\
&\leq
(\|v_0\|_{\sup}+o(1))\norm{S(k)(\bar{\xi} - \bar{\xi}\circ T)}_{\sup}.
\end{align*}
Hence $|c(\varphi)|\leq \|v_0\|_{\sup}$.

Let $v:I\to\R$, $v=\widetilde{v}_0-c(\varphi)\bar{\xi}$. Then, $\varphi=v\circ T-v$ and $v\in C^{n+1-a-\tau}(I)$ with
\begin{align*}
\|Dv\|_{C^{n-a-\tau}}&=\|v_0\|_{C^{n-a-\tau}}+|c(\varphi)|\leq \|v_0\|_{C^{n-a-\tau}}+\|v_0\|_{\sup}\leq 2C_{\tau,n}\|D\varphi\|_{C^{n+\pa}}.
\end{align*}
As $v(0)=0$, this gives
\begin{align*}
\|v\|_{C^{n+1-a-\tau}}&=\|v\|_{\sup}+\|Dv\|_{C^{n-a-\tau}}\leq |I|\|Dv\|_{\sup}+\|Dv\|_{C^{n-a-\tau}}\\
&\leq (|I|+1)\|Dv\|_{C^{n-a-\tau}}\leq
2(|I|+1)C_{\tau,n}\|D\varphi\|_{C^{n+\pa}}.
\end{align*}
This completes the proof.
\end{proof}

\begin{corollary}\label{cor:cobh}
For every $n\geq 0$, there exists a polynomial $v_n\in \R_{n+1}[x]$ such that $h_{-1,n}=v_n\circ T-v_n$ and $v_n(0)=0$.

For every $\bar t\in\mathscr{TF}$, if $\mathfrak{o}(\bar t)>r>0$, then there exists $v_{\bar t}\in C^r(I)$ such that $h_{\bar t}(0)=0$ and $h_{\bar t}=v_{\bar t}\circ T-v_{\bar t}$.
\end{corollary}

\begin{proof}
In view of \eqref{eq:prophch1} and \eqref{eq:grhi}, for every $0\leq l\leq n$, we have $D^l h_{-1,n}= h_{-1,n-l}$, and
\[\lim_{k\to\infty}\frac{1}{k}\log\|S(k)D^l h_{-1,n}\|_{\sup}=-\lambda_1(n-l+1)\text{ with }D^{n+1} h_{-1,n}=0.\]
Therefore, $h_{-1,n} \in C^{n+1+\pag}(\sqcup_{\alpha \in \mathcal{A}} I_\alpha)$ satisfies \eqref{asscoh1} and \eqref{asscoh2} for $a=0$.
Then, by Theorem~\ref{thm:cohsolmain}, there exists $v_n\in C^{n}(I)$ such that $h_{-1,n}=v_n\circ T-v_n$ and $v_n(0)=0$.
{Thus $h_{-1}=D^nh_{-1,n}=D^nv_n\circ T-D^nv_n$. On the other hand, by Remark~\ref{rmk:h-1}, $h_{-1}=c(\bar{\xi}\circ T-\bar{\xi})$ for some $c\neq 0$. Therefore, $D^n v_n - c\bar{\xi}$ is invariant under the map $T$. Then, by the ergodicity of $T$,
we have $D^n v_n(x)=c\bar{\xi}(x)+d=cx+d$.}
It follows that $v_n\in \R_{n+1}[x]$.

\medskip

Suppose that $\bar t\in\mathscr{TF}$ and $\mathfrak{o}(\bar t)>r>0$. Let $n:=\lceil\mathfrak{o}(\bar t)\rceil-1$ and $a:=\mathfrak{o}(\bar t)-n$,  and choose $\tau>0$ so that
$r<n-a-\tau<n-a=\mathfrak{o}(\bar t)$. In view of \eqref{eq:prophch1} and \eqref{eq:grhi}, for every $0\leq l\leq n$,
\[\lim_{k\to\infty}\frac{1}{k}\log\|S(k)D^l h_{\bar t}\|_{\sup}=-\lambda_1(\mathfrak{o}(\bar t)-l)=-\lambda_1(n-l-a).\]
As $h_{\bar t} \in C^{n+\pag}(\sqcup_{\alpha \in \mathcal{A}} I_\alpha)$, by Theorem~\ref{thm:cohsolmain}, there exists $v_{\bar t}\in C^{n-a-\tau}(I)$  such that $h_{\bar t}(0)=0$ and $h_{\bar t}=v_{\bar t}\circ T-v_{\bar t}$. As $r<n-a-\tau$, this gives our claim.
\end{proof}

We finish the section by summarizing the complete conditions under which smooth solutions to the cohomological equation exist for almost every IETs.
\begin{theorem}\label{thm:cohsolmain1}
Let $n\geq 1$, $0\leq a <1$, and $0<r<n-a$ such that $r\notin \{\mathfrak{o}(\bar t):\bar t\in\mathscr{T}_{a,n}\}$. Assume that $T$ satisfies the \ref{FDC}.
Let $\varphi \in C^{n+\pag}(\sqcup_{\alpha \in \mathcal{A}} I_\alpha)$ be a map such that $\mathfrak{f}_{\bar t}(\varphi)=0$ for all $\bar t \in \mathscr{T}_{a,n}$ with $\mathfrak{o}(\bar t)<r$.
Then there exists a solution $v\in C^r(I)$ to the cohomological equation $\varphi=v\circ T-v$ such that $v(0)=0$. The operator
\begin{equation}\label{eq:opu}
\bigcap_{\bar t \in \mathscr{T}_{a,n},\ \mathfrak{o}(\bar t)<r} \ker(\mathfrak{f}_{\bar t})\ni \varphi\mapsto v\in  C^{r}(I)
\end{equation}
is linear and bounded.

Moreover, there exist  bounded operators $\Gamma_n: C^{n+\pa}(\sqcup_{\alpha\in \mathcal{A}}I_{\alpha}) \rightarrow \Gamma_n(\sqcup_{\alpha \in \mathcal{A}} I_\alpha)$  and $V_n :C^{n+\pag}(\sqcup_{\alpha \in \mathcal{A}} I_\alpha) \rightarrow C^{n-1}(I)$
such that
\[
\varphi = V_n(\varphi) \circ T -V_n(\varphi)  + \Gamma_n(\varphi).
\]
More precisely,  for every $0<\tau<1-a$, the operator $V_n$ takes {values} in  $C^{n-a-\tau}(I)$ and $V_n :C^{n+\pag}(\sqcup_{\alpha \in \mathcal{A}} I_\alpha) \rightarrow C^{n-a-\tau}(I)$ is also bounded.
\end{theorem}
\begin{proof}
Assume that $\mathfrak{f}_{\bar t}(\varphi) = 0$ for every $\bar t \in \mathscr{T}_{a,n}$ with $\mathfrak{o}(\bar t)<r$. Then
\[\varphi = \mathfrak{r}_{a,n}(\varphi)+\sum_{\bar t\in  \mathscr{T}_{a,n},\ \mathfrak{o}(\bar t)> r}\mathfrak{f}_{\bar t}(\varphi)h_{\bar t}+\sum_{\bar t\in \mathscr{T}^*_{a,n}\setminus \mathscr{T}_{a,n}}\mathfrak{f}_{\bar t}(\varphi)h_{\bar t}.\]
Choose $\tau>0$ such that $r<n-a-\tau$. In view of Theorem~\ref{thm;spdecomp}~and~\ref{thm:cohsolmain},
there exists $\bar{v}\in C^{n-a-\tau}(I)$ such that $\mathfrak{r}_{a,n}(\varphi)=\bar{v}\circ T-\bar{v}$ and $\bar{v}(0)=0$. There also exists a constant  $C_{\tau,n}>0$ such that  $\|\bar{v}\|_{C^{n-a-\tau}}\leq C_{\tau,n}\|\mathfrak{r}_{a,n}(\varphi)\|_{C^{n+\pa}}$.
By Corollary~\ref{cor:cobh}, for every $\bar t\in\mathscr{T}^*_{a,n}\setminus \mathscr{T}_{a,n}$, there exists a polynomial $v_{\bar t}$ such that $h_{\bar t}= v_{\bar t}\circ T-v_{\bar t}$ and $v_{\bar t}(0)=0$.
Moreover, if $\bar t\in\mathscr{T}_{a,n}$ and $\mathfrak{o}(\bar t)>r>0$, then, again by Corollary~\ref{cor:cobh}, there exists $v_{\bar t}\in C^r(I)$ such that $h_{\bar t}=v_{\bar t}\circ T-v_{\bar t}$ and $v_{\bar t}(0)=0$.
It follows that
\[\varphi = \bar{v}\circ T-\bar{v}+\sum_{\bar t\in  \mathscr{T}_{a,n},\ \mathfrak{o}(\bar t)> r}\mathfrak{f}_{\bar t}(\varphi)
(v_{\bar t}\circ T-v_{\bar t})+\sum_{\bar t\in \mathscr{T}^*_{a,n}\setminus \mathscr{T}_{a,n}}\mathfrak{f}_{\bar t}(\varphi)
(v_{\bar t}\circ T-v_{\bar t}),\]
and
\[v=\bar{v}+\sum_{\bar t\in  \mathscr{T}_{a,n},\ \mathfrak{o}(\bar t)> r}\mathfrak{f}_{\bar t}(\varphi)v_{\bar t}+\sum_{\bar t\in \mathscr{T}^*_{a,n}\setminus \mathscr{T}_{a,n}}\mathfrak{f}_{\bar t}(\varphi)v_{\bar t}\in C^{r}(I)\]
satisfies $\varphi=v\circ T-v$ and $v(0)=0$. Moreover,
\begin{align*}
\|v\|_{C^{r}}&\leq  C_{\tau,n}\|\mathfrak{r}_{a,n}(\varphi)\|_{C^{n+\pa}}+\sum_{\bar t\in  \mathscr{T}^*_{a,n}}|\mathfrak{f}_{\bar t}(\varphi)|\|v_{\bar t}\|_{C^{r}}\\
&\leq C_{\tau,n}\|\varphi\|_{C^{n+\pa}}+\sum_{\bar t\in  \mathscr{T}^*_{a,n}}(C_{\tau,n}\|h_{\bar t}\|_{C^{n+\pa}}+\|v_{\bar t}\|_{C^{r}})|\mathfrak{f}_{\bar t}(\varphi)|.
\end{align*}
As all functionals $\mathfrak{f}_{\bar t} :C^{n+\pa}(\sqcup_{\alpha\in \mathcal{A}}I_{\alpha})\to \C$ are bounded, the operator \eqref{eq:opu} is bounded as well.

\medskip

The second part of the theorem follows directly from  Theorems~\ref{thm;spdecomp}~and~\ref{thm:cohsolmain}, with $\Gamma_n(\varphi)=\sum_{\bar t\in \mathscr{T}^*_{a,n}}\mathfrak{f}_{\bar t}(\varphi)h_{\bar t}$
and $V_n(\varphi)$ being the solution to the cohomological equation  $\mathfrak{r}_{a,n}(\varphi)=V_n(\varphi)\circ T-V_n(\varphi)$.
\end{proof}

\begin{remark}
In view of the second part of Theorem~\ref{thm;spdecomp}, the regularity of the solution $v$ to the equation $\varphi=v\circ T-v$, with $\varphi\in C^{n+\pa}(\sqcup_{\alpha\in \mathcal{A}}I_{\alpha})$, proved in Theorem~\ref{thm:cohsolmain1}, is optimal. Indeed, let $r_0=\mathfrak{o}(\bar{t}_0)>0$
for some $\bar{t}_0\in \mathscr{T}_{a,n}$. Let $\varphi\in C^{n+\pa}(\sqcup_{\alpha\in \mathcal{A}}I_{\alpha})$ be such that $\mathfrak{f}_{\bar{t}_0}(\varphi)\neq 0$ and $\mathfrak{f}_{\bar t}(\varphi)=0$ for all $\bar t\in \mathscr{T}_{a,n}$ with $\mathfrak{o}(\bar t)<r_0$. By Theorem~\ref{thm:cohsolmain1}, the solution $v$ to the cohomological equation belongs to $C^r(I)$ for any $r<r_0$. On the other hand, by Theorem~\ref{thm;spdecomp}, $v\notin C^r(I)$ for any $r>r_0$.
Hence, the exponent $r_0=\mathfrak{o}(\bar{t}_0)$ is a threshold for the regularity of the solution.

Similarly, if $r_0=n-a$, $C_{\alpha,n}^{a,\pm}(\varphi)\neq 0$ for some $\alpha\in\mathcal{A}$, and $\mathfrak{f}_{\bar t}(\varphi)=0$ for all $\bar t\in \mathscr{T}_{a,n}$ with $\mathfrak{o}(\bar t)<r_0$ (in fact, by \eqref{eq:maxex}, for any $\bar t\in \mathscr{T}_{a,n}$), then $v\in C^r(I)$ for every $r<r_0$, and $v\notin C^r(I)$ for every $r>r_0$.
\end{remark}

\section{Proofs of the main  theorems}\label{sec;last}
In this final section, we construct generalized Forni invariant distributions $\mathfrak{F}_{\bar t}$, on function spaces on a compact surface $M$.
Roughly speaking, $\mathfrak{F}_{\bar t}$  is achieved by composing the operator $f\mapsto \varphi_f$ with the functional $\mathfrak{f}_{\bar t}$.
Since the invariant distributions $\mathfrak{f}_{\bar t}$ are on $C^{n+\pa}$, we need to perform in Section~\ref{sec:Forgendis} an additional correction of $\varphi_f$ so that the resulting function belongs to $C^{n+\pa}$.

Finally, in Section~\ref{sec:proofmain}, we apply the tools developed in \cite{Fr-Ki2} to  transition from cohomological equations over IETs to equations for locally Hamiltonian flows on any minimal component $M'\subset M$.
 Then, by combining these with the cohomological results over IETs from Section \ref{sec;coh-iet}, we obtain  optimal regularity of solutions to the cohomological equations $Xu=f$. The regularity is determined by the order (or the hat-order) of three different types of invariant distributions: $\mathfrak{C}^{k}_{\sigma,l}$,  $\mathfrak{d}^{k}_{\sigma,j}$, and $\mathfrak{F}_{\bar t}$.

\subsection{Counterparts of Forni's invariant distributions}\label{sec:Forgendis}
Let $M$ be a {compact, connected, orientable} $C^\infty$-surface. Let $\psi_\R$ be a locally Hamiltonian $C^\infty$-flow on $M$ with isolated fixed points, such that all its saddles are perfect and all saddle connections are loops. Let $M'\subset M$ be a minimal component of the flow, and let $I\subset M'$ be a transversal curve. The corresponding IET $T:I\to I$ exchanges the intervals $\{I_\alpha:\alpha \in \mathcal{A}\}$. Let $\tau:I\to\R_{>0}$ be the first return time map. We define the operator $f\mapsto \varphi_f$  for every integrable map $f:M\to\R$ as follows:
\[\varphi_f(x)=\int_0^{\tau(x)}f(\psi_tx)dt\text{ for every }x\in I.\]
If $f$ is a smooth function on $M$, then $\varphi_f$ is also smooth on every $\Int I_\alpha$, $\alpha\in\mathcal{A}$. The function $\varphi_f$ may be discontinuous at the endpoints of the intervals or may have singularities.
 A detailed description of its behavior near the endpoints of the exchanged
intervals is given in  \cite{Fr-Ki2}.

Suppose the equation $\varphi_f=v\circ T-v$ has a smooth solution $v:I\to\R$. This is a necessary condition for the existence of a smooth solution to the equation $Xu=f$. In a sense, this is also a sufficient condition for the existence of a smooth solution to the equation $Xu=f$. We can define $u_{v,f}:M'\setminus (\mathrm{Sd}(\psi_\R)\cup \SL)\to\R$  as follows:
if $\psi_tx\in I$ for some $t\in\R$, then
\[u_{v,f}(x):=v(\psi_tx)-\int_{0}^tf(\psi_sx)\,ds.\]
The map $u_{v,f}$ is a smooth solution of $Xu=f$, but only on $M'\setminus (\mathrm{Sd}(\psi_\R)\cup \SL)$ which is an open subset of $M'$.
Usually, $u_{v,f}$ cannot be smoothly extended to $M'$ or even to the end compactification  $M'_e$ defined in \cite{Fr-Ki2}.
As proven in \cite[Theorem 1.2]{Fr-Ki2}, the vanishing of some invariant distributions $\mathfrak{d}^k_{\sigma,j}(f)$ and $\mathfrak{C}^k_{\sigma,l}(f)$ is the necessary and sufficient condition for the existence of a smooth solution (an extension of $u_{v,f}$) to $Xu=f$ on $M'_e$.

\medskip

After \cite{Fr-Ki2}, for any $[(\sigma,k,l)]\in \mathscr{TC}/\sim$, we define a map $\widehat{\xi}_{[(\sigma,k,l)]}:I\to\R$. For any closed interval $J\subset I_\alpha$, we denote by $J^\tau\subset M$ the closure of the set of orbit segments starting from $\Int J$ and running until the first return to $I$. For any $[(\sigma,k,l)]\in \mathscr{TC}/\sim$,  there exist $\alpha\in\mathcal{A}$ and an interval $J$ of the form $[l_\alpha,l_\alpha+\vep]$ or $[r_\alpha-\vep,r_\alpha]$ such that $l_\alpha$ or $r_\alpha$ is
the first backward meeting point of a separatrix incoming to $\sigma\in\mathrm{Sd}(\psi_\R)$, and $J^\tau$ contains all angular sectors $U_{\sigma,l'}$ for which $(\sigma,k,l')\sim(\sigma,k,l)$. Let $\widehat{\xi}_{[(\sigma,k,l)]}:I\to\R$ be a map such that:
\begin{itemize}
\item $\widehat{\xi}_{[(\sigma,k,l)]}$ is zero on any interval $I_\beta$ with $\beta\neq \alpha$;
\item if $J=[l_\alpha,l_\alpha+\vep]$, then for any $s\in I_\alpha$,
\begin{align*}
\widehat{\xi}_{[(\sigma,k,l)]}(s)&=\frac{(s-l_\alpha)^{\frac{k-(m_\sigma-2)}{m_\sigma}}}{m_\sigma^2 k!}\text{ if }k\neq m_\sigma-2\ \operatorname{mod}m_\sigma,\\
\widehat{\xi}_{[(\sigma,k,l)]}(s)&=-\frac{(s-l_\alpha)^{\frac{k-(m_\sigma-2)}{m_\sigma}}\log(s-l_\alpha)}{m_\sigma^2 k!} \text{ if }k= m_\sigma-2\ \operatorname{mod}m_\sigma;
\end{align*}
\item if $J=[r_\alpha-\vep,r_\alpha]$, then for any $s\in I_\alpha$,
\begin{align*}
\widehat{\xi}_{[(\sigma,k,l)]}(s)&=\frac{(r_\alpha-s)^{\frac{k-(m_\sigma-2)}{m_\sigma}}}{m_\sigma^2 k!}\text{ if }k\neq m_\sigma-2\ \operatorname{mod}m_\sigma,\\
\widehat{\xi}_{[(\sigma,k,l)]}(s)&=-\frac{(r_\alpha-s)^{\frac{k-(m_\sigma-2)}{m_\sigma}}\log(r_\alpha-s)}{m_\sigma^2 k!} \text{ if }k= m_\sigma-2\ \operatorname{mod}m_\sigma.
\end{align*}
\end{itemize}
As $\frac{k-(m_\sigma-2)}{m_\sigma}=\mathfrak{o}(\sigma,k)$,  we have $\widehat{\xi}_{[(\sigma,k,l)]}\in C^{n_{\sigma,k}+\mathrm{P}_{a_{\sigma,k}}\mathrm{G}}(\sqcup_{\alpha \in \mathcal{A}}I_\alpha)$ with $n_{\sigma,k}:=\lceil\mathfrak{o}(\sigma,k)\rceil$ and $a_{\sigma,k}:=n-\mathfrak{o}(\sigma,k)$, and exactly one of  $C_\alpha^{+}(D^{n_{\sigma,k}}\widehat{\xi}_{[(\sigma,k,l)]})$, $C_\alpha^{-}(D^{n_{\sigma,k}}\widehat{\xi}_{[(\sigma,k,l)]})$ is non-zero.

Let us consider ${\xi}_{[(\sigma,k,l)]}\in C^{n_{\sigma,k}+\mathrm{P}_{a_{\sigma,k}}\mathrm{G}}(\sqcup_{\alpha \in \mathcal{A}}I_\alpha)$ given by
\begin{align}\label{def:corrsing}
\begin{aligned}
{\xi}_{[(\sigma,k,l)]}&:=\mathfrak{r}_{a_{\sigma,k},n_{\sigma,k}}(\widehat{\xi}_{[(\sigma,k,l)]}) =\widehat{\xi}_{[(\sigma,k,l)]}-\sum_{\bar{t}\in \mathscr{T}^*_{a_{\sigma,k},n_{\sigma,k}}}\mathfrak{f}_{\bar{t}}(\widehat{\xi}_{[(\sigma,k,l)]})h_{\bar t}\\
&=\widehat{\xi}_{[(\sigma,k,l)]}-\sum_{\bar{t}\in \mathscr{TF}^*,\mathfrak{o}(\bar t)<\mathfrak{o}(\sigma,k)}\mathfrak{f}_{\bar{t}}(\widehat{\xi}_{[(\sigma,k,l)]})h_{\bar t}.
\end{aligned}
\end{align}
In view of Lemma~\ref{lem:unidecom},
\begin{equation}\label{eq:zerod}
\mathfrak{f}_{\bar t}(\xi_{[(\sigma,k,l)]})=0 \quad\text{if}\quad\mathfrak{o}(\bar t)<\mathfrak{o}(\sigma,k).
\end{equation}
Since $C_\alpha^{\pm}(D^{n_{\sigma,k}}{\xi}_{[(\sigma,k,l)]})=C_\alpha^{\pm}(D^{n_{\sigma,k}}\widehat{\xi}_{[(\sigma,k,l)]})\neq 0$, by Theorem~\ref{thm;spdecomp},
\begin{align}\label{eq:exxi}
\begin{split}
&\lim_{j\to\infty}\frac{1}{j}\log\big(\|S(j)({\xi}_{[(\sigma,k,l)]})\|_{L^1(I^{(j)})}/|I^{(j)}|\big)=-\lambda_1(n_{\sigma,k}-a_{\sigma,k})=-\lambda_1\mathfrak{o}(\sigma,k),\\
&\lim_{j\to\infty}\frac{1}{j}\log\|S(j)({\xi}_{[(\sigma,k,l)]})\|_{\sup}=-\lambda_1\mathfrak{o}(\sigma,k)\quad\text{if}
\quad\mathfrak{o}(\sigma,k)>0.
\end{split}
\end{align}

{
For any $r\geq-\frac{m-2}{m}$, let $n=\lceil r\rceil$ and $a=n-r$. For any $f\in C^{k_r}(M)$, let
\begin{align}
\label{def:rsk}
\mathfrak{s}_{r}(f)&=\varphi_{f}-\sum_{\substack{[(\sigma,k,l)]\in \mathscr{TC}/\sim\\ \mathfrak{o}(\sigma,k)<r}}
\mathfrak{C}_{[(\sigma,k,l)]}(f)\xi_{[(\sigma,k,l)]}.
\end{align}

\begin{lemma}\label{lem:sr}
For any $f\in C^{k_r}(M)$, we have
$\mathfrak{s}_{r}(f)\in C^{n+\pa}(\sqcup_{\alpha \in \mathcal{A}}I_\alpha)$,
and the operator $\mathfrak{s}_{r}:C^{k_r}(M)\to C^{n+\pa}(\sqcup_{\alpha \in \mathcal{A}}I_\alpha)$ is bounded.
\end{lemma}
}

\begin{proof}
By Theorem~5.6 in \cite{Fr-Ki2},
\[
\widehat{\mathfrak{s}}_{r}(f):=\varphi_{f}-\sum_{\substack{[(\sigma,k,l)]\in \mathscr{TC}/\sim\\ \mathfrak{o}(\sigma,k)<r}}
\mathfrak{C}_{[(\sigma,k,l)]}(f)\widehat{\xi}_{[(\sigma,k,l)]}\in C^{n+\pa}(\sqcup_{\alpha \in \mathcal{A}}I_\alpha),
\]
and the operator $\widehat{\mathfrak{s}}_{r}:C^{k_r}(M)\to C^{n+\pa}(\sqcup_{\alpha \in \mathcal{A}}I_\alpha)$ is bounded. Moreover.
\[\mathfrak{s}_{r}(f)=\widehat{\mathfrak{s}}_{r}(f)+\sum_{\substack{[(\sigma,k,l)]\in \mathscr{TC}/\sim\\ \mathfrak{o}(\sigma,k)<r}}
\mathfrak{C}_{[(\sigma,k,l)]}(f)\big(\widehat{\xi}_{[(\sigma,k,l)]}-{\xi}_{[(\sigma,k,l)]}\big).\]
Since $\widehat{\xi}_{[(\sigma,k,l)]}-{\xi}_{[(\sigma,k,l)]}$ is a polynomial over any exchanged interval, this gives our claim.
\end{proof}

\begin{definition}\label{def:Ft}
Let $r\geq-\frac{m-2}{m}$. For any $\bar t\in\mathscr{TF}^*$ with $\mathfrak{o}(\bar{t})<r$, we denote by $\mathfrak{F}_{\bar t}:C^{k_r}(M)\to\C$ the operator given by $\mathfrak{F}_{\bar t}:=\mathfrak{f}_{\bar t}\circ \mathfrak{s}_{r}$. As $\mathfrak{s}_{r}:C^{k_r}(M)\to C^{n+\pa}(\sqcup_{\alpha \in \mathcal{A}}I_\alpha)$ with $n=\lceil r\rceil$, $a=\lceil r\rceil-r$, and $\bar t\in \mathscr{T}^*_{a,n}$ by \eqref{eq:maxex}, the operator {$\mathfrak{F}_{\bar t}$}
is well-defined and bounded.
\end{definition}
\begin{remark}
Note that the definition of $\mathfrak{F}_{\bar t}$ does not depend on the choice of $r$. Indeed, suppose that $\mathfrak{o}(\bar{t})<r_1<r_2$. Then for every $f\in C^{k_{r_2}}(M)$,
\[\mathfrak{s}_{r_1}(f)-\mathfrak{s}_{r_2}(f)=\sum_{\substack{[(\sigma,k,l)]\in \mathscr{TC}/\sim\\ r_1\leq\mathfrak{o}(\sigma,k)<r_2}}
\mathfrak{C}_{[(\sigma,k,l)]}(f){\xi_{[(\sigma,k,l)]}}.\]
%
%
In view of \eqref{eq:zerod}, it follows that $\mathfrak{f}_{\bar t}(\mathfrak{s}_{r_1}(f))=\mathfrak{f}_{\bar t}(\mathfrak{s}_{r_2}(f))$, which yields our claim.
\end{remark}
\begin{remark}
For any $\bar{t}\in\mathscr{TF}^*$, take any $r$ such that $\mathfrak{o}(\bar{t})<r<\mathfrak{o}(\bar{t})+\frac{1}{m}$.
By definition, $k_{r}\leq k_{\mathfrak{o}(\bar{t})}+1$. It follows that
 the functional $\mathfrak{F}_{\bar{t}}$ is defined on $C^{k_{\mathfrak{o}(\bar{t})+1}}(M)$. If $\mathfrak{o}(\bar{t})\notin \Z/m$, then the domain of $\mathfrak{F}_{\bar{t}}$ is enlarged to $C^{k_{\mathfrak{o}(\bar{t})}}(M)$.
\end{remark}

\subsection{Proofs of the main results}\label{sec:proofmain}
\begin{proof}[Proof of Theorem~\ref{thm:main1}]
Choose $r_0\in \R_{>0}$ which is the smallest element of $\{\mathfrak{o}(\sigma,k):k\geq 0,\sigma\in \mathrm{Sd}(\psi_\R)\cap M'\}
\cup\{\mathfrak{o}(\bar{t}):\bar{t}\in \mathscr{TF}\}$ larger than $r$. By assumption, $T$ satisfies the \ref{FDC}, $f\in C^{k_{r}}(M) =C^{k_{r_0}}(M)$, and
\begin{itemize}
\item $\mathfrak{d}^k_{\sigma,j}(f)=0$ for all $(\sigma,k,j)\in\mathscr{TD}$ with $\widehat{\mathfrak{o}}(\mathfrak{d}^k_{\sigma,j})<r_0$;
\item $\mathfrak{C}^k_{\sigma,l}(f)=0$ for all $(\sigma,k,l)\in\mathscr{TC}$ with ${\mathfrak{o}}(\mathfrak{C}^k_{\sigma,l})<r_0$;
\item $\mathfrak{F}_{\bar{t}}(f)=0$ for all $\bar{t}\in\mathscr{TF}$ with ${\mathfrak{o}}(\mathfrak{F}_{\bar{t}})<r_0$.
\end{itemize}
By Theorem~1.1 in \cite{Fr-Ki2},  $\varphi_f\in C^{n+\pa}(\sqcup_{\alpha \in \mathcal{A}}I_\alpha)$ with $n=\lceil r_0\rceil$ and $a=\lceil r_0\rceil-r_0$. Moreover, there exists $C_r>0$ such that $$\|\varphi_f\|_{C^{n+\pa}(\sqcup_{\alpha \in \mathcal{A}}I_\alpha)}\leq C_r\|f\|_{C^{k_r}(M)}$$ for all $f\in C^{k_r}(M)\cap \ker(\mathfrak{C}^k_{\sigma,l})$ for $(\sigma,k,l)\in\mathscr{TC}$ with  $\mathfrak{o}(\mathfrak{C}^k_{\sigma,l})<r$.

By assumption, in view of \eqref{def:rsk}, $\varphi_f=\mathfrak{s}_{r_0}(f)$. It follows that $\mathfrak{f}_{\bar t}(\varphi_f)=\mathfrak{f}_{\bar t}(\mathfrak{s}_{r_0}(f))=\mathfrak{F}_{\bar t}(f)=0$ for all $\bar t\in \mathscr{TF}$ with $\mathfrak{o}(\bar t)<r$. As $r<r_0=n-a$, in view of Theorem~\ref{thm:cohsolmain1}, there exists a solution $v\in C^r(I)$ to the cohomological equation $\varphi=v\circ T-v$ such that $v(0)=0$. Moreover, there exists $C'_r>0$ such that $\|v\|_{C^r(I)}\leq C'_r\|\varphi_f\|_{C^{n+\pa}(\sqcup_{\alpha \in \mathcal{A}}I_\alpha)}$. By Theorem~1.2 in \cite{Fr-Ki2}, there exists $u_{v,f}\in C^r(M'_e)$ satisfying  $Xu_{v,f}=f$ on $M'_e$. Moreover, there exists a constant $C''_r>0$ such that
$$\|u_{v,f}\|_{C^r(M'_e)}\leq C''_r(\|v\|_{C^{r}(I)}+\|f\|_{C^{{k}_{r}}(M)}).$$
It follows that
\[\|u_{v,f}\|_{C^r(M'_e)}\leq C''_r(1+C_rC'_r)\|f\|_{C^{{k}_{r}}(M)},\]
which completes the proof.
\end{proof}

\begin{proof}[Proof of Theorem~\ref{thm:main2}]
If $f\in C^{k_r}(M)$ and there exists $u\in C^r(M'_e)$ such that $Xu=f$ on $M'_e$, then by Theorem 1.3 in \cite{Fr-Ki2},
 $\mathfrak{d}^k_{\sigma,j}(f)=0$ for all $(\sigma,k,j)\in\mathscr{TD}$ with $\widehat{\mathfrak{o}}(\mathfrak{d}^k_{\sigma,j})<r$,
and $\mathfrak{C}^k_{\sigma,l}(f)=0$ for all $(\sigma,k,l)\in\mathscr{TC}$ with ${\mathfrak{o}}(\mathfrak{C}^k_{\sigma,l})<r$.
In view of Theorem~1.1 in \cite{Fr-Ki2},  $\varphi_f\in C^{n+\pa}(\sqcup_{\alpha \in \mathcal{A}}I_\alpha)$ with $n=\lceil r\rceil$ and $a=\lceil r\rceil-r$.
By \eqref{def:rsk}, it follows that  $\varphi_f=\mathfrak{s}_{r}(f)$. Hence $\mathfrak{F}_{\bar t}(f)=\mathfrak{f}_{\bar t}(\mathfrak{s}_{r}(f))=\mathfrak{f}_{\bar t}(\varphi_f)$ for all $\bar t\in \mathscr{TF}$ with $\mathfrak{o}(\bar t)<r$.

On the other hand, $\varphi_f=v\circ T-v$, where $v\in C^r(I)$ is the restriction of $u$ to $I$. By Theorem~\ref{thm;spdecomp}, this gives $\mathfrak{f}_{\bar t}(\varphi_f)=0$ for all $\bar t\in \mathscr{TF}$ with $\mathfrak{o}(\bar t)<r$. Therefore, $\mathfrak{F}_{\bar t}(f)=0$ for all $\bar t\in \mathscr{TF}$ with $\mathfrak{o}(\bar t)<r$.
\end{proof}

\begin{proof}[Proof of Theorem~\ref{thm:spect}]
In view of Lemma~\ref{lem:sr} and Theorem~\ref{thm;spdecomp}, for any $f\in C^{k_r}(M)$,
\begin{align*}
\varphi_{f}&=\mathfrak{s}_{r}(f)+\sum_{\substack{[(\sigma,k,l)]\in \mathscr{TC}/\sim\\ \mathfrak{o}(\sigma,k)<r}}
\mathfrak{C}_{[(\sigma,k,l)]}(f)\xi_{[(\sigma,k,l)]}\\
&=
\sum_{\substack{\bar{t}\in \mathscr{TF}^*\\ \mathfrak{o}(\bar{t})<r}}\mathfrak{f}_{\bar{t}}(\mathfrak{s}_{r}(f))h_{\bar t}+\mathfrak{r}_{a,n}(\mathfrak{s}_{r}(f))+\sum_{\substack{[(\sigma,k,l)]\in \mathscr{TC}/\sim\\ \mathfrak{o}(\sigma,k)<r}}
\mathfrak{C}_{[(\sigma,k,l)]}(f)\xi_{[(\sigma,k,l)]}\\
&=\sum_{\substack{\bar{t}\in \mathscr{TF}^*\\\mathfrak{o}(\bar{t})<r}}\mathfrak{F}_{\bar t}(f)h_{\bar{t}}+\sum_{\substack{[(\sigma,k,l)]\in\mathscr{TC}/\sim\\ \mathfrak{o}(\sigma,k)<r}}\mathfrak{C}_{[(\sigma,k,l)]}(f)\xi_{[(\sigma,k,l)]}+\mathfrak{r}_r(f)
\end{align*}
with $\mathfrak{r}_r:=\mathfrak{r}_{a,n}\circ\mathfrak{s}_{r}$, where $n=\lceil r\rceil$ and $a=\lceil r\rceil-r$. Note that \eqref{eq:expt}, \eqref{eq:expsi1}, and \eqref{eq:expsi2} follow directly from  \eqref{eq:grhi} and \eqref{eq:exxi}. Moreover, \eqref{eq:exrem} follows from \eqref{eq:skdl} and \eqref{eq:skdn}.
\end{proof}

\section*{Acknowledgements}
The authors would like to thank Giovanni Forni for his helpful comments and for assisting in completing the list of references. They are also grateful to the referees for their valuable comments and suggestions that helped improve the presentation of this work.

The authors acknowledge the Center of Excellence ``Dynamics, mathematical analysis and artificial intelligence'' at the Nicolaus Copernicus University in Toru\'n and  Centro di Ricerca Matematica Ennio De Giorgi - Scuola Normale Superiore, Pisa, and KTH Royal Institute of Technology for their hospitality during the visits.

The research was partially supported by  the Narodowe Centrum Nauki Grant 2022/45/B/ST1/00179.
M.K.\ was partially supported by UniCredit Bank R\&D group through the `Dynamics and Information Theory Institute' at the Scuola Normale Superiore, {foundations for the Royal Swedish Academy of Sciences and Carl Trygger's Foundation for Scientific Research.}

\end{document}